\documentclass[11 pt]{article}
\usepackage{amsmath}
\usepackage{amssymb}
\usepackage{amsthm}
\usepackage{amsfonts}
\usepackage[dvips]{graphicx}
\usepackage{pinlabel}

\theoremstyle{plain}
\newtheorem{theorem}{Theorem}[section]
\newtheorem{proposition}[theorem]{Proposition}
\newtheorem{lemma}[theorem]{Lemma}
\newtheorem{mainlemma}[theorem]{Main Lemma}
\newtheorem{corollary}[theorem]{Corollary}

\newtheorem*{claim*}{Claim} 
\newtheorem*{case*}{Case}

\theoremstyle{definition}
\newtheorem{definition}[theorem]{Definition}
\newtheorem{example}[theorem]{Example}

\theoremstyle{remark}
\newtheorem*{remark}{Remark}

\newcommand\Z{\mathbb{Z}}
\newcommand\Q{\mathbb{Q}}

\newcommand\co{\colon}

\newcommand\abs[1]{\lvert #1\rvert}
\newcommand\abso[1]{\lvert #1\rvert_O}

\DeclareMathOperator{\Aut}{Aut}
\DeclareMathOperator{\Out}{Out}
\DeclareMathOperator{\SL}{SL}
\DeclareMathOperator{\GL}{GL}
\DeclareMathOperator{\st}{st}

\DeclareMathOperator{\dom}{dom}

\DeclareMathOperator{\supp}{supp}

\newcommand{\genby}[1]{\langle #1\rangle}
\newcommand{\whset}[1]{\Omega_{[#1]}}
\newcommand{\whsetsr}[2]{\Omega_{[#1],#2}}

\newcommand{\pcount}[3]{\langle #1,#2\rangle_{#3}}

\title{
Full-featured peak reduction in right-angled Artin groups
}
\author{Matthew B. Day}
\date{November 19, 2013}

\begin{document}
\maketitle

\begin{abstract}
We prove a new version of the classical peak-reduction theorem for automorphisms of free groups in the setting of right-angled Artin groups.
We use this peak-reduction theorem to prove two important corollaries about the action of the automorphism group of a right-angled Artin group $A_\Gamma$ on the set of $k$-tuples of conjugacy classes from $A_\Gamma$: orbit membership is decidable, and stabilizers are finitely presentable.
Further, we explain procedures for checking orbit membership and building presentations of stabilizers.
This improves on a previous result of the author's.
We overcome a technical difficulty from the previous work by considering infinite generating sets for the automorphism groups.
The method also involves a variation on the Hermite normal form for matrices.
\end{abstract}

\section{Introduction}
\subsection{Background}
Let $F$ denote a finite-rank free group with automorphism group $\Aut(F)$.
\emph{Peak reduction} is a technique in the study of $\Aut(F)$ that is a key ingredient in the solution of several important problems.
J.H.C.\ Whitehead invented the technique in the 1930's in~\cite{Whitehead} to provide an algorithm that takes in two conjugacy classes (or more generally, $k$-tuples of conjugacy classes) from $F$ and determines whether there is an automorphism in $\Aut(F)$ that carries one to the other.
In a series of papers in the 1970's~\cite{McCool1,McCool2,McCool3}, McCool used peak-reduction methods to reprove Nielsen's result that $\Aut(F)$ is finitely presented, and to prove that the stabilizer in $\Aut(F_n)$ of a tuple of conjugacy classes in $F$ is also finitely presented.

Given a finite simplicial graph $\Gamma$, the \emph{right-angled Artin group} $A_\Gamma$ is the  group given by the finite presentation whose generators are the vertices of $\Gamma$, and whose only relations are that two generators commute if and only if they are adjacent as vertices in $\Gamma$.
These groups are also called ``partially commutative groups'' and ``graph groups''.
Free groups are extreme examples of right-angled Artin groups, but this class of groups contains also contains free abelian groups and many other groups.

The goal of this paper is to generalize the peak-reduction method to right-angled Artin groups.
In a previous paper~\cite{Day1}, the author found a weak generalization of peak reduction and used it to prove that $\Aut(A_\Gamma)$ is finitely presented for every $\Gamma$.
That paper and its sequel~\cite{Day2} used peak reduction to prove that the stabilizers of certain specific tuples of conjugacy classes in $A_\Gamma$ are finitely generated, and to investigate an analogue of mapping class groups of surfaces inside $\Aut(A_\Gamma)$.
There is also a version of peak reduction for right-angled Artin groups from recent work of Charney--Stambaugh--Vogtmann~\cite{CharneyStambaughVogtmann}; 
they use peak reduction to study an outer space for $\Aut(A_\Gamma)$ (a contractible cell complex that $\Aut(A_\Gamma)$ acts on geometrically). 
We relate their theorem to our results after we state them below.
However, the other applications of peak reduction do not seem to follow directly from either of the peak-reduction theorems just mentioned.
The present paper proves a strong generalization of peak reduction to right-angled Artin groups and uses it to prove two important corollaries of peak reduction.

\subsection{Results}
First we state our two corollaries of peak reduction.
Let $A_\Gamma$ be a right-angled Artin group.
\begin{theorem}\label{th:raagcheckorbit}
There is an algorithm that takes in two tuples $U$ and $V$ of conjugacy classes from $A_\Gamma$ and either produces an automorphism $\alpha\in \Aut(A_\Gamma)$ with $\alpha\cdot U=V$ or determines that no such automorphism exists.
\end{theorem}

\begin{theorem}\label{th:raagstabpres}
There is an algorithm that takes in a tuple $W$ of conjugacy classes from $A_\Gamma$ and produces a finite presentation for its stabilizer $\Aut(A_\Gamma)_W$.
\end{theorem}

These results follow from our result on peak reduction, which will take a few definitions to state.
Let $X$ denote the vertex set of $\Gamma$.
The length $\abs{u}$ of an element $u\in A_\Gamma$ is the usual length in terms of the generating set $X$.
The length $\abs{w}$ of a conjugacy class $w$ is the minimum length of a representative, and the length $\abs{W}$ of a tuple $W$ of conjugacy classes is the sum of the lengths of its entries.

For $a\in X$, the \emph{star} $\st(a)$ is the subset of $X$ consisting of $a$ and all vertices adjacent to $a$.
For $a$ and $b$ in $X$, $a$ and $b$ are in the same \emph{adjacent domination equivalence class} if $\st(a)=\st(b)$.
We denote the adjacent domination equivalence class of $a$ by $[a]$; if $b\in[a]$ with $a\neq b$, then $a$ is necessarily adjacent to $b$.
It is entirely possible that $[a]=\{a\}$.
We define the set of \emph{generalized Whitehead automorphisms with multiplier set $[a]$}, denoted $\whset{a}$, to be the subgroup of automorphisms in $\Aut(A_\Gamma)$ such that
$\alpha\in\whset{a}$ if and only if
\begin{itemize}
\item for each $b\in X\setminus [a]$, there are $u,v\in \genby{[a]}$ with $\alpha(b)=ubv$ and
\item for each $b\in [a]$, we have $\alpha(b)\in\genby{[a]}$.
\end{itemize}
In general each $\whset{a}$ is an infinite group.
We define the \emph{permutation automorphisms} $P$ to be the finite subgroup of $\Aut(A_\Gamma)$ with $\alpha$ in $P$ if and only if $\alpha$ restricts to a permutation of $X\cup X^{-1}$.
We define the set of \emph{generalized Whitehead automorphisms} $\Omega$ to be the union of $P$ with all the $\whset{a}$ as $a$ varies over the vertices of $\Gamma$.
It follows from a result of Laurence~\cite{Laurence} (see Theorem~\ref{th:Laurence} below) that $\Omega$ is a generating set for $\Aut(A_\Gamma)$.
\begin{theorem}[Peak reduction]\label{th:fullfeaturedpeakreduction}
Suppose $\alpha\in\Aut(A_\Gamma)$ and $W$ is a tuple of conjugacy classes in $A_\Gamma$.
Then there is a factorization
\[\alpha=\beta_m\beta_{m-1}\dotsm\beta_2\beta_1\]
with $\beta_1,\dotsc,\beta_m$ in $\Omega$, such that the sequence of intermediate lengths
\[k\mapsto \abs{\beta_k\dotsm \beta_1\cdot W}\quad\text{for $k=0,\dotsc,m$}\]
strictly decreases from $k=0$ to $k=k_1$ for some $k_1$ with $0\leq k_1\leq m$,
stays constant from $k=k_1$ to $k=k_2$ for some $k_2$ with $k_1\leq k_2\leq m$, 
and strictly increases from $k=k_2$ to $k=m$.
Further, there is an algorithm to find such a factorization.
\end{theorem}

A major difference between this formulation and the one for free groups is that the set $\Omega$ is finite in that setting, but is usually infinite in this setting.
Since the classic applications rely on this set being finite, we need some additional results to get the applications to work.
\begin{proposition}\label{pr:insamewhorbit}
There is an algorithm that takes in two tuples $U$ and $V$ of conjugacy classes from $A_\Gamma$, and a vertex $a$ in $\Gamma$, and produces an automorphism $\alpha$ in $\whset{a}$ with $\alpha\cdot U=V$ or determines that no such automorphism exists.
In particular, it is possible to determine whether a tuple $U$ can be shortened by an automorphism from $\whset{a}$.
\end{proposition}
\begin{proposition}\label{pr:whstab}
There is an algorithm that takes in a tuple $U$ of conjugacy classes from $A_\Gamma$ and a vertex $a$ in $\Gamma$, and produces a finite presentation for the stabilizer $(\whset{a})_U$.
\end{proposition}
It turns out that the groups $\whset{a}$ embed in integer general linear groups, and the key to these propositions is a modified version of the Hermite normal form for integer matrices. 

The following example motivates the use of the infinite generating sets.
\begin{example}\label{ex:nofiniteprgenset}
This is Example~4.1 from Day~\cite{Day1}, which shows that no finite generating set will work for a peak-reduction theorem formulated for the entirety of $\Aut(A_\Gamma)$.
In this example, $\Gamma$ is the four-vertex path graph with vertices $a$, $b$, $c$ and $d$ in that order.
In~\cite[Proposition~C]{Day1}, we show that for $k\in\Z$, the conjugacy class $w=ad^k$ in this $A_\Gamma$ is fixed by an automorphism $\alpha$ with $\alpha(a)=ac^k$ and $\alpha(d)=c^{-1}d$ and with $\alpha$ fixing $b$ and $c$.
Further, we show that there is no peak-reduced factorization of $\alpha$ with respect to $w$ by automorphisms that are simpler in a specific sense.
This contradicts the existence of a certain formulation of peak-reduction theorem, because such a theorem could be used to produce such factorizations with automorphisms taken from a fixed finite set; however, such a set cannot work for all choices of $k$.
But since the automorphism $\alpha$ is in our set $\whset{c}$ for any $k$, this example does not contradict Theorem~\ref{th:fullfeaturedpeakreduction}.
\end{example}

The results in the present paper are somewhat similar to certain other results.
However, Theorems~\ref{th:raagcheckorbit} and~\ref{th:raagstabpres} do not appear to be a direct consequence of any results in the literature.
First we compare Theorem~\ref{th:fullfeaturedpeakreduction} to Theorem~B of Day~\cite{Day1}.
Although part~(3) of Theorem~B is a peak-reduction theorem, it applies only to a proper subgroup of $\Aut(A_\Gamma)$ (in general) and there does not appear to be a straightforward way to apply that theorem to characterize orbits or stabilizers under the entirety of $\Aut(A_\Gamma)$.

Next we note that there is a peak-reduction theorem for right-angled Artin groups in a recent preprint of Charney--Stambaugh--Vogtmann~\cite{CharneyStambaughVogtmann}.
Theorem~5.19 of that paper proves that $\Aut(A_\Gamma)$ has peak reduction using a finite generating set, but only with respect to a specific kind of tuple of conjugacy classes: tuples containing all conjugacy classes of length one or two, and the images of such tuples under automorphisms.
Their proof is elegant, but the methods do not seem to apply to other kinds of tuples of conjguacy classes (in particular, Example~\ref{ex:nofiniteprgenset} is still a problem).

We also note another special case where results like these are previously known.
Collins--Zieschang have a series of papers,~\cite{CollinsZieschang84a}, \cite{CollinsZieschang84b}, \cite{CollinsZieschang84c} and \cite{CollinsZieschang87}, on the Whitehead method for free products of groups.
If $\Gamma$ is a disjoint union of complete graphs, then $A_\Gamma$ is a free product of free abelian groups, and results of Collins--Zieschang similar to Theorems~\ref{th:raagcheckorbit} and~\ref{th:raagstabpres} apply.
Their methods do not extend to $A_\Gamma$ for general $\Gamma$ since $A_\Gamma$ is freely indecomposable if $\Gamma$ is connected.

\begin{remark}
Another potential application of results of this paper is to the algebraic geometry of groups.
Casals-Ruiz and Kazachkov define an algorithm in~\cite{CasalsRuizKazachkov} for parametrizing solutions to systems of equations in right-angled Artin groups, and our Theorem~\ref{th:raagstabpres} is relevant to that algorithm.
Specifically, a system of equations over a right-angled Artin group $A_\Gamma$ corresponds to tuple $W$ of conjugacy classes (those which are set equal to $1$) in $A_\Gamma * F_n$, a free product of $A_\Gamma$ with a finite-rank free group.
This free product is another right-angled Artin group $A_\Delta$.
Theorem~\ref{th:raagstabpres} gives us a finite presentation, in an effective way, for the stabilizer $(A_\Delta)_W$, and this stabilizer maps to the automorphism group of the coordinate group of the system of equations. 
Casals-Ruiz and Kazachkov have informed the author that the finite presentability of these groups could be used to make improvements to their algorithm.
This connection is not pursued in the present paper.
\end{remark}

\subsection{Organization of the paper}
We postpone the proofs of the more technical results to later sections in the paper.
Section~\ref{se:preliminaries} contains preliminary facts about right-angled Artin groups.
We reduce Propositions~\ref{pr:insamewhorbit} and~\ref{pr:whstab} to propositions about linear groups in Section~\ref{se:linearreduction}, and we prove our two main application theorems in Section~\ref{se:applications}.
Section~\ref{se:linearproblems} proves the propositions about linear groups and we finally prove the peak reduction theorem in Section~\ref{se:peakreduction}.

\section{Preliminaries}\label{se:preliminaries}
\subsection{Combinatorial group theory of right-angled Artin groups}\label{ss:combinatorial}
The survey by Charney~\cite{Charney} is a good general reference for right-angled Artin groups.
We recall a few facts here that are important for understanding the paper.
As mentioned above, the \emph{length} of an element $u$ of $A_\Gamma$ is always the minimal length of an expression for $u$ as a product of generators from $\Gamma$ and their inverses.
Servatius~\cite{Servatius} points out that a word in the generators representing $u$ is of minimal length if and only if it is \emph{graphically reduced}, meaning that there is no subword $xvx^{-1}$ where $x$ is a generator or inverse generator and $v$ is a word in those generator that commute with $x$.
Further, Servatius shows that it is possible to get between any two minimal-length representatives for a word by a sequence of commutation moves.
Normal forms found by VanWyk~\cite{VanWyk} and by Hermiller--Meier~\cite{HermillerMeier} are a convenient way to check whether two words represent the same element.
Algorithms for checking conjugacy have been described by Liu, Wrathall and Zeger~\cite{LiuWrathallZeger}, and by Wrathall~\cite{Wrathall}.

One important result is the centralizer theorem of Servatius, from~\cite{Servatius}.
It states that for any $u\in A_\Gamma$, we can find an element $v$ with $u'=vuv^{-1}$ cyclically reduced (minimal length in its conjugacy class), and for $u'$ cyclically reduced, its centralizer is exactly what one would guess.
Specifically, $u'$ can be written as a graphically reduced product  $t_1^{k_1}\dotsm t_m^{k_m}$ where each $t_i$ is not a proper power and every vertex of $\Gamma$ used in writing $t_i$ commutes with every vertex of $\Gamma$ used in writing $t_j$, for each $i\neq j$.
Then the centralizer of $u'$ is generated by the $\{t_i\}_i$ together with each vertex in $\Gamma$ that commutes with all the vertices appearing in $u'$.
The centralizer of the original $u$ is of course the conjugate of the centralizer of  $u'$ by the same conjugating element $v$.

\subsection{Generating Automorphisms}
If $a$ and $b$ are vertices in $\Gamma$ and every vertex that is adjacent to $b$ commutes with $a$, we say that $a$ \emph{dominates} $b$.
This can happen whether or not  $a$ is adjacent to $b$.
Domination is important because it determines whether certain maps defined on generators actually extend to automorphisms.
The following definitions are originally from Servatius~\cite{Servatius}:
\begin{itemize}
\item If $a$ and $b$ are distinct vertices in $\Gamma$ and $a$ dominates $b$, then there is an automorphism in $\Aut(A_\Gamma)$ sending $b$ to $ba$ and fixing all other vertices in $\Gamma$.
This is a \emph{dominated transvection}.
If $a$ is not adjacent to $b$, then there is a different automorphism sending $b$ to $ab$ and fixing all other generators; this is also a dominated transvection.
When we need to distinguish between these, we refer to \emph{right} and \emph{left} dominated transvections.
\item If $a$ is a vertex in $\Gamma$ and $Y$ is a connected component of $\Gamma\setminus\st(a)$, then there is an automorphism in $\Aut(A_\Gamma)$ sending each $c$ in $Y$ to $aca^{-1}$ and fixing all other vertices in $\Gamma$.
This is a \emph{partial conjugation}.
\item If $\pi$ is a symmetry of the graph $\Gamma$, then there is an automorphism in $\Aut(A_\Gamma)$ sending each generator $x$ to its image $\pi(x)$.
This is called a \emph{graphic automorphism}.
\item For each vertex $a$ in $\Gamma$, there is an automorphism sending $a$ to $a^{-1}$ and fixing all other generators.
This is the \emph{inversion} in $a$.
\end{itemize}
If $a$ is adjacent to $b$ then the right and left dominated transvections are the same.
Dominated transvections generalize Nielsen moves and also those elementary matrices with a single nonzero off-diagonal entry of one.
It is well known that transvections and inversions generate automorphism groups of free groups and general linear groups over the integers, but for general right-angled Artin groups, there are nontrivial partial conjugations and graphic automorphisms that cannot be expressed as products of inversions and transvections.
However, these four kinds of automorphisms generate:
\begin{theorem}[Laurence~\cite{Laurence}]\label{th:Laurence}
For any $\Gamma$, the group $\Aut(A_\Gamma)$ is generated by the finite set consisting of all dominated transvections, partial conjugations, graphic automorphisms, and inversions.
\end{theorem}

\subsection{Labeled directed graphs}

A \emph{labeled directed graph} is a directed multigraph, with self-loops allowed, so that each directed edge carries a label from a pre-specified label set.
In this paper, the label sets are always subsets of groups, and we use the convention that a directed edge from a vertex $v_1$ to a vertex $v_2$ with label $g$ is also a directed edge from $v_2$ to $v_1$ with label $g^{-1}$.
Since all directed edges are considered directed in both directions (with different labels) this means that edge paths are the same as in the underlying undirected graph.
In particular, the connected components of a labeled directed graph are the same as the connected components of the underlying undirected graph.

Again we emphasize that in this paper, edge labels are always elements of a pre-specified group.
By the \emph{composition of edge labels} of an edge path $p$ in a labeled directed graph, we mean the following: for each edge $e$ on $p$, we record the label of $e$ in $S$ if the orientations of $e$ and $p$ agree, and if the orientations disagree, we record the inverse of the label of $e$; the composition of edge labels is then the composition of these labels and inverses that we recorded, in the order that $p$ traverses them.
Or since we consider reversed edges to be labeled with the inverse group element, we can take the composition of edge labels of $p$ simply to be the composition of the labels on the edges in $p$, interpreted with orientation agreeing with $p$.

One note: we compose automorphisms like functions, but we compose edge labels on paths in the opposite order, like the usual convention for fundamental groups.
This means that if $\Delta$ is a directed graph with labels in an automorphism group, then the composition of edge labels along a path labeled with $\alpha_1$--$\alpha_2$--$\dotsm$--$\alpha_k$ in $\Delta$ is $\alpha_k\dotsm\alpha_2\alpha_1$.

Let $S$ be a set of elements of an arbitrary group $G$ and let $H$ be a subgroup of $G$.
The \emph{Schreier graph} of $H$ in $G$ with respect to $S$ is the labeled, directed multigraph whose vertex set is the set of left cosets of $H$ in $G$ and with an edge labeled by $c$ in $S$ from coset $aH$ to $bH$ if and only if $caH=bH$.
Note that since we are not supposing that $S\cup H$ generates $G$, it will not generally be the case that these Schreier graphs are connected.

\begin{lemma}\label{le:travelSchreier}
Suppose $S$ is a set of elements of $G$ and $\Delta$ is the Schreier graph of $H$ in $G$ with respect to $S$.
Suppose $p$ is an edge path in  $\Delta$, starting at a vertex $aH$ and ending at a vertex $bH$.
Let $c$ be the composition of edge labels along $p$.
Then $caH=bH$.
\end{lemma}
The proof is an induction argument using the definitions, and is omitted.

In our arguments we need to be able to promote a presentation from a finite-index subgroup of a group to the entire group.
Although it is common knowledge that this is possible, we could not find a reference on how to do it.
So for completeness, we provide an argument.
\begin{lemma}\label{le:presfromfinidx}
Suppose $H$ is a finite-index subgroup of a group $G$ and we are given a finite presentation $H=\genby{S_H|R_H}$ and a finite set $S_G\subset G$ such that the Schreier graph $\Delta$ of $H$ in $G$ with respect to $S_H\cup S_G$ is connected.
Suppose we are also given an explicit description of $\Delta$ as a labeled directed multigraph.
Then we can write down a finite presentation for $G$.
\end{lemma}

\begin{proof}
Let $p_1,\dotsc,p_k$ be generators for the fundamental group of $\Delta$ based at $H$.
Since $\Delta$ is a finite graph, this is a finite set (generators can be found by selecting a maximal subtree of $\Delta$).
Let $u_i$ be the composition of edge labels along $p_i$ for each $i$; each $u_i$ is a word in $S_H\cup S_G$.
By Lemma~\ref{le:travelSchreier}, we know that $u_i\in H$.
Let $w_i$ be a word in $S_H$ representing the same element as $u_i$.
Let $R_G$ be the set of relations $\{u_1w_1^{-1},\dotsc,u_kw_k^{-1}\}$.
Let $G'=\genby{S_H\cup S_G|R_H\cup R_G}$.

Let $\phi\co G'\to G$ send each generator to the element of the same name.
Since $\phi$ sends relators to the identity, it is a well defined group homomorphism.
We claim that it is an isomorphism.
The group $G'$ acts on $\Delta$ through $\phi$; since $S_G\subset G$ is in the image of $\phi$, this action is transitive.
Let $K$ be the stabilizer of the vertex $H$ of $\Delta$ in $G'$.
We note that $\Delta$ is then isomorphic to the Schreier graph of $K$ in $G'$, and therefore the fundamental group of $\Delta$ based at $H$ surjects to $K$ by reading off edge labels.
In particular, $K$ is generated by the elements $u_1,\dotsc,u_k$.
Since $G'$ contains relations $R_G$ turning the generators $u_1,\dotsc,u_k$ into words in $S_H$, it follows that $K$ is the subgroup of $G'$ generated by $S_H$.
The map $\phi$ restricts to a map $K\to G$; since $K$ fixes the vertex $H$ in $\Delta$, the image of this map is in $H$.
We have a map $H\to G'$ sending each generator the generator of the same name; since $H$ fixes the vertex $H$ in $\Delta$, the image of this map lies in $K$.
So we have natural maps $H\to K$ and $K\to H$; since these are defined by sending generators to generators of the same name, they are inverses of each other.
This implies that $\phi$ is injective: anything in the kernel acts trivially on $H$ and is therefore in $K$, but $G'\to G$ restricts to an isomorphism on $K$.

Now we claim that $\phi$ is surjective.
If not, there is some $g\in G$ not in $\phi(G')$.
But $G'$ acts transitively on $\Delta$, so there is $g'\in G'$ with $g'\cdot H=gH$.
Then $\phi(g')g^{-1}$ is not in $\phi(G')$, however, it fixes $H$ and is therefore in $H$.
But $\phi$ maps $K$ maps surjectively to $H$, a contradiction.
\end{proof}

\section{Reduction to linear problems}\label{se:linearreduction}
\subsection{Structure of generalized Whitehead automorphisms}\label{ss:genwharelinear}
In this section we fix a vertex $a$ and consider the set $\whset{a}$ of generalized Whitehead automorphisms with multiplying set $[a]$.
As defined in the introduction, this is the subgroup of $\Aut(A_\Gamma)$ consisting of automorphisms that multiply elements not in $[a]$ on the right and left by elements of $[a]$, and send elements of $[a]$ to products of elements of $[a]$.
Let $\dom(a)$ denote the set of vertices that $a$ dominates, including $a$ itself.

We define $Z_{[a]}$ to be the free abelian group generated by the following symbols:
\begin{itemize}
\item for each vertex $b$ in $\st(a)\cap\dom(a)$, a generator $r_b$,
\item for each vertex $b$ in $\dom(a)\setminus\st(a)$, two generators $r_b$ and $l_b$, and
\item for each connected component $Y$ of $\Gamma\setminus\st(a)$ with at least two vertices, a generator $r_Y$.
\end{itemize}

We define a map $\eta\co\whset{a}\to\Aut(Z_{[a]})$ as follows, given $\alpha\in\whset{a}$:
\begin{itemize}
\item For $b$ in $\st(a)\cap\dom(a)$ and $c$ in $[a]$, the coefficient of $r_c$ in $\eta(\alpha)(r_b)$ is the sum exponent of $c$ in $\alpha(b)$, if $b\notin [a]$ then the coefficient of $r_b$ in $\eta(\alpha)(r_b)$ is one, and the coefficients of all other generators in $\eta(\alpha)(r_b)$ are zero.
\item For $b$ in $\dom(a)\setminus\st(a)$ and $c$ in $[a]$, the coefficient of $r_c$ in $\eta(\alpha)(r_b)$ is the sum exponent of $c$ in $v$ and the coefficient of $r_c$ in $\eta(\alpha)(l_b)$ is the sum exponent of $c$ in $u$, where $\alpha(b)=ubv$; the coefficient of $r_b$ in $\eta(\alpha)(r_b)$ is one, the coefficient of $l_b$ in $\eta(\alpha)(l_b)$ is one, and coefficients of all other basis elements in $\eta(\alpha)(r_b)$ and $\eta(\alpha)(l_b)$ are zero.
\item For $Y$ a connected component of $\Gamma\setminus\st(a)$ with at least two vertices and $c$ in $[a]$, the coefficient of $r_c$ in $\eta(\alpha)(r_Y)$ is the sum exponent of $c$ in $u$ where $\alpha(x)=u^{-1}xu$ for some $x$ in $Y$; the coefficient of $r_Y$ in $\eta(\alpha)(r_Y)$ is one and the coefficients of all other basis elements in $\eta(\alpha)(r_Y)$ are zero.
\end{itemize}

Our next goal is to show that $\eta$ is well defined and describe its image.
Let $n=\abs{[a]}$ and let $k=\dim(Z_{[a]})-n$.
To describe this image precisely, we pick an isomorphism between $Z_{[a]}$ and $\Z^{n+k}$ to identify $\Aut(Z_{[a]})$ with a matrix group.
Let $[a]=\{a_1,\dotsc,a_n\}$.
We map the basis elements $r_{a_1},\cdots, r_{a_n}$ to the first $n$ basis elements of $\Z^{n+k}$, and we map the remaining basis elements of $Z_{[a]}$ to the remaining basis elements of $\Z^{n+k}$ arbitrarily.

Throughout the paper, we use $\GL(n,R)$ to denote the general linear group of invertible $n\times n$ matrices with entries in a ring $R$, and we use $M_{n,k}(A)$ to denote the abelian group of $n\times k$ matrices with entries in an abelian group $A$.
\begin{lemma}\label{le:etainjective}
The map $\eta$ is a well defined injective homomorphism.

Further, we can describe its image.
Under the identification of $Z_{[a]}$ with $\Z^{n+k}$ above, the image of $\eta$ is 
 the set of block-upper-triangular matrices of the form
\[
\left(
\begin{array}{cc}
A & B \\ O & I
\end{array}\right),
\]
where $A$ is in $\GL(n,\Z)$,
$B$ is in $M_{n,k}(\Z)$,
$O$ is the zero matrix and $I$ is the $k\times k$ identity matrix.
Here the matrix $A$ records the coefficients of $r_b$ for $b\in [a]$ and the matrix $B$ records the remaining coefficients.
\end{lemma}

\begin{proof}
Since $\alpha\in\whset{a}$ sends each element $b$ of $[a]$ to some $\alpha(b)\in\genby{[a]}$, and the elements of $[a]$ commute, the image of $\eta(\alpha)(r_b)$ is well defined for each $b\in[a]$.
For each $b\in\dom(a)\cap\st(a)\setminus[a]$, we know that $\alpha(b)$ is $bu$ for some $u\in\genby{[a]}$, since we can commute elements of $[a]$ to the right side of $b$.
Since elements of $[a]$ commute, again $\eta(\alpha)(r_b)$ is well defined.
For each $b\in\dom(a)\setminus\st(a)$, we know $\alpha(b)=ubv$ for some $u,v\in\genby{[a]}$.
These elements $u$ and $v$ are well defined because we cannot commute elements of $[a]$ across this $b$.
Again the coefficients are well defined because $\genby{[a]}$ is abelian.

Now consider $Y$ a connected component of $\Gamma\setminus\st(a)$ with at least two vertices.
For any $b$ in $Y$ there is some $c$ in $Y$ such that $b$ commutes with $c$.
We know $\alpha(b)=u_1bv_1$ and $\alpha(c)=u_2cv_2$, with $u_i,v_i\in \genby{[a]}$.
It must be the case that $\alpha(b)$ commutes with $\alpha(c)$.
We consider the centralizer of $\alpha(b)$, as indicated by the Servatius centralizer theorem (see Section~\ref{ss:combinatorial}).
If $u_1\neq v_1^{-1}$, then the cyclically reduced form of $\alpha(b)$ contains $b$ and elements of $[a]$.
This means that there is an element of $\genby{[a]}$ that conjugates $\alpha(c)$ into $\genby{\st(b)\cap\st(a)}$.
However, this is impossible---any such conjugate of $\alpha(c)$ would have $c$ in it, which is not in $\st(a)$.
So this implies that $u_1=v_1^{-1}$, and of course a parallel argument implies that $u_2=v_2^{-1}$.
Further, the centralizer theorem then  implies that $u_1=u_2$.
Repeating this argument on all adjacent pairs in $Y$, we see that there is a fixed $u\in \genby{[a]}$, depending only on $\alpha$, such that $\alpha(b)=u^{-1}bu$ for any $b$ in $Y$.
Then since $\genby{[a]}$ is abelian, the image $\eta(\alpha)(r_Y)$ is well defined.

It is straightforward computation to check that $\eta$ is a homomorphism.

To see that $\eta$ is injective and has the specified image, we construct an inverse map $\theta$.
Really there is only one way to do this---given a matrix $M$ in the specified image, we read off the first $n$ entries in each column and multiply together elements of $[a]$ with these exponents to get $n+k$ elements of $\genby{[a]}$.
For each basis element, let $u_b$ be the element determined by the column for $r_b$, let $v_b$ be determined by the column for $l_b$, and let $u_Y$ be determined by the column for $r_Y$.
We construct an automorphism $\theta(M)$ that sends $b\in[b]$ to $u_b$; $\theta(M)$ sends $b\in\dom(a)\cap\st(a)\setminus[a]$ to $bu_b$; $\theta(M)$ sends $b\in\dom(a)\setminus\st(a)$ to $v_bbu_b$; and $\theta(M)$ sends each $c$ in a connected component $Y$ with at least two vertices to $u_Y^{-1}cu_Y$.
In fact this completely specifies $\theta(M)$ on generators; $\theta(M)$ extends to an endomorphism of $A_\Gamma$ because the images of generators still satisfy the defining relations for $A_\Gamma$, and $\theta(M)$ specifies an automorphism of $A_\Gamma$ because $\theta(M^{-1})$ is its inverse.
Since these maps are clearly inverses of each other, this finishes the proof.
\end{proof}

Let $m,n,k$ be integers with $k\geq 0$ and $n,m\geq 1$.
In light of Lemma~\ref{le:etainjective}, we define define the group $G_1$ to be $\GL(n,\Z)\ltimes M_{n,k}(\Z)$ and we state versions of Propositions~\ref{pr:insamewhorbit} and~\ref{pr:whstab} in this setting.
The reason for the subscript $1$ will be clear later; we will consider other versions of this group where the subscript denotes a common denominator for certain rational matrix entries.
\begin{proposition}\label{pr:matrixorbitalgorithm}
There is an algorithm that takes in matrices $A$ and $B$ in $M_{n+k,m}(\Z)$, and returns a matrix $D$ in $G_1$ with $DA=B$, or determines that there is no such matrix.
\end{proposition}
\begin{proposition}\label{pr:matrixstabpres}
Suppose $A$ is in $M_{n+k,m}(\Z)$.
Then there is an algorithm that produces a finite presentation for the stabilizer of $A$ in $G_1$.
\end{proposition}
The proofs of these propositions are postponed to Section~\ref{se:linearproblems}.

\subsection{Algorithms for groups of generalized Whitehead automorphisms}
In this section we prove Propositions~\ref{pr:insamewhorbit} and~\ref{pr:whstab} modulo Propositions~\ref{pr:matrixorbitalgorithm} and~\ref{pr:matrixstabpres}, which we prove later.
We again fix $a\in \Gamma$.
The following definition is not only central to this section but is also important for much of the rest of the paper.
\begin{definition}\label{de:syllable}
A \emph{syllable} in $A_\Gamma$ with respect to $a$ is a graphically reduced product of the form $c u d\in  A_\Gamma$, where $u$ is an element of $\genby{\st(a)}$, and $c$ and $d$ are elements of $(\Gamma\setminus\st(a))^{\pm1}$ or $c=d=1$.
It is a \emph{linear syllable} if $c\neq 1$ and $d\neq 1$, and a \emph{cyclic syllable} if $c=d=1$.
If $cud$ is a linear syllable, then $c$ and $d$ are the \emph{endpoints}.

If $w$ is a conjugacy class in $A_\Gamma$, a \emph{decomposition of $w$ into syllables with respect to $a$} is either of the following:
\begin{itemize}
\item if the conjugacy class $w$ has a representative element $u$ in $\genby{\st(a)}$, then $u$ itself is a cyclic syllable, and the singleton $(u)$ is a decomposition of $w$ into syllables;
\item if $w$ does not have a representative in $\genby{\st(a)}$, then a decomposition of $w$ into syllables is a $k$-tuple of linear syllables for some $k$:
\[(c_1u_1c_2, c_2u_2c_3,\dotsc,c_ku_kc_1),\]
such that the product
\[c_1u_1c_2u_2c_3u_3\dotsm c_ku_k\in A_\Gamma\]
is a graphically reduced and cyclically reduced product representing the class $w$.
\end{itemize}
The products $u$ or $c_1u_1\dotsm c_ku_k$ above are the \emph{representatives associated with the decomposition}.

If $W=(w_1,\dotsc,w_k)$ is a $k$-tuple of conjugacy classes in $A_\Gamma$, a \emph{decomposition of $W$ into syllables with respect to $a$} is the concatenation, into a single tuple, of some choice of decompositions of the $w_i$ into syllables.
Given a decomposition of $W$ into syllables, the \emph{representative associated to the decomposition} is the tuple of elements of $A_\Gamma$ consisting of the representatives associated with the decompositions of the $w_i$.
\end{definition}

We give some specific examples of syllable decompositions in Example~\ref{ex:syllablealgorithm} below.
To clarify the definition, we note a couple of unusual cases.
If $c$ and $d$ distinct elements not in $\st(a)$, and $u\in\genby{\st(a)}$, then $(cd,duc)$ is a syllable decomposition for the conjugacy class of $cdu$.
Likewise, $(cuc)$ is a syllable decomposition for the conjugacy class of $cu$.
Generally, syllable decompositions are far from being unique; however, it is not hard to see that certain aspects of syllable decompositions are determined by the tuples being decomposed.
First of all, for a given conjugacy class, it is either the case that it has a unique decomposition as a cyclic syllable, or that all of its syllable decompositions are tuples of linear syllables.
Second, if a conjugacy class decomposes into linear syllables, then the number of linear syllables in the decomposition is the number of instances of letters from $X\setminus\st(a)$ appearing in any cyclically and graphically reduced representative.
In particular, the number of syllables in a decomposition is the same for all decompositions.

If $w$ is a conjugacy class in $A_\Gamma$, there are two possible kinds ambiguities that come up in decomposing $w$ into syllables.
First of all, if $T$ is a decomposition of $w$ into syllables, then any cyclic permutation of $T$ is also a decomposition of $w$ into syllables.
Second, if we take the representative associated with a decomposition and get a different graphically reduced product by commuting some letters across syllable boundaries, the resulting product will be associated with a different decomposition of $w$.
These two kinds of ambiguities also come up in decomposing tuples of conjugacy classes into syllables.
As we will see below, this second kind of ambiguity is not  very important.

The purpose of decomposing things into syllables with respect to $a$ is that it gives us a convenient way to encode the action of $\whset{a}$ and to keep track of its effect on lengths.
We recall the free abelian group $Z_{[a]}$ from Section~\ref{ss:genwharelinear}.
We define a map $\nu$ to $Z_{[a]}$ from the set of syllables in $A_\Gamma$ with respect to $a$:
\begin{itemize}
\item If $u$ is a cyclic syllable then $\nu(u)$ is $\sum_{b\in\st(a)\cap \dom(a)} n_b r_b$, where $n_b$ is the sum exponent of $b$ in $u$ for each $b\in\st(a)\cap\dom(a)$.
\item If $cud$ is a linear syllable then $\nu(cud)$ is $v_c+v_d+\sum_{b\in\st(a)\cap\dom(a)} n_b r_b$, where $n_b$ is the sum exponent of $b$ in $u$ for each $b\in\st(a)\cap\dom(a)$, and
\begin{itemize}
\item if $c$ is a positive generator and $c\in \dom(a)$, then $v_c=r_c$,
\item if $c^{-1}$ is a positive generator and $c^{-1}\in\dom(a)$, then $v_c=-l_c$, 
\item otherwise $c$ or $c^{-1}$ is in $Y$, a connected component of $\Gamma\setminus\st(a)$ with two or more vertices, and $v_c=r_Y$;
\item if $d$ is a positive generator and $d\in \dom(a)$, then $v_d=l_c$,
\item if $d^{-1}$ is a positive generator and $d^{-1}\in\dom(a)$, then $v_d=-r_d$, and
\item otherwise $d$ or $d^{-1}$ is in $Y$, a connected component of $\Gamma\setminus\st(a)$  with two or more vertices, and $v_d=-r_Y$.
\end{itemize}
\end{itemize}
The map $\nu$ extends diagonally to define a map $\nu$ from tuples of syllables to tuples of vectors of $Z_{[a]}$.
Note that $\nu$ makes no record of the letters in $cud$ from $\st(a)\setminus\dom(a)$.
In fact, the map $\nu$ gets rid of one of the ambiguities possible in decomposing a class into syllables.

\begin{lemma}\label{le:decompwelldef}
Suppose $w$ is a conjugacy class in $A_\Gamma$ and $T$ and $T'$ are decompositions of $w$ into syllables with respect to $a$.
Then the tuple $\nu(T)$ is a cyclic permutation of the entries of the tuple $\nu(T')$.
Further, if $W$ is a tuple of conjugacy classes and $T$ and $T'$ are decompositions of $W$, then $\nu(T)$ is a permutation of $\nu(T')$.
\end{lemma}
\begin{proof}
As mentioned in Section~\ref{ss:combinatorial}, any graphically reduced word representing a given element can be transformed into any other by a sequence of commutations.
Further, any graphically and cyclically reduced word representing a conjugacy class can be transformed into any other by a sequence of commutations and cyclic permutations.
This observation has a corollary for decompositions of $w$ into syllables.
Any decomposition $T=(c_1u_1c_2,\dotsc,c_ku_kc_1)$ of $w$ into syllables can be turned into any other by a sequence of the following moves:
\begin{itemize}
\item If some $u_q=u'_qu''_q$ (possibly with $u'_q=1$ or $u''_q=1$) and for some $p\leq q$ we have that $c_p$ commutes with  $c_q$, $u'_q$, $u_p$, and all $c_i$ and $u_i$ for $i=p+1,\dotsc,q-1$, and none of the intervening $c_i$ are equal to $c_p^{\pm1}$, then we can replace the syllable $c_{p-1}u_{p-1}c_p$ with $c_{p-1}u_{p-1}u_pc_{p+1}$, delete $c_pu_pc_{p+1}$ from the list, and break $c_qu_qc_{q+1}$ into two adjacent syllables $c_qu'_qc_p$ and $c_pu''_qc_{q+1}$.
In the case that $p=q$, we simply replace $c_{p-1}u_{p-1}c_p$ with $c_{p-1}u_{p-1}u'_pc_p$ and replace $c_pu_pc_{p+1}$ with $c_pu''_pc_{p+1}$.
\item The previous move can be done in reverse.
\item If some $u_p=u'_pxu''_p$ and for some $q>p$ we have $u_q=u'_qu''_q$ (possibly with any of  $u'_p,u''_p,u'_q,u''_q$ equal to $1$) such that $x$ commutes with $u''_p$, $c_q$, $u'_q$ and all $c_i$ and $u_i$ for $i=p+1,\dotsc, q-1$, then we can replace the syllable $c_pu_pc_{p+1}$ with $c_pu'_pu''_pc_{p+1}$ and replace $c_qu_qc_{q+1}$ with $c_qu'_qxu''_qc_{q+1}$.
\item The previous move can be done in reverse.
\item We can cyclically permute the entries of $T$. 
\end{itemize}
One issue that needs to be addressed with these moves is the fact that they send a well formed syllable decomposition to a well formed syllable decomposition.
This follows from the fact that syllables are graphically reduced products.
If one of the replacements above results in a syllable that is not graphically reduced, for example, where the $c_i$ cancels with the $c_{i+1}$, then the original syllable was not graphically reduced.
By virtue of moving things only across things they commute with, if a cancellation is possible after the replacement, it was also possible in the first place.

We consider the effect the first move has on $\nu(T)$ in the case that $p<q$.
This replaces $c_{p-1}u_{p-1}c_p$ with $c_{p-1}u_{p-1}u_pc_{p+1}$.
Since $c_p$ commutes with $u_p$, but $c_p$ is not adjacent to $a$, no generator from $\st(a)\cap\dom(a)$ appears in $u_p$.
Since $c_p$ commutes with $c_{p+1}$, and neither is adjacent to $a$, this means $a$ dominates neither one and they are in the same component $Y$ of $\Gamma\setminus\st(a)$.
These facts imply that $\nu(c_{p-1}u_{p-1}c_p)=\nu(c_{p-1}u_{p-1}u_pc_{p+1})$.
The syllable $c_pu_pc_{p+1}$ is deleted, but since $c_p$ and $c_{p+1}$ are both in $Y$ and since no generator from $\st(a)\cap\dom(a)$ appears in $u_p$, we have $\nu(c_pu_pc_{p+1})=0$.
The intervening syllables $c_iu_ic_{i+1}$ for $i=p+1,\dotsc,q-1$ commute entirely with $c_p$; this implies for each $i$ that $u_i$ contains no generators from $\st(a)\cap\dom(a)$ and that both $c_i$ and $c_{i+1}$ are in $Y$; this implies that for each $i$, $\nu(c_iu_ic_{i+1})=0$.
So although a syllable is deleted before this sequence, thus shifting this sequence forward, there is no effect because this is a sequence of zeros.
The syllable $c_qu_qc_{q+1}$ is broken into $c_qu'_qc_p$ and $c_pu''_qc_{q+1}$.
For the same reason as the intervening syllables, we see $\nu(c_qu'_qc_p)=0$.
Similarly, $c_q$ and $c_p$ must both be in $Y$, and $u'_q$ makes no contribution, so $\nu(c_qu_qc_{q+1})=\nu(c_pu''_qc_{q+1})$.
So the effect of the move on $\nu(T)$ is to replace the entry in position $p$ with the same entry, replace zero entries from position $p+1$ to position $q-1$ with zero entries, and replace the entry in position $q$ with an identical entry.
Therefore the first move does not affect $\nu(T)$ if $p<q$.

If $p=q$, the first move replaces $c_{p-1}u_{p-1}c_p$ with $c_{p-1}u_{p-1}u'_pc_p$ and replaces $c_pu_pc_{p+1}$ with $c_pu''_pc_{p+1}$.
Since $u'_p$ commutes with $c_p$ it contains no generators from $\st(a)\cap\dom(a)$, and we deduce that $\nu(c_{p-1}u_{p-1}c_p)=\nu(c_{p-1}u_{p-1}u'_pc_p)$ and $\nu(c_pu_pc_{p+1})=\nu(c_pu''_pc_{p+1})$.
So the move has no effect on $\nu(T)$ in this case as well.

Next we consider the second kind of move.
We replace $c_pu_pc_{p+1}=c_pu_p'xu_p''c_{p+1}$ with $c_pu_p'u_p''c_{p+1}$; since $x$ commutes with $c_{p+1}$, it is not in $\st(a)\cap\dom(a)$, and therefore
$\nu(c_pu_pc_{p+1})=\nu(c_pu_p'u_p''c_{p+1})$.
We also replace  $c_qu_qc_{q+1}=c_qu'_qu''_qc_{q+1}$ with $c_qu'_qxu''_qc_{q+1}$; for the same reason, $\nu(c_qu_qc_{q+1})=\nu(c_qu'_qxu''_qc_{q+1})$.
Therefore $\nu(T)$ is unchanged for this kind of move.
Of course, if a kind of move leaves $\nu(T)$ unchanged, then the same kind of move in reverse will also leave it unchanged.
So the moves above can only change $\nu(T)$ by cyclically permuting its entries.

If $W$ is a tuple of cyclic words and $T$ is a decomposition of $W$ into syllables, then of course, we can cyclically permute the decompositions of the entries in $W$.
Since $T$ is a concatenation of these decompositions of entries, it is possible that different decompositions of $W$ will differ by more complicated permutations---ones that cyclically permute the segments coming from different entries of $W$.
\end{proof}

\begin{definition}\label{de:syllableaction}
Fix a vertex $a$ in $\Gamma$.
We define an action of $\whset{a}$ on the set of syllables with respect to $a$ as follows:
if $\alpha\in\whset{a}$ and $s=cud$ where  $u\in\genby{\st(a)}$ and $c,d\in(\Gamma\setminus\st(a))^{\pm1}$ and $\alpha(c)=w_1cw_2$ and $\alpha(d)=w_3dw_4$, then define
\[\alpha\cdot s=cw_2\alpha(u)w_3d.\]
We extend this action diagonally to the set of $k$-tuples of syllables for each $k$.
\end{definition}
This action loses the element that $\alpha$ places on the left of the left endpoint of the syllable, and loses the element that $\alpha$ places on the right of the right endpoint of the syllable.
We recall the map $\eta\co\whset{a}\to \Aut(Z_{[a]})$ from Section~\ref{ss:genwharelinear}.
Note that the action of $\whset{a}$ on syllables just described is the same as the action of $\eta(\whset{a})$ on $Z_{[a]}$, in that the injective map $\eta$ intertwines the two actions.

Let $\{a_1,\dotsc,a_n\}=[a]$.
We note that any syllable with respect to $a$ may be written without loss of generality in the form $ca_1^{p_1}\dotsm a_n^{p_n}ud$, where $u\in\genby{\st(a)\setminus[a]}$.
This is because in a general syllable $cu'd$, everything from $[a]$ in $u'$ can be commuted to the left, since $u\in\genby{\st(a)}$.
The following result is a kind of equivariance between the actions of $\whset{a}$ on syllables and on $Z_{[a]}$.

\begin{lemma}\label{le:syllableequivariance}
Suppose $ca_1^{p_1}\dotsm a_n^{p_n}ud$ is a syllable of $A_\Gamma$ with respect to $a$, and suppose $\alpha\in\whset{a}$.
Let $p'_i$ be the coefficient of $r_{a_i}$ in $\alpha\cdot \nu(ca_1^{p_1}\dotsm a_n^{p_n}ud)$.
Then there are elements $v_1,v_2\in\genby{[a]}$ with
\[\alpha(ca_1^{p_1}\dotsm a_n^{p_n}ud)=v_1ca_1^{p'_1}\dotsm a_n^{p'_n}udv_2.\]
\end{lemma}

\begin{proof}
This is a computational exercise in the definitions.
\end{proof}

\begin{lemma}\label{le:decompositionequivariance}
Suppose $W$ is a tuple of cyclic words and 
\[T=(c_1a_1^{p_{1,1}}\dotsm a_n^{p_{1,n}}u_1d_1,\dotsc,c_ma_1^{p_{m,1}}\dotsm a_n^{p_{m,n}}u_md_m)\]
 is a decomposition of $W$ into syllables, with each $u_i\in\genby{\st(a)\setminus[a]}$.
Suppose $\alpha\in\whset{a}$.
Let $p'_{i,j}$ be the coefficient on $r_{a_j}$ in the $i$th coordinate of $\eta(\alpha)\cdot \nu(T)$, for $i=1,\dotsc,m$ and $j=1,\dotsc, n$.
Then
\[\alpha\cdot T=(c_1a_1^{p'_{1,1}}\dotsm a_n^{p'_{1,n}}u_1d_1,\dotsc,c_ma_1^{p'_{m,1}}\dotsm a_n^{p'_{m,n}}u_md_m),\]
and $\alpha\cdot T$ is a syllable decomposition of $\alpha\cdot W$.
\end{lemma}

\begin{proof}
This follows from repeated application of Lemma~\ref{le:syllableequivariance}.
In the case that $w$ is a cyclic word and $(c_1u_1c_2,\dotsc,c_mu_mc_1)$ is its decomposition,
we rewrite the representative associated to this decomposition as
\[c_1u_1c_2 \cdot c_2^{-1}\cdot c_2u_2c_3 \cdot c_2^{-1} \dotsm c_mu_mc_1\cdot c_1^{-1}.\]
We apply $\alpha$ to each factor in this product separately.
Lemma~\ref{le:syllableequivariance} applies to each syllable factor in the product.
When we act on a standalone $c_i^{-1}$ factor in the product, we cancel away the leading and trailing elements from $\genby{[a]}$ from the preceding and following syllables (the $v_1$ and $v_2$ elements from Lemma~\ref{le:syllableequivariance}).
From the resulting product, in which the exponents on the $a_i$--factors inside the syllables are from the action of $\eta(\alpha)$ (as in Lemma~\ref{le:syllableequivariance}), it is easy to see that the corresponding syllable decomposition $T'$ from the lemma statement is a decomposition of $\alpha\cdot w$.
The case that $W$ is a tuple of cyclic words follows by applying the same argument to each entry in the tuple separately.
\end{proof}

\begin{example}\label{ex:syllablealgorithm}
This example illustrates why we need to take a little care with the algorithms for Proposition~\ref{pr:insamewhorbit} and Proposition~\ref{pr:whstab}.
Suppose for this example that $\Gamma$ is the graph with four vertices $\{a,b,c,d\}$, with an edge from $a$ to $b$ and an edge from $c$ to $d$ (so $A_\Gamma$ is $\Z^2 * \Z^2$).
Consider the conjugacy classes $u$ and $v$ represented by $cacbcb$ and $cbcabcb$ respectively.
Choosing syllable decompositions with respect to $a$ arbitrarily, we might choose $T=(cac,cbc,cbc)$ for $u$ and $T'=(cbc,cabc,cbc)$ for $v$.
The group $Z_{[a]}$ is generated by $r_a$, $r_b$ and $r_Y$, where $Y=\{c,d\}$, and $\nu(T)=(r_a,r_b,r_b)$ and $\nu(T')=(r_b,r_b+r_a,r_b)$.
To check whether $\nu(T)$ and $\nu(T')$ are in the same orbit, we apply the algorithm from Proposition~\ref{pr:matrixorbitalgorithm}, after choosing an appropriate identification between $Z_{[a]}$ and $\Z^3$.
We find that $\nu(T)$ and $\nu(T')$ are not in the same orbit.
However, $T''=(cabc,cbc,cbc)$ is also a syllable decomposition for $v$ with respect to $a$, and it is not hard to see that $\nu(T)$ and $\nu(T'')$ are in the same orbit under $\eta(\whset{a})$: the automorphism sending $a$ to $ab^{-1}$ and fixing the other generators sends one to the other.
This automorphism also sends $u$ to $v$.  
This example illustrates the need to consider permutations of a syllable decomposition, instead of only considering a single arbitrary decomposition.

Now consider the conjugacy class $u$ represented by $cacb$ in the same group.
One syllable decomposition for the conjugacy class is $T=(cac,cbc)$.
The automorphism $\alpha$ sending $a$ to $b$ and $b$ to $a$ and fixing the other generators is in $\whset{a}$, and $\alpha\cdot u=u$.
However, $\eta(\alpha)\cdot \nu(T)\neq\nu(T)$.
This illustrates the possibility of automorphisms fixing a conjugacy class but not a particular syllable decomposition of that class.
\end{example}

For our finite presentation result, we need refined versions of Propositions~\ref{pr:insamewhorbit} and~\ref{pr:whstab}.
Specifically, we need to perform the algorithms in these propositions while respecting certain restrictions on the support of automorphisms, which we now define.
\begin{definition}\label{de:support}
The \emph{support} of a generalized Whitehead automorphism $\alpha\in\whset{a}$ is the subset of $X^{\pm1}$ with 
\begin{itemize}
\item for $b$ adjacent to or equal to $a$, $b$ and $b^{-1}$ are both in $\supp(\alpha)$ if $\alpha(b)\neq b$ and neither $b$ nor $b^{-1}$ is in $\supp(\alpha)$ if $\alpha(b)=b$, and
\item for $b$ not adjacent to $a$ with $\alpha(b)=ubv$, $b\in\supp(\alpha)$ if and only if $v\neq 1$ and $b^{-1}\in\supp(\alpha)$ if and only if $u\neq 1$.
\end{itemize}
For $a\in X$ and $S\subset  (X\setminus\st(a))^{\pm1}$, we define $\whsetsr{a}{S}$ to be the subset of $\whset{a}$ consisting of automorphisms $\alpha$ with $\supp(\alpha)\cap S=\varnothing$.
\end{definition}

Suppose $a\in X$ and $S\subset(X\setminus\st(a))^{\pm1}$.
Now we consider the image of $\whsetsr{a}{S}$ under $\eta$.
Say that a basis element of $Z_{[a]}$ does not intersect $S$ if it is of the form $r_b$ or $l_b$ with $b\notin S\cup [a]$, or of the form $r_Y$ with $Y\cap S=\varnothing$.
Suppose $|[a]|=n$, and there are $k$ basis elements of $Z_{[a]}$ not intersecting $S$, and $l$ remaining basis elements for $Z_{[a]}$.
Suppose we identify $Z_{[a]}$ with $\Z^{n+k+l}$ so that the basis elements of the form $r_b$ with $b\in[a]$ map to the first $n$ basis elements, the basis elements not intsecting $S$ map to the next $k$ basis elements, and the remaining basis elements map to the remaining basis elements.

\begin{lemma}\label{le:supportrestrictedsubgroup}
With $a$, $S$ as above, $\whsetsr{a}{S}$ is a subgroup of $\whset{a}$.
Identifying $\Aut(Z_{[a]})$ with $\GL(n+k+l,\Z)$ using the identification of $Z_{[a]}$ with $\Z^{n+k+l}$ above, the image of $\whsetsr{a}{S}$ under $\eta$ is the set of matrices of the form
\[
\begin{pmatrix}
A & B & O_{n,l} \\ 
O_{k,n} & I_k & O_{k,l} \\
O_{l,n} & O_{l,k} & I_l
\end{pmatrix}
\]
where $A\in\GL(n,\Z)$, $B\in M_{n,k}(\Z)$, and the $O$'s and $I$'s represent zero and identity blocks of the indicated dimensions.
\end{lemma}

\begin{proof}
The assertion that this subset is a subgroup is left as an exercise for the reader.

For $\alpha\in\whset{a}$, the definition of $\eta$ tells us that $\eta$ counts the sum exponent of elements of $[a]$ on the right and left sides of elements of $X$.
If $\supp(\alpha)\cap S=\varnothing$, then for any element of $S$, all of these counts are zero.
As explained in the proof of Lemma~\ref{le:etainjective}, the sum exponents are the same for both sides of all elements in the same connected component of $X\setminus\st(a)$, if that component has at least two vertices.
So if a basis element intersects $S$, then our counts of sum exponents are all zero for $\alpha\in\whsetsr{a}{S}$, which explains the shape of the matrix.

To see that any matrix of this shape is in the image, we use the same argument as in Lemma~\ref{le:etainjective}.
\end{proof}

\begin{lemma}\label{le:supportrestrictedmatrixalgs}
There is an algorithm to check whether two matrices in $M_{n+k+l,m}(\Z)$ are in the same orbit under the group $G$ of block matrices of the form
\[
\left(
\begin{array}{ccc}
A & B & O_{n,l} \\ O_{k,n} & I_k & O_{k,l} \\  O_{l,n} & O_{l,k} & I_l
\end{array}\right),
\]
where $A\in\GL(n,\Z)$, $B\in M_{n,k}(\Z)$, and the $O$'s and $I$'s represent zero and identity blocks of the indicated dimensions.

Further, there is an algorithm that returns a presentation for the stabilizer of a matrix in $M_{n+k+l,m}$ under the action of $G$.
\end{lemma}

\begin{proof}
Suppose $C$ and $D$ are two matrices in $M_{n+k+l,m}$.
If the last $l$ rows of $C$ do not match the last $l$ rows of $D$, then they cannot be in the same orbit.
So we suppose these last $l$ rows match.
Next we consider $C'$ and $D'$ in $M_{n+k.m}$, where each is $C$ or $D$ respectively with the last $l$ rows omitted.
The group $G$ above is isomorphic to the group $G_1$ of Proposition~\ref{pr:matrixorbitalgorithm}, by the mapping that omits the last $l$ rows and the last $l$ columns,
and $C$ is in $G\cdot D$ if and only if $C'$ is in $G_1\cdot D'$.
Of course, this is exactly what Proposition~\ref{pr:matrixorbitalgorithm} checks.

Similarly, the stabilizer of $C$ in $G$ will be isomorphic to the stabilizer of $C'$ in $G_1$, and Proposition~\ref{pr:matrixstabpres} provides a presentation for this.
\end{proof}

So instead of proving Proposition~\ref{pr:insamewhorbit},
we prove the following proposition.
Proposition~\ref{pr:insamewhorbit} is the special case where the set $S$ is empty.
\begin{proposition}\label{pr:insamewhorbitsupportrestricted}
There is an algorithm that takes in two tuples $U$ and $V$ of conjugacy classes from $A_\Gamma$, a vertex $a$ of $\Gamma$ and a subset $S$ of $(X\setminus\st(a))^{\pm1}$, and produces an automorphism $\alpha\in \whsetsr{a}{S}$ with $\alpha\cdot U=V$, or determines that no such automorphism exists.
\end{proposition}

\begin{proof}
Suppose $U$ and $V$ are two tuples of conjugacy classes of $A_\Gamma$.
The first step in the algorithm is to form syllable decompositions $T$ of $U$ and $T'$ of $V$.
If the syllable decompositions $T$ and $T'$ do not have the same number of entries, then $U$ and $V$ are not in the same orbit (this follows from Lemma~\ref{le:decompositionequivariance}, since the decomposition of $\alpha\cdot W$ is the same length as the decomposition of $W$).
So we suppose that $T$ and $T'$ have the same number of entries.
We consider all the permutations of the entries of $T'$ and select from these the ones $T'_1,\dotsc, T'_m$ that are also syllable decompositions of $V$.

Suppose that
\[T=(c_1a_1^{p_{1,1}}\dotsm a_k^{p_{1,k}}u_1d_1,\dotsc,c_ma_1^{p_{m,1}}\dotsm a_k^{p_{m,k}}u_md_m).\]
We fix an $r$ from $1$ through $m$, and suppose
\[T'_r=(c'_1a_1^{p'_{1,1}}\dotsm a_k^{p'_{1,k}}u'_1d'_1,\dotsc,c'_ma_1^{p'_{m,1}}\dotsm a_k^{p'_{m,k}}u'_md'_m).\]
We define a tuple $\hat T_r$ to be $T$ with the exponents of the $a_i$ replaced by those from $T'_r$:
\[\hat T_r=(c_1a_1^{p'_{1,1}}\dotsm a_k^{p'_{1,k}}u_1d_1,\dotsc,c_ma_1^{p'_{m,1}}\dotsm a_k^{p'_{m,k}}u_md_m).\]
At this point, we check whether $\hat T_r$ is a decomposition of $V$ (this amounts to finding the representative associated to $\hat T_r$ and checking whether it represents the $V$).

If the answer is yes, we use the algorithm from Proposition~\ref{pr:matrixorbitalgorithm} to check whether there is $A\in \eta(\whset{a})$ with $A\cdot \nu(T)=\nu(\hat T_r)$ and with $A$ in the image of $\whsetsr{a}{S}$.
By Lemma~\ref{le:supportrestrictedsubgroup} we know that $\eta(\whsetsr{a}{S})$ is $G_1$ with the last $l$ columns (without loss of generality) zeroed out in its upper right block (for some $l$) and we use the modification of the algorithm from Proposition~\ref{pr:matrixorbitalgorithm} described in Lemma~\ref{le:supportrestrictedmatrixalgs} to check for such automorphisms.

If our algorithm finds such a matrix $A$, let $\alpha\in\whsetsr{a}{S}$ map to $A$.
By Lemma~\ref{le:decompositionequivariance}, we have that $\hat T_r$ is a representative for $\alpha\cdot U$.
Of course, in that case, we have $\alpha\cdot U=V$ and the algorithm returns $\alpha$.

If we get negative answers (either there is no matrix $A$ as above or $\hat T_r$ does not represent $V$), we increment $r$ and check the next $T_r$.
If we try this for each $T_r$ and none of them have a matrix $A$ as above, we declare that $U$ and $V$ are in different orbits.

To show the correctness of the algorithm, we suppose that we have $\alpha\in\whsetsr{a}{S}$ with $\alpha\cdot U=V$.
We want to show that the algorithm finds an automorphism sending one tuple to the other.
We take the coefficients from $\eta(\alpha)\cdot \nu(T)$ and substitute them into the exponents of $T$ to get a tuple $T''$;
by Lemma~\ref{le:decompositionequivariance}, $T''$ is a syllable decomposition for $V$.
By the argument in Lemma~\ref{le:decompwelldef}, $T''$ will differ from other syllable decompositions for $V$ by a sequence of permutations and commutation of elements not in $\dom(a)$ across each other.
This means that one of the $T'_r$ will differ from $T''$ only by the positions of elements not in $\dom(a)$.
In particular, $\hat T_r=T''$.
This means that the algorithm will catch that $\hat T_r$ is a representative for $V$, and then will catch that there is some $\beta\in\whsetsr{a}{S}$ with  
$\eta(\beta)\cdot \nu(T)=\nu(\hat T_r)$ (of course this is true for $\beta=\alpha$, but there is no guarantee that the algorithm will catch the same automorphism).
So by the contrapositive, if the algorithm does not catch any such automorphism, then no such automorphism exists.
\end{proof}

Our next goal is to prove Proposition~\ref{pr:whstab}.
Again, we need a slight refinement of the proposition for our argument for finite presentability of stabilizers.
So we prove the following, which implies Proposition~\ref{pr:whstab} as a special case.
\begin{proposition}\label{pr:whstabsupportrestricted}
There is an algorithm that takes in a tuple $U$ of conjugacy classes from $A_\Gamma$, and an element $a\in X$ and a subset $S\subset(X\setminus\st(a))^{\pm1}$ and returns a presentation for the stabilizer $(\whsetsr{a}{S})_U$.
\end{proposition}

\begin{proof}
Let $U$ be a tuple of conjugacy classes in $A_\Gamma$.
Let $T_1$ in $Z_{[a]}^M$ be a syllable decomposition of $U$ with respect to $a$.
Let $T_1,\dotsc, T_m$ be all the permutations of $T_1$ that are also syllable decompositions of $U$.
The group $\whsetsr{a}{S}$ acts on $Z_{[a]}^M$ and a given $\nu(T_i)$ may or may not be in the orbit of $\nu(T_1)$ under this action.
We reorder $T_1,\dotsc,T_m$ so that the intersection of the $\{\nu(T_i)\}_i$ with the orbit of $\nu(T_1)$ is $\nu(T_1),\dotsc,\nu(T_k)$ for some $k$.
If $\alpha\in\whsetsr{a}{S}$ with $\alpha\cdot U=U$ and $i=1,\dotsc,k$, then by Lemma~\ref{le:decompositionequivariance}, the element $\alpha\cdot\nu(T_i)$ is the image under $\nu$ of a syllable decomposition of $U$.
Then by Lemma~\ref{le:decompwelldef}, the element $\alpha\cdot\nu(T_1)$ is one of $\nu(T_1),\dotsc,\nu(T_k)$.
Therefore $(\whsetsr{a}{S})_U$ acts on the finite set $\{\nu(T_1),\dotsc,\nu(T_k)\}$, and by construction, it acts transitively.
Then $(\whsetsr{a}{S})_{\nu(T_1)}$ is a finite index subgroup of $(\whsetsr{a}{S})_U$.

The observations we just made make it possible to see the correctness of the following algorithm.
First we find a syllable decomposition $T_1$ of $U$.
Then we enumerate all the permutations of $T_1$ that are also syllable decompositions of $U$, say $T_1,\dotsc, T_m$.
Then we use the algorithm from Proposition~\ref{pr:matrixorbitalgorithm} and its modification in Lemma~\ref{le:supportrestrictedmatrixalgs}  to go through the list $\nu(T_1),\dotsc,\nu(T_m)$ to determine which of these are in the same $\whsetsr{a}{S}$--orbit as $\nu(T_1)$.
By relabeling the list, we assume that these are $\nu(T_1),\dotsc, \nu(T_k)$.
Also, as we use the algorithm to check which $\nu(T_i)$ are in the same orbit, we record an example automorphism in $\whsetsr{a}{S}$ that sends $\nu(T_1)$ to $\nu(T_i)$, for each $i$ from $1$ to $k$.
Let $S_1$ denote the set of these automorphisms.

Next we use the algorithm from Proposition~\ref{pr:matrixstabpres} and its modification from Lemma~\ref{le:supportrestrictedmatrixalgs} to find a presentation for $(\whsetsr{a}{S})_{\nu(T_1)}$.
Let $S_2$ denote the generating set for $(\whsetsr{a}{S})_{\nu(T_1)}$ from this presentation.
For each $\nu(T_i)$ for $i$ from $1$ to $k$, we check where each element of $S_1\cup S_2$ sends it (since $(\whsetsr{a}{S})_U$ acts on $\{\nu(T_1),\dotsc,\nu(T_k)\}$, the image will be one of these).
We construct the finite graph whose vertices are $\nu(T_1),\dotsc,\nu(T_k)$ and whose edges are labeled by the elements of $S_1\cup S_2$, with an edge labeled by $\alpha\in S_1\cup S_2$ from $\nu(T_i)$ to $\nu(T_j)$ if and only if $\alpha\cdot \nu(T_i)=\nu(T_j)$.
It is not hard to see that this graph is isomorphic to the Schreier graph of $(\whsetsr{a}{S})_{\nu(T_1)}$ in $(\whsetsr{a}{S})_U$ with respect to $S_1\cup S_2$.
Note that by the choice of $S_1$, this Schreier graph is connected.
Using this Schreier graph together with the presentation for $(\whsetsr{a}{S})_{\nu(T_1)}$, we then construct a finite presentation for $(\whsetsr{a}{S})_U$ using the procedure in Lemma~\ref{le:presfromfinidx}.
\end{proof}

\section{Applications}\label{se:applications}
\subsection{A useful finite graph}\label{ss:usefulgraph}
In this section we prove our two applications, Theorem~\ref{th:raagcheckorbit} and Theorem~\ref{th:raagstabpres}, modulo the technical results that we prove in the later sections.

Suppose $W$ is an $M$--tuple of cyclic words in $A_\Gamma$ and $W$ is of minimal length in its $\Aut(A_\Gamma)$--orbit.
We construct a directed, labeled multigraph $\Delta$ associated to $W$ as follows.
\begin{itemize}
\item The vertices of $\Delta$ are the set of $M$--tuples of cyclic words of $A_\Gamma$ of the same length as $W$.
\item For each pair of vertices $W_1$ and $W_2$, possibly with $W_1=W_2$, and for each permutation automorphism $\alpha\in P$ with $\alpha\cdot W_1=W_2$, there is a directed edge from $W_1$ to $W_2$ labeled by $\alpha$.
\item For each pair of distinct vertices $W_1$ and $W_2$ and each generator $a\in X$, if there is an automorphism in $\whset{a}$ sending $W_1$ to $W_2$, then there is a directed edge from $W_1$ to $W_2$ labeled by some $\alpha\in\whset{a}$ with $\alpha \cdot W_1=W_2$ (this involves a choice).
\item For each vertex $W_1$ and each generator $a\in X$, there are edges from $W_1$ to $W_1$ labeled by a finite generating set for the stabilizer $(\whset{a})_{W_1}$ (this also involves choices).
\end{itemize}

\begin{lemma}
The graph $\Delta$ associated to the minimal tuple $W$ is finite and can be effectively constructed.
\end{lemma}

\begin{proof}
First of all, $\Delta$ has finitely many vertices because there are finitely many tuples of a given length.
We construct $\Delta$ by finding the required edges and attaching them to the $0$--skeleton.
There are finitely many permutation automorphisms, so we can explicitly check which ones send which vertices to which vertices.
For each generator $a\in X$ and each pair of vertices, we use Proposition~\ref{pr:insamewhorbit} to check whether there is an automorphism in $\whset{a}$ sending one to the other, and if there is, we add an edge labeled by such an automorphism (Proposition~\ref{pr:insamewhorbit} gives us one if one exists).
For each generator $a$ and each vertex $W_1$, we use Proposition~\ref{pr:whstab} to get a finite generating set for the stabilizer $(\whset{a})_{W_1}$.
These last two steps are effective since there are only finitely many generators in $X$ and vertices in $\Delta$.
\end{proof}

\begin{lemma}\label{le:composedelta}
If $\alpha$ is the composition of edge labels on a path from a vertex $W_1$ to a vertex $W_2$ in $\Delta$, then $\alpha\cdot W_1=W_2$.
\end{lemma}

\begin{proof}
This is true for paths of length one by construction and true in general by Lemma~\ref{le:travelSchreier}.
\end{proof}

\begin{lemma}\label{le:peakredongraph}
Suppose $W_0$ is a vertex in $\Delta$ and $W_0$ is minimal length in its automorphism orbit.
If $W'$ is also a vertex in $\Delta$ and $\alpha\in\Aut(A_\Gamma)$ with $\alpha\cdot W_0=W'$, then there is a path $p$ in $\Delta$ from $W_0$ to $W'$ such that the composition of edge labels along $p$ is $\alpha$.
\end{lemma}

\begin{proof}
We peak reduce $\alpha$ with respect to $W_0$ by elements of $\Omega$, which is possible by Theorem~\ref{th:fullfeaturedpeakreduction}.
Suppose $\alpha=\alpha_k\dotsm\alpha_1$ is the resulting factorization.
Let $W_i=\alpha_i\dotsm\alpha_1\cdot W_0$ for $i=1,\dotsc,k$, so that $W_k=W'$.
Since $W_0$ is minimal length, the factorization being peak reduced means that  $\abs{W_k}=\abs{W_0}$ for $i=0,\dotsc,k$.
Then each $W_i$ is a vertex in $\Delta$.
If $\alpha_i$ is a permutation automorphism, then there is an edge from $W_{i-1}$ to $W_i$ labeled by $\alpha_i$ by construction.
If $\alpha_i\in\whset{a}$ for some $a$, then there is an edge from $W_{i-1}$ to $W_i$ labeled by $\beta_i$ for some $\beta_i\in\whset{a}$.
Then $\beta_i^{-1}\alpha_i$ stabilizes $W_{i-1}$.
By construction, the edge loops at $W_{i-1}$ contain generators for that stabilizer, and there is a path $p_i$ in the loops at $W_{i-1}$ whose edge labels compose to be $\beta_i^{-1}\alpha_i$.
So following $p_i$ and then the edge labeled by $\beta_i$ gives a path $p'_i$ from $W_{i-1}$ to $W_i$ such that the composition of labels on $p'_i$ is $\alpha_i$.
Composing these paths as $i$ goes from $1$ to $k$ gives a path from $W_0$ to $W_k$ whose edge label composition is $\alpha$.
\end{proof}

\subsection{Orbit membership and finite generation}
\begin{lemma}\label{le:checkminlength}
Suppose $W$ is a tuple of conjugacy classes from $A_\Gamma$.
Then $W$ is minimal length in its $\Aut(A_\Gamma)$--orbit if and only if
it cannot be shortened by any element of $\whset{a}$ for any $a\in X$.
\end{lemma}
\begin{proof}
It is clear that a minimal-length tuple cannot be shortened, so we prove the other direction.
Suppose for contradiction that there is some $W'$ with $\abs{W'}<\abs{W}$ and some $\alpha\in\Aut(A_\Gamma)$ with $\alpha\cdot W= W'$.
We peak reduce $\alpha$ with respect to $W$ by Theorem~\ref{th:fullfeaturedpeakreduction}.
Then $\alpha=\beta_k\dotsm\beta_1$ with $\beta_i\in\Omega$ for all $i$, where $k\mapsto\abs{\beta_k\dotsm\beta_1\cdot W}$ is a sequence of lengths that decreases, stays level, and then increases (with any of these phases possibly omitted).
Since $\abs{\alpha\cdot W}<\abs{W}$, the decreasing phase cannot be omitted, and therefore $\abs{\beta_1\cdot W}<\abs{W}$.
This contradiction proves the lemma.
\end{proof}

\begin{proof}[Proof of Theorem~\ref{th:raagcheckorbit}]
Let $U$ and $V$ be two $M$-tuples of conjugacy classes from $A_\Gamma$; we want to check whether they are in the same $\Aut(A_\Gamma)$--orbit.
We start by enumerating the tuples of conjugacy classes from $A_\Gamma$ that are strictly shorter than $\abs{U}$.
Of course there are only finitely many such conjugacy classes (even so, this step is a disappointing bottleneck in the algorithm).
For each $U'$ strictly shorter than $U$, and each class $[a]$ with $a\in\Gamma$, we use Proposition~\ref{pr:insamewhorbit} to check whether $U'$ is in the same orbit as $U$ under $\whset{a}$.
If it is, we replace $U$ by $U'$ and repeat the previous step again, checking whether one of the $\{\whset{a}\}_a$ can shorten $U$.
We stop when we have verified that none of these sets of automorphisms can shorten $U$.
By Lemma~\ref{le:checkminlength}, this resulting $U$ is of minimal length.

After we shorten $U$ as much as possible,
we do the same to $V$.
If the minimal lengths are different, we declare that the orbits are different.
Now we suppose that $U$ and $V$ are both minimal length in their automorphism orbits with $\abs{U}=\abs{V}$; 
let $\Delta$ be the graph from Section~\ref{ss:usefulgraph} above, constructed using $U$
(the vertices of $\Delta$ are tuples of length $\abs{U}$).
At this point, we check whether $U$ and $V$ are in the same connected component of $\Delta$ (this is doable since $\Delta$ is a finite graph).
By Lemma~\ref{le:composedelta}, if $U$ and $V$ are in the same component, then there is an automorphism $\alpha\in \Aut(A_\Gamma)$ with $\alpha\cdot U=V$, and we can find such an $\alpha$ by composing the edge labels on a path from $U$ to $V$.
Conversely, if there is an automorphism $\alpha$ sending $U$ to $V$, 
then there is a path from $U$ to $V$ in $\Delta$ and both are in the same connected component.
This is true by Lemma~\ref{le:peakredongraph}, which uses Theorem~\ref{th:fullfeaturedpeakreduction}. 
\end{proof}

At this point we can quickly prove an intermediate result.
By $\pi_1(\Delta,W)$ we mean the fundamental group of $\Delta$ based at $W$; this may be interpreted either combinatorially or topologically.
\begin{proposition}\label{pr:stabfg}
Suppose $W$ is a tuple of cyclic words in $A_\Gamma$.
Then there is an algorithm to find a finite generating set for the stabilizer $\Aut(A_\Gamma)_W$.
\end{proposition}
\begin{proof}
Altering $W$ by an automorphism sends $\Aut(A_\Gamma)_W$ to a conjugate, so we are free to replace $W$ with a representative of its orbit of minimal length.
We find such an element using Lemma~\ref{le:checkminlength} and Proposition~\ref{pr:insamewhorbit}.
Now assuming that $W$ has minimal length in its orbit, we construct the graph $\Delta$ as above.
Then composition of edge labels defines a homomorphism $\pi_1(\Delta,W)\to \Aut(A_\Gamma)$ (this cannot fail because the domain is free).
By Lemma~\ref{le:composedelta} (with $W_1=W_2=W$), the image of this homomorphism lies in $\Aut(A_\Gamma)_W$.
By Lemma~\ref{le:peakredongraph} (with $W_0=W'=W$), this homomorphism surjects on $\Aut(A_\Gamma)_W$.
\end{proof}

\subsection{Relations}\label{ss:relations}
Our next goal is to show that stabilizers of tuples of conjugacy classes are finitely presentable.
Before we start, we need to record some relations that hold among the generalized Whitehead automorphisms.
First we mention relations between classic Whitehead automorphisms.
In our terminology, a \emph{classic Whitehead automorphism} is either
\begin{itemize}
\item a permutation automorphism from $P$ (an automorphism that restricts to a permutation of $X^{\pm1}$) or
\item an automorphism $\alpha$ with a special element $a$ of $X$, called its \emph{multiplier}, such that $\alpha$ sends $a$ to $a$ and sends each $b\in \Gamma\setminus \{a\}$ to one of $b$, $ba$, $a^{-1}b$ or $a^{-1}ba$.
\end{itemize}
The relations between classic Whitehead automorphisms from Definition~2.6 of Day~\cite{Day1} play a limited role in the current paper.
We can essentially treat these as a black box.

We also need peak reduction of long-range Whitehead automorphisms, which we likewise can treat as a black box.
A generalized Whitehead automorphism $\alpha$ in $\whset{a}$ 
is \emph{short-range} if $\alpha(b)=b$ for all $b$ not adjacent to $a$ (but $\alpha$ may do anything to $\st(a)$).
It is \emph{long-range} if the  restriction of $\alpha$ to $\st(a)^{\pm1}$ is a permutation of that set (but $\alpha$ may do anything to $X\setminus\st(a)$).
More generally, an automorphism in $\Aut(A_\Gamma)$ is \emph{long-range} if and only if it can be factored as a product of long-range dominated transvections, partial conjugations, inversions and graphic automorphisms.
The important fact about long-range automorphisms is:
\begin{theorem}[Day, from~\cite{Day1}, Theorem~A]\label{th:longrangepeakreduction}
If $W$ is a tuple of conjugacy classes in $A_\Gamma$ and $\alpha\in\Aut(A_\Gamma)$ is a long-range automorphism, then $\alpha$ has a factorization by classic long-range Whitehead automorphisms and permutation automorphisms that is peak reduced with respect to $W$ (that is, $\alpha$ has a factorization by these kinds of automorphisms that satisfies the conclusions of Theorem~\ref{th:fullfeaturedpeakreduction}).
\end{theorem}
We need the following relations.
These are like the Steinberg relations from algebraic K-theory.
\begin{lemma}\label{le:steinbergrel}
Suppose $a,b\in\Gamma$ with $[a]\neq[b]$, and we have $\alpha\in\whset{a}$, and $\beta\in\whset{b}$.
Further suppose $\alpha$ restricts to the identity on $[b]$.
Then $\alpha\beta\alpha^{-1}\in\whset{b}$ and $\alpha\cdot \beta\cdot \alpha^{-1}=(\alpha\beta\alpha^{-1})$ is an identity among generalized Whitehead automorphisms
if either of the following is true:
\begin{itemize}
\item $a$ is adjacent to $b$, or
\item  $\supp(\alpha)\cap\supp(\beta)=\varnothing$ and $\beta$ restricts to the identity on $[a]$.
\end{itemize}
\end{lemma}

We further note that 
if $\beta$ restricts to the identity on $[a]$, then
in fact $\alpha\beta\alpha^{-1}=\beta$.

\begin{proof}[Proof of Lemma~\ref{le:steinbergrel}]
This follows by straightforward computations, which we describe in broad strokes.
Let $\gamma$ denote $\alpha\beta\alpha^{-1}$ and suppose that $c$ is in $X$; we need to show that $\gamma(c)$ is in $\genby{[b]}$ if $c\in[b]$, and that $\gamma(c)=u_1cu_2$ with $u_1,u_2\in\genby{[b]}$ if $c\notin[b]$.
If $c\in[b]$, then it is clear that $\gamma(c)=\beta(c)$.
If $c\in[a]$, then in all three cases it is straightforward to show that $\gamma(c)=ucv$ for some $u,v$ in $\genby{[b]}$.
The reasons for this are different in the three cases above.
Now suppose that $c\notin[a]\cup[b]$.
Of course, there are $u_1$ and $u_2$ in $\genby{[a]}$ with $\alpha(c)=u_1cu_2$ and $v_1$ and $v_2$ in $\genby{[b]}$ with $\beta(c)=v_1cv_2$.
Then $\alpha\beta\alpha^{-1}$ sends $c$ to
\[\alpha\beta\alpha^{-1}(u_1^{-1})v_1u_1cu_2v_2\alpha\beta\alpha^{-1}(u_2^{-1}).\]
In the first case, $\alpha\beta\alpha^{-1}(u_1)$ differs from $u_1$ by an element of $\genby{[b]}$ and $u_1$ and $v_1$ commute, so the result follows.
In the second case, either $u_i$ or $v_i$ is trivial for $i=1,2$ and the result follows.
\end{proof}

\subsection{The stabilizer presentation complex}\label{ss:stabpres}
Let $W$ be an $M$--tuple of cyclic words in $A_\Gamma$ that is minimal length in its automorphism orbit.
To prove Theorem~\ref{th:raagstabpres}, we build a finite cellular $2$--complex $Z$ whose fundamental group is the stabilizer $\Aut(A_\Gamma)_W$.
The $1$--skeleton $Z^1$ is like the graph $\Delta$ defined earlier, but with some extra edges.
In order to define a map $\pi_1(Z,W)\to \Aut(A_\Gamma)$, we give the $Z^1$ the structure of a labeled multigraph.
\begin{itemize}
\item The vertices $Z^0$ are the $M$--tuples of cyclic words in $A_\Gamma$ of the same length as $W$ and in the same orbit.
\item For each pair of vertices (not necessarily distinct) $W_1$ and $W_2$ and each classic Whitehead automorphism $\alpha$, we add an edge from $W_1$ to $W_2$ labeled by $\alpha$ if $\alpha\cdot W_1=W_2$.
This includes the cases where $\alpha$ is a permutation automorphism.
\item For each pair of distinct vertices $W_1$ and $W_2$, each generator $a\in X$ and each subset $S\subset (X\setminus\st(a))^{\pm1}$, if there is an element of $\whsetsr{a}{S}$ sending $W_1$ to $W_2$, then we make sure there is an edge from $W_1$ to $W_2$ labeled by some such element.
\item For each vertex $W_1$, each generator $a\in X$ and each subset $S\subset(X\setminus\st(a))^{\pm1}$, the labels on edges from $W_1$ to itself must include a generating set from a presentation for the stabilizer $(\whsetsr{a}{S})_{W_1}$.
\end{itemize}
Like $\Delta$, the graph $Z^1$ can be effectively constructed using Propositions~\ref{pr:insamewhorbitsupportrestricted} and~\ref{pr:whstabsupportrestricted}.
Instead of checking whether a tuple is in the same length as $W$ to decide whether to use it as a vertex, it may be more efficient to construct $\Delta$ as above, discard other connected components, and then add extra edges to form $Z^1$.
Since we only consider vertices in the same orbit as $W$, $Z^1$ is automatically connected and each vertex $W_1$ of $Z$ is minimal length in its orbit.

Next we define the several situations where we add $2$--cells to $Z$.
When we say a $2$--cell ``reads off" a word starting at a given vertex, we mean that we glue in the $2$--cell so that its boundary follows the path whose edge labels form that word.
\begin{itemize}
\item[(C1)] For each vertex $W_1$, each generator $a\in X$ and each subset $S\subset(X\setminus\st(a))^{\pm1}$, the self-edges at $W_1$ labeled by elements of $\whsetsr{a}{S}$ give a generating set for $(\whsetsr{a}{S})_{W_1}$ by construction.
We add $2$--cells reading off the relations between these elements, and we add enough $2$--cells so that the subcomplex of $Z$ spanned by these edges forms a presentation complex for $(\whsetsr{a}{S})_{W_1}$.
This is possible by Proposition~\ref{pr:whstabsupportrestricted}.
\item[(C2)] 
Suppose there is an edge starting at the vertex $W_1$ with label $\alpha$, where $\alpha$ is a long-range generalized Whitehead automorphism but not a classic Whitehead automorphism.
We find a path from $W_1$ to $\alpha\cdot W_1$ with label sequence $\gamma_1$--$\gamma_2$--$\dotsm$--$\gamma_k$, where each $\gamma_i$ is a classic long-range Whitehead automorphism, and glue in a $2$--cell reading off the difference between these two paths.
This is possible by Theorem~\ref{th:longrangepeakreduction}: we peak-reduce $\alpha$ with respect to $W_1$ to get the factorization $\alpha=\gamma_k\dotsm\gamma_1$; since the factorization is peak reduced and $W_1$ is minimal length, each intermediate image $\gamma_i\dotsm\gamma_1\cdot W_1$ is also a vertex of $Z$ and this word defines an edge path.
\item[(C3)]
Whenever we find an edge loop in $Z$ whose labels read off one of the relations between classic Whitehead automorphisms from Day~\cite[Definition~2.6]{Day1}, we add a $2$--cell bounding this edge loop.
These relations fall into ten classes, are easily recognizable, and each such relation has length at most five.
\item[(C4)]
Suppose $W_1$ is in $Z^0$, $\alpha$ is a generalized Whitehead automorphism labeling an edge starting at $W_1$, and $\beta$ is an inner automorphism that is also a classic Whitehead automorphism.
Then there is a factorization $\alpha\beta\alpha^{-1}=\gamma_k\dotsm\gamma_1$, where the $\gamma_i$ are also inner classic Whitehead automorphisms.
The inner classic Whitehead automorphisms label loops in $Z^1$ and we glue in a $2$--cell reading off the difference between these two factorizations starting at $\alpha\cdot W_1$.
We repeat this for each such $W_1$, $\alpha$ and $\beta$.
\item[(C5)]
Suppose $W_1$ is in $Z^0$ and there is a closed edge path $p$ starting at $W_1$ whose edge labels $\alpha,\beta,\gamma$ are in $\whset{a}$ for some $a\in X$.
Then $\gamma\beta\alpha\in(\whset{a})_{W_1}$.
Since the labels on the loops at $W_1$ include generators for $(\whset{a})_{W_1}$, we know that there is an edge path $w$ consisting of loops at $W_1$ whose composition represents the same automorphism as $\gamma\beta\alpha$.
Then we add a $2$--cell to $Z$ whose boundary follows $p$ and then follows $w$ backwards.
We add such a cell for each vertex $W_1$ on each such path $p$ involving at least two vertices.
\item[(C6)]
Suppose $W_1$ is in $Z^0$, $\alpha\in P$ and $\beta\in\whset{b}$ for some $b$, and  both $\alpha$ and $\beta$ are edge labels on edges starting at $W_1$.
Then since $\alpha$ is a permutation automorphism, $\alpha\beta\cdot W_1$ is also a vertex of $W_1$.
It is easy to see that the element $\alpha\beta\alpha^{-1}$ is in $\whset{\alpha(b)}$ and sends $\alpha\cdot W_1$ to $\alpha\beta\cdot W_1$.
By construction, there is an edge in $Z^1$ labeled by some $\gamma\in\whset{\alpha(b)}$ with $\gamma\alpha\cdot W_1=\alpha\beta\cdot W_1$, and therefore $\alpha\beta^{-1}\alpha^{-1}\gamma$ is in $(\whset{\alpha(b)})_{W_1}$.
Since the loop edge labels at $W_1$ include a generating set for $(\whset{\alpha(b)})_{W_1}$, we have a path $w$ in these loops where the composition of edge labels represents $\alpha\beta^{-1}\alpha^{-1}\gamma$.
Then we add a $2$--cell to $Z$ whose boundary, starting at $\alpha\cdot W_1$, follows $\alpha^{-1}$ then $\beta$ then $\alpha$ then $\gamma^{-1}$, and then $w$.
We repeat this for each vertex $W_1$ and each such pair $\alpha$ and $\beta$.
\item[(C7)]
Suppose $W_1$ is in $Z^0$ and $a,b\in X$ and $\alpha\in\whset{a}$ and $\beta\in\whset{b}$ are edge labels on edges starting at $W_1$ 
and $\alpha$ and $\beta$ satisfy the hypotheses of Lemma~\ref{le:steinbergrel} (so  $\alpha|_{[b]}$ is the identity, and either $a$ is adjacent to $b$, or $\beta|_{[a]}$ is also the identity and $\supp(\alpha)\cap\supp(\beta)=\varnothing$).
Then Proposition~\ref{pr:steinberglengthchange} below implies that $\alpha\beta\cdot W_1$ is the same length as $W_1$ and therefore is a vertex in $Z^0$.
So by construction, there is an automorphism $\gamma\in\whset{a}$ labeling an edge from $\beta\cdot W_1$ to $\alpha\beta\cdot W_1$ ($\alpha$ is such an automorphism, but there is no guarantee that our construction of $Z^1$ found this particular automorphism.)
Further, $\alpha\beta\alpha^{-1}$ is in $\whset{b}$ by Lemma~\ref{le:steinbergrel}, and $\alpha\beta\alpha^{-1}$ sends $\alpha\cdot W_1$ to $\alpha\beta\cdot W_1$.
By construction of $Z^1$, there is $\delta\in\whset{b}$ sending $\alpha\cdot W_1$ to $\alpha\beta\cdot W_1$, and there is an edge in $Z^1$ labeled by $\delta$ from $\alpha\cdot W_1$ to $\alpha\beta\cdot W_1$.
We note that $\alpha\gamma^{-1}$ fixes $\alpha\beta\cdot W_1$, and therefore there is a path $w_1$ in the edge loops at $\alpha\beta\cdot W_1$ where the composition of the edge labels represents $\alpha\gamma^{-1}$.
Similarly, $\delta\alpha\beta^{-1}\alpha^{-1}$ fixes $\alpha\beta\cdot W_1$ and there is a path $w_2$ in the edge loops at $\alpha\beta\cdot W_1$ where the composition of these edge loops represents $\delta\alpha\beta^{-1}\alpha^{-1}$.
We glue a $2$--cell into $Z$ whose boundary, starting at $W_1$, follows $\beta$, then $\gamma$, then $w_1$, then $w_2$, then $\delta^{-1}$, and finally $\alpha^{-1}$.
We repeat this for each vertex $W_1$ and each pair $\alpha,\beta$ at $W_1$ satisfying the hypotheses of Lemma~\ref{le:steinbergrel}.
\end{itemize}
This completes the construction of $Z$.
We note that in principle, $Z$ can be effectively constructed from $\Gamma$, since in each of the cases (C1) through (C7), there are only finitely many cases in which we may have to add a $2$--cell.

\begin{lemma}\label{le:surjhomZ}
Composition of edge labels defines a surjective homomorphism $\pi_1(Z,W)\to \Aut(A_\Gamma)_W$.
\end{lemma}
\begin{proof}
The proof of Lemma~\ref{le:composedelta} goes through with $\Delta$ replaced by $Z^1$, since the extra edges in $Z^1$ still indicate the action by their labels.
Then we have a well defined homomorphism $\pi_1(Z^1,W)\to \Aut(A_\Gamma)_W$ for the same reasons as in the proof of Proposition~\ref{pr:stabfg}.
By Lemma~\ref{le:peakredongraph} (with $W_0=W'=W$), this homomorphism surjects on $\Aut(A_\Gamma)_W$.
By the Seifert--Van Kampen theorem, the kernel of the natural map $\pi_1(Z^1,W)\to\pi_1(Z,W)$ is normally generated by the boundary loops of the $2$--cells.
By construction, each of these boundary loops maps to the trivial automorphism.
Then the homomorphism descends to a homomorphism $\pi_1(Z,W)\to\Aut(A_\Gamma)_W$, which is necessarily surjective.
\end{proof}

The following proposition is the key to Theorem~\ref{th:raagstabpres}.
\begin{proposition}\label{pr:homotopeinZ}
Suppose $p$ is an edge loop in $Z$ based at $W$ that maps to the trivial automorphism.
Then $p$ can be homotoped relative to $W$ to an edge loop whose edge labels consist entirely of permutation automorphisms and inner automorphisms.
\end{proposition}
Since the proof of Proposition~\ref{pr:homotopeinZ} uses the structure of the proof of Theorem~\ref{th:fullfeaturedpeakreduction}, we postpone it to Section~\ref{se:peakreduction}.
This statement is all we need to prove the finite presentation result.

\begin{proof}[Proof of Theorem~\ref{th:raagstabpres}]
Suppose $W$ is a tuple of conjugacy classes in $A_\Gamma$.
First we find a minimal-length representative of the orbit of $W$ using Proposition~\ref{pr:insamewhorbit} and Lemma~\ref{le:checkminlength}.
Of course, replacing $W$ with an image of itself under an automorphism will not change the isomorphism type of $\Aut(A_\Gamma)_W$---it will only replace it with a corresponding conjugate of itself in $\Aut(A_\Gamma)$.
So we replace $W$ with a minimal representative of its orbit.

We construct the complex $Z$ with respect to $W$ as described above, and consider the map $\pi_1(Z,W)\to \Aut(A_\Gamma)_W$ from Lemma~\ref{le:surjhomZ}.
If this map is an isomorphism, then the Seifert--Van Kampen theorem implies that the stabilizer is finitely presented, since $Z$ is a finite complex.
To prove the theorem, it is enough to show that the map is injective, since it is already surjective by Lemma~\ref{le:surjhomZ}.

To show injectivity, we assume we have an edge loop $p$ based at $W$, such that the composition of the edge labels of $p$ yields the trivial homomorphism.
By Proposition~\ref{pr:homotopeinZ}, we assume we have homotoped $p$ to an edge loop whose edge labels are permutation automorphisms and inner automorphisms.
We use the $2$--cells of type (C6) to slide these inner automorphisms past the permutation automorphisms.
Then the inner automorphisms label loops at the base vertex $W$.
The multiplication table of the group $P$ is included in the relations that the (C3) cells bound, so we can eliminate all the permutation automorphisms from $p$ by homotoping across these cells.
Then $p$ reads off a product of inner automorphisms representing the trivial automorphism.
We can rewrite any inner automorphism in $\Omega$ as a product of inner automorphisms that are also classic Whitehead automorphisms by homotoping across (C1) cells.
The group of inner automorphisms of $A_\Gamma$ is isomorphic to another right-angled Artin group ($A_{\Gamma'}$, where $\Gamma'$ is $\Gamma$ with all the vertices representing central generators deleted).
Further, this isomorphism carries the inner classic Whitehead automorphisms to the standard generating set of the right-angled Artin group.
So any word in the inner classic Whitehead automorphisms that represents the trivial automorphism can be eliminated by applying commutation relations.
These commutation relations are given by $2$--cells in $Z$ (redundantly as (C3) or (C7) cells), so we can homotope $p$ to the trivial edge path at $W$.
\end{proof}

\section{Orbits of matrices}
\label{se:linearproblems}
In this section we prove Propositions~\ref{pr:matrixorbitalgorithm} and~\ref{pr:matrixstabpres}.
\subsection{Partly rational linear problems}
We fix  $n\geq 1$ and $k\geq 0$ and consider block-upper triangular matrices of the form
\[
\left(
\begin{array}{cc}
A & B \\ O & I
\end{array}\right),
\]
with $A$ in $\GL(n,\Z)$ and $B$ in $M_{n,k}(\Q)$.
This is the semidirect product $\GL(n,\Z)\ltimes M_{n,k}(\Q)$, where $\GL(n,\Z)$ acts on $M_{n,k}(\Q)$ on the left by multiplication.
We will vary the kinds of entries we wish to consider in the upper-right block, so we let
$G_{\Q}$ denote $\GL(n,\Z)\ltimes M_{n,k}(\Q)$, and for a positive integer $d$, we let $G_d$ denote $\GL(n,\Z)\ltimes M_{n,k}(\frac{1}{d}\Z)$.
We deliberately restrict the coefficients in the upper left-block to $\Z$, so that $G_d$ will always be a finite-index subgroup in $G_1$ (we discuss this more below).
A key tool in our discussion will be the following modified version of the Hermite normal form for integer row reduction.
For more on the Hermite normal form, see Cohen~\cite{Cohen}, section 2.4.2.
Note that Cohen describes Hermite normal form for column reduction, whereas we use Hermite normal form for row reduction.

\begin{definition}
Suppose $A$ is a matrix in $M_{n+k,m}(\Q)$.
Roughly speaking, $A$ is in \emph{$G_{\Q}$--normal form} if
\begin{itemize}
\item $\Q$-linear combinations of rows $n+1$ through $n+k$ have been added to rows $1$ through $n$ to reduce them as much as possible, and
\item there is a multiple of the block of rows $1$ through $n$ that is a matrix in Hermite normal form for integer row reduction.
\end{itemize}
Precisely, $A=(a_{ij})$ is in $G_{\Q}$--normal form if
\begin{itemize}
\item for each $j$, $1\leq j\leq m$, if there is a linear combination of rows $n+1$ through $n+k$ that is nonzero in column $j$ but is zero in all previous rows, then every entry in column $j$ from row $1$ through row $n$ is zero,
\item there is an increasing sequence $p_1,\dotsc,p_l$ of pivot column positions, for some $l$ with $0\leq l\leq n$, such that for each $i$ from $1$ through  $l$:
\begin{itemize}
\item the entry $a_{i,p_i}$ is positive,
\item the entries $a_{i,1},\dotsc,a_{i,p_i-1}$ in row $i$ preceding column $p_i$ are all $0$,
\item for each $i'$, $i'=1,\dotsc, i-1$, the entry $a_{i',p_i}$ satisfies $0\leq a_{i',p_i} < a_{i,p_i}$ (the entries above the pivot position are nonnegative and less than the pivot), and
\end{itemize}
\item rows from $l+1$ through $n$ contain only zero entries.
\end{itemize}
\end{definition}

We want to show that every matrix is equivalent to a unique one in this form, and to understand the stabilizer of a matrix in this form.
First we prove uniqueness.
\begin{lemma}\label{le:normalformunique}
Suppose $A$ and $B$ are matrices in $M_{n+k,m}(\Q)$ in $G_{\Q}$--normal form and $Q$ is in $G_{\Q}$ with $B=QA$.
Let $d$ be  the smallest positive integer with $Q\in G_d$.
Then $Q=Q_1Q_2$ with $Q_2\in \{I\}\ltimes M_{n,k}(\frac{1}{d}\Z)$ and $Q_1\in\GL(n,\Z)\ltimes \{O\}$, such that $Q_2A=A$ and such that
the first $l$ columns of $Q_2$ are the same as those of the identity matrix, where $l$ is the number of pivots in $A$.
In particular, $A=B$.
\end{lemma}

\begin{proof}
First of all, since $G_{\Q}$ cannot alter rows $n+1$ through $n+k$ by multiplication on the left, we see that the bottom $k$ rows of $B$ and $A$ must be identical.
In particular, for each position $j$, if there is a linear combination of the bottom $k$ rows in $A$ that is trivial in columns $1$ through $j-1$ but nontrivial in column $j$ then there is one for $B$ as well, and vice versa.
Then the set of columns that are forced to be trivial by the first condition in the definition of $G_{\Q}$--normal form is the same in both $A$ and $B$.
The matrix $Q$ can be factored as $Q_1Q_2$, where $Q_2\in \{I\}\ltimes M_{n,k}(\frac{1}{d}\Z)$ and $Q_1\in\GL(n,\Z)\ltimes \{O\}$, because $G_d$ is an internal semidirect product of these two subgroups.
Then $Q_2$ fixes $A$, because $Q_2$ can only change an entry in $A$ that is forced to be zero by the definition, and $Q_1$ cannot change the set of columns that only have zero entries in their first $n$ rows.

So we have $B=Q_1A$, with $Q_1\in\GL(n,\Z)\ltimes\{O\}$.
The uniqueness of Hermite normal form implies that $A=B$, since $Q_1$ only changes the top $n$ rows of $A$ and the top $n\times k$ blocks of $A$ and $B$ are already rational multiples of matrices in Hermite normal form.
However, our statement about the form of $Q_1$ is a little stronger, so we prove the lemma as stated.

Let $p_1,\dotsc,p_l$ be the pivots of $A$.
We perform induction on the hypothesis that for $i$ with $i\leq l$, columns $1$ through $p_{i+1}-1$ of $A$ and $B$ match and columns $1$ through $i$ of $Q_1$ are the same as those of identity matrix.
The hypothesis is true for $i=0$ since columns $1$ through $p_1-1$ of $A$ are trivial and the equation $B=Q_1A$ is only possible if the corresponding columns of $B$ are also trivial.
Now we fix an $i$ with $1\leq i\leq l$ and consider column $p_i$ of $A$ and $B$.
Certainly entries $i+1$ through $n$ of column $p_i$ of $B$ are zero, since $B$ is in $G_{\Q}$--normal form and column $p_i$ is left of the $(i+1)$st pivot column of $B$ (if it exists).
Consider $i'$ with $i<i'\leq n$.
By the inductive hypothesis, the first $i-1$ entries of the $i'$th row in $Q_1$ are zero.
We dot this row $i'$ with the column $p_i$ in $A$ to get an entry in $B$ that we know to be zero.
Since entries $i+1$ through $n$ of column $p_i$ in $A$ are zero, the only possibly nonzero term in this dot product is the product of the $i',i$ entry of $Q_1$ with the $i,p_i$ entry of $A$.
So entry $i',i$ of $Q_1$ is zero, and varying $i'$, we see that every below-diagonal entry in $Q_1$ in column $i$ is zero.
Next we note that if the diagonal entry $i,i$ of $Q_1$ were zero, then $Q_1$ would have determinant zero.
So this entry is nonzero, and therefore position $i,p_i$ in $B$ has a nonzero entry and is the pivot there.
Since the pivot entries are positive, the fact that $Q_1$ has determinant $\pm1$ implies that the $i,i$ entry of $Q_1$ is $1$.
Then position $i,p_i$ matches in $A$ and $B$.
If any above-diagonal entry in column $i$ of $Q_1$ is nonzero, then an entry in column $p_i$ of $B$ above the pivot will not be reduced modulo the pivot entry.
This then implies that all the columns before $p_{i+1}$ of $A$ and $B$ match, since these columns have only zero entries in positions $i+1$ through $n$.

The induction continues until we reach the last pivot position of $A$, showing that the first $l$ columns of $Q_1$ match those of the identity matrix.
Since rows $l+1$ through $n$ of $A$ are zero rows, this is enough to deduce that $B=A$.
\end{proof}

Now we show existence of matrices in normal form.
\begin{proposition}\label{pr:normalformexists}
Every matrix $A$ in $M_{n+k,m}(\Q)$ is associated to a matrix $B$ in $M_{n+k,m}(\Q)$ in $G_{\Q}$--normal form, with $A=QB$ for some $Q\in G_{\Q}$.
The matrix $B$ is unique.
\end{proposition}

\begin{proof}
We prove existence by supplying an algorithm; of course uniqueness is then the result of Lemma~\ref{le:normalformunique}.
The algorithm is a row reduction algorithm.
Multiplication on the left by elements of $G_{\Q}$ allows us to replace any of the top  $n$  rows of $A$ by itself plus a rational linear combination of the bottom $k$ rows, or to replace any row in the top $n$ by itself plus an integer linear combination of the other top $n$ rows, or to permute the top $n$ rows, or to multiply any of the top $n$ rows by $-1$.

The first part of the algorithm is to use the bottom $k$ rows of $A$ to simplify $A$ as much as possible; this is step 1 below.
The second part is to perform integer row reduction and reduce the entries above the pivots as much as possible.

\textbf{Step 1:}
We start by setting $j=1$; $j$ is the position of the column we are trying to simplify.
We consider the map $\Q^k\to \Q^{j-1}$ that sends a $k$-tuple of coefficients to the corresponding linear combination of the bottom $k$ rows of $A$, restricted to their first $j-1$ columns.
We find generators for the kernel of this map.
Each generator gives us a linear combination of the bottom $k$ rows of $A$ that is zero in its first $j-1$ entries; if some generator's linear combination is nonzero in the $j$th column, we add rational multiples of this linear combination to the top $n$ rows of $A$ to zero out their $j$th column entries.
Of course this leaves the previous columns unaffected.
(In the case that $j=1$, we simply check whether some row in the bottom $k$ rows has a nonzero entry in its first column, and if so,  we use its multiples to zero out the top $n$ entries of the first column.)
So we replace $A$ with an equivalent matrix with the first $n$ entries of column $j$ zeroed out if possible.
We  then replace $j$ with $j+1$ and repeat this step.
We do this until we have tried it for all columns of $A$.

After completing the previous step, we perform the procedure to turn the top $n\times m$ block of $A$ into a rational multiple of a matrix in Hermite normal form for integer row reduction. 
Although this is standard, we include it here for completeness.

\textbf{Step 2:}
If the top $n\times m$ block of $A$ is now the zero matrix, then $A$ is in $G_{\Q}$--normal form and we are done.
Otherwise we start the second part by setting $j$ to be the first column with a nonzero entry in its first $n$ rows.
We initialize our sequence of pivots by setting $l=0$, so that there are no pivots in the pivots sequence $p_1,\dotsc,p_l$.

\textbf{Step 3:}
By construction, $j>l$ and all entries in row $l+1$ through row $n$ in columns $1$ through $j-1$ are zero.
We look the entries of column $j$ from row $l+1$ through row $n$, and choose a row $i$ whose column-$j$ entry has minimal nonzero absolute value among these
(if all the entries were zero, we would increment $j$ and loop, but the choice of $j$ should prevent this).
We add integer multiples of row $i$ to each row from $l+1$ through $n$ in order to diminish the sum of the absolute values of the entries in column $j$.
We continue this until either another row's $j$th entry becomes smaller in absolute value than that of row $i$, or else row $i$ becomes the only entry from position $l+1$ through $n$ with a nonzero entry.
If another row's $j$th entry becomes smaller in absolute value than that of row $i$, then we replace $i$ with the position of that row and repeat this step.
If position $i$ becomes the only entry in column $j$ from position $l+1$ through $n$ that is nonzero, then we proceed to the next step.

\textbf{Step 4:}
The entry in position $i$ is the unique nonzero entry in column $j$ among rows from $l+1$ through $n$.
We permute rows $l+1$ through $n$ of $A$ so that this nonzero entry is now in row $l+1$.
We replace $l$ with $l+1$ and set the new pivot position $p_l$ to be $j$.
We replace row $l$ with its multiple by $\pm1$ to ensure that the pivot entry in position $l,p_l$ is positive.
We then add integer multiples of row $l$ to rows $1$ through $l-1$ to make sure that these entries are nonnegative and strictly less than the pivot entry.
Since all entries to the left of the pivot entry are zero, this does not affect the previous columns of the matrix.
We then set $j$ to be the next column with  a nonzero entry in rows $l+1$ through $n$ and return to step~3.
If no such column exists, we are done.
This finishes the algorithm.

It is clear that this procedure terminates.
The loop in step 1 repeats once for each column.
After step 1, the least common denominator of the entries of $A$ does not change; call this number $d$.
The loop in step 3 always terminates because we decrease the sum of the absolute values of the entries of column $j$ of $A$ by at least $1/d$ with each iteration.
The loop going back from step 4 to step 3 can be repeated at most $m$ times since it requires a new column each time.

It is also clear that the output of this procedure is a matrix in $G_{\Q}$--normal form.
The matrix coming out of step 1 satisfies the first property in the definition: if there were a way to use the bottom $k$ rows to zero out the top $n$ in the $j$th column without disturbing the previous columns, we would have used it already.
This is not disturbed by the remainder of the algorithm, which never changes a column with only zeros in its first $n$ rows to one with a nonzero entry there.
The output of the algorithm has its list of pivot columns, and by construction, the pivots satisfy the conditions in the definition.
Of course, we only stop producing pivots when all the remaining rows in the first $n$ are zero rows, so we satisfy the last condition in the definition.

Let $A$ denote the matrix input to the procedure and $B$ the output, which is in $G_{\Q}$--normal form.
Keeping track of the row moves performed in this algorithm and composing them gives us a matrix $Q$ in $G_{\Q}$ with $B=QA$.
\end{proof}

Now we turn our attention to stabilizers in $G_d$.  We need the following.
\begin{lemma}\label{le:sdpres}
Suppose the group $G$ acts on the group $H$, $G$ has the presentation $\genby{S_G | R_G}$, $H$ has the presentation $\genby{S_H|R_H}$, and the set $R_C$ consists of words $gh^{-1}g^{-1}w_{g,h}$ for all $g\in S_G$ and $h\in S_H$, where $w_{g,h}$ is a word in $S_H$ representing $ghg^{-1}$.
Then 
\[\genby{S_G\cup S_H | R_G\cup R_H\cup R_C}\]
is a presentation for the semidirect product $G\ltimes H$.
\end{lemma}
The proof is left as an exercise for the reader.

\begin{proposition}\label{pr:normalformQstabpres}
Suppose $d$ is a positive integer and $A$ is a matrix in $M_{n+k,m}(\frac{1}{d}\Z)$ in $G_{\Q}$--normal form.
Then there is an effective procedure to give a finite presentation for the stabilizer $(G_d)_A$.
\end{proposition}

\begin{proof}
First we pin down the stabilizer; then we will find a presentation for it.
Let $l$ be the number of pivot rows of $A$.
Suppose $Q\in G_d$ and $A=QA$.
By Lemma~\ref{le:normalformunique}, we have $Q=Q_1Q_2$ where $Q_2\in \{I\}\ltimes M_{n,k}(\frac{1}{d}\Z)$, $Q_1\in\GL(n,\Z)\ltimes \{O\}$, $Q_2A=A$ and the first $l$ columns of $Q_1$ are the same as those of the identity matrix.
This tell us that the nontrivial block of $Q_1$ is in $M_{l,n-l}(\Z)\rtimes \GL(n-l,\Z)$, in other words the nontrivial block of $Q_1$ is itself a block-upper-triangular matrix of the form
\[
\left(
\begin{array}{cc}
I & B \\ O & C
\end{array}
\right)
\]
where $B\in M_{l,n-l}(\Z)$, $C\in \GL(n-l,\Z)$, $I$ is the $l\times l$ identity matrix and $O$ is the $(n-l)\times l$ zero matrix.
The configuration of blocks implies that $\GL(n-l,\Z)$ acts on $M_{l,n-l}(\Z)$ on the right, as the semidirect product notation reflects.

The matrix $Q_2$ is in the stabilizer of $A$ in $\{I\}\ltimes M_{n,k}(\frac{1}{d}\Z)$.
Each row of $Q_2$ acts by adding a rational linear combination of rows $n+1$ through $n+k$ of $A$ to some row of $A$ in the top $n$, and being in the stabilizer means that each such rational linear combination has trivial value.
So each row of the upper-right $n\times k$ block of $Q_2$ is an element of the kernel of the group homomorphism $(\frac{1}{d}\Z)^k\to \Q^m$ that sends a $k$-tuple of coefficients to its corresponding linear combination of rows $n+1$ through $n+k$ of $A$.
Letting $K\subset (\frac{1}{d}\Z)^k$ denote this kernel, we see that the stabilizer of $A$ in $\{I\}\ltimes M_{n,k}(\frac{1}{d}\Z)$ is isomorphic to $K^n$.

This makes it easy to see that the stabilizer of $A$ in $G_d$ is
\[\big(M_{l,n-l}(\Z)\rtimes \GL(n-l,\Z)\big)\ltimes K^n,\]
since each element of this group stabilizes $A$ and any $Q$ stabilizing $A$ is certainly in this group by the above argument.

Since $K$ is the kernel of a map $(\frac{1}{d}\Z)^k\to \Q^m$, it is a finite-rank free abelian group and we can find a basis for $K$.
This in turn yields a basis for $K^n$, which is also free abelian.
Likewise $M_{l,n-l}(\Z)$ is a free abelian group with an obvious basis.
Of course, this means that $K^n$ and $M_{l,n-l}(\Z)$ have obvious finite presentations, where the generators are the given bases and the relations state that all pairs of basis elements commute.
The group $\GL(n-l,\Z)$ has a generating set given by transvections  (elementary matrices with a single nonzero off-diagonal entry of $1$) and inversions (matrices sending a single basis element to its inverse and fixing the others).
The finite  presentation for $\SL(n-l,\Z)$ from Milnor~\cite[Chapter~10]{Milnor}, can easily be modified to give a finite presentation for $\GL(n-l,\Z)$.
The conjugate of a generator of $M_{l,n-l}(\Z)$ by a generator of $\GL(n-l,\Z)$ can easily be written down as a product of generators of $M_{l,n-l}(\Z)$.
We can tabulate this data for all choices of pairs of generators.
Using Lemma~\ref{le:sdpres}, this data together with the presentations for $\GL(n-l,\Z)$ and $M_{l,n-l}(\Z)$ can be combined into a finite presentation for
$M_{l,n-l}(\Z)\rtimes \GL(n-l,\Z)$.
Finally, we can tabulate the action of generators of $M_{l,n-l}(\Z)\rtimes \GL(n-l,\Z)$ on generators of $K^n$ (again in terms of generators of $K^n$).
This data, together with the obvious presentation for $K^n$ and the presentation for $M_{l,n-l}(\Z)\rtimes \GL(n-l,\Z)$, can be combined to give a presentation for the stabilizer of $A$ in $G_d$, again using Lemma~\ref{le:sdpres}.
\end{proof}

\subsection{Integer linear problems}
Let $d$ be a fixed positive integer.
To exploit our rational results in the previous section, we use a crossed homomorphism to keep track of the cosets of $G_d$ in $G_1$.
Let $\rho\co G_d\to M_{n,k}(\Z/d\Z)$ be the following composition
\[G_d = \GL(n,\Z)\ltimes M_{n,k}(\frac{1}{d}\Z)\to M_{n,k}(\frac{1}{d}\Z) \to M_{n,k}(\Z) \to M_{n,k}(\Z/d\Z),\]
where the maps are the second coordinate projection from the semidirect product, then multiplication by $d$, then reduction modulo $d$.

\begin{lemma}\label{le:crohomcosets}
The map $\rho$ is a crossed homomorphism: for $A,B\in G_d$, we have
\[\rho(AB)=A\cdot \rho(B) + \rho(A),\]
where $G_d$ acts on $M_{n,k}(\Z/d\Z)$ via the projection $G_d\to \GL(n,\Z)$ and the standard left action of $\GL(n,\Z)$ on $M_{n,k}(\Z/d\Z)$.

The kernel of $\rho$ is $G_1$ and further, the set of preimages of elements of $M_{n,k}(\Z/d\Z)$ under $\rho$ is precisely the set of left cosets of $G_1$ in $G_d$. 
\end{lemma} 

\begin{proof}
We consider the definition of the product operation in $\GL(n,\Z)\ltimes M_{n,k}(\Z/d\Z)$:
\[(A,B)\cdot (C,D)=(AC,AD+B).\]
Of course this means that the projection $G_d\to M_{n,k}(\frac{1}{d}\Z)$ is a crossed homomorphism  with respect to the action via the canonical left action of $\GL(n,\Z)$.
Since the map $M_{n,k}(\frac{1}{d}\Z)\to M_{n,k}(\Z/d\Z)$ is equivariant with respect to this action, the map $\rho$ is a crossed homomorphism.

If $A\in G_d$ is in the kernel of $\rho$, this means that the entries in its upper-right block are divisible by $d$ after being multiplied by $d$, in other words that they are integers.
So the kernel of $\rho$ is $G_1$.
Now suppose $B\in M_{n,k}(\Z/d\Z)$.
We can pick representatives for the residue classes of the entries of $B$ to get an element $\tilde B$ in $M_{n,k}(\Z)$ mapping to $B$.
Then $\rho$ maps $(I,\frac{1}{d}\tilde B)$ in $G_d$ to $B$.
We claim that $\rho^{-1}(B)$ is the coset $(I,\frac{1}{d}\tilde B)\cdot G_1$.
If $(C,D)$ is in $G_1$, then  
\[\rho( (I,\frac{1}{d}\tilde B)(C,D))=\rho(\frac{1}{d}\tilde B)+\rho(D)=B+0.\]
This implies that the coset is a subset of the preimage.
On the other hand, if $\rho((C,D))=B$, then $D-\frac{1}{d}\tilde B$ is in $M_{n,k}(\Z)$, and $(C, D-\frac{1}{d}\tilde B)\in G_1$ with $(C,D)=(I,\frac{1}{d}\tilde B)(C,D-\frac{1}{d}\tilde B)$.
In other words, the preimage is a subset of the coset.
This shows that every preimage is a coset.

Now suppose that $(A,B)G_1$ is a coset.
Let $\overline B=\rho((A,B))$.
If $(C,D)\in G_1$, then
$\rho((A,B)(C,D))=B\cdot \rho(D)+\rho(B)=0+\overline B$, so this coset is a subset of this preimage.
If $(C,D)\in G_d$ and $\rho((C,D))=\overline B$, then $(C,D)=(A,B)(A^{-1}C, A^{-1}(D-B))$ with $(A^{-1}C,A^{-1}(D-B))\in G_1$, so this preimage is a subset of this coset.
So every coset is a preimage.
\end{proof}

Now we can prove our proposition on determining orbit membership under the action of $G_1$ on $M_{n+k,m}(\Z)$.
\begin{proof}[Proof of Proposition~\ref{pr:matrixorbitalgorithm}]
Let $A$ and $B$ be in $M_{n+k,m}(\Z)$; we wish to find a matrix $D\in G_1$ with $DA=B$, or show that no such matrix exists.
We start by computing the $G_{\Q}$--normal forms of $A$ and $B$, using the algorithm in Proposition~\ref{pr:normalformexists}.
If these normal forms are different, then $A$ and $B$ are in distinct $G_{\Q}$--orbits (by uniqueness of the normal form, Lemma~\ref{le:normalformunique}).
Certainly if $A$ and $B$ are in different $G_{\Q}$--orbits, then they are also in different $G_1$--orbits.
So we suppose that $A$ and $B$ have the same $G_{\Q}$--normal form $N\in M_{n+k,m}(\Q)$.
Suppose $Q,R\in G_{\Q}$ with $N=QA$ and $N=RB$.
Let $d$ be the least common denominator of the entries of $N$, $Q$ and $R$.
Let $S$ be the generating set from the presentation for the stabilizer of $N$ in $G_d$ from Proposition~\ref{pr:normalformQstabpres}.
Let $\Delta$ be the Schreier graph of $G_1$ in $G_d$ with respect to $S$.
\begin{claim*}
$A$ and $B$ are in the same $G_1$--orbit if and only if the vertices $QG_1$ and $RG_1$ are in the same connected component of $\Delta$.
\end{claim*}

First we suppose that $QG_1$ and $RG_1$ are in the same connected component of $\Delta$.
Let $C\in G_d$ be the composition of edge labels on an edge path from $QG_1$ to $RG_1$.
Then $CQG_1=RG_1$ by Lemma~\ref{le:travelSchreier}.
In particular, $R^{-1}CQ$ is in $G_1$, and $R^{-1}CQ$ sends $A$ to $B$.

Conversely, we suppose that there is $D$ in $G_1$ with $DA=B$.
Then $RDQ^{-1}$ fixes $N$.
By Proposition~\ref{pr:normalformQstabpres},  $RDQ^{-1}$ is a word in $S$.
Starting at $QG_1$, we form an edge path in $\Delta$ by following the edges labeled by this expression for $RDQ^{-1}$.
This is possible and unambiguous since $\Delta$, being a Schreier graph, has exactly one edge with each label entering and leaving each vertex.
Then by Lemma~\ref{le:travelSchreier}, the terminus of the path is $RDQ^{-1}QG_1=RG_1$.
This proves the claim.

The algorithm should then be clear at this point.
First we compute the $G_{\Q}$--normal forms of $A$ and $B$ and report that $A$ and $B$ are in different $G_1$--orbits if these $G_{\Q}$--normal forms differ.
If the normal forms are the same matrix $N$, we find matrices $Q$ and $R$ with $QA=N$ and $RB=N$ and find the lowest common denominator $d$ of the entries of $Q$, $R$ and $N$.
Then we find a generating set $S$ for the stabilizer of $N$ in $G_d$ and construct the Schreier graph $\Delta$ of $G_1$ in $G_d$ with respect to $S$.
This is possible since $\Delta$ is finite by Lemma~\ref{le:crohomcosets}.
The crossed homomorphism $\rho$ gives a convenient way to construct $\Delta$: $CDG_1=EG_1$ if and only if $\rho(CD)=\rho(E)$, if and only if $C\rho(D)+\rho(C)=\rho(E)$ for any $D,E\in G_d$ and $C\in S$.
Next in the algorithm, we check whether $QG_1$ and $RG_1$ are in the same connected component.
If not, we report that $A$ and $B$ are in different $G_1$--orbits.
If they  are in the same connected component, we take $C$ to be the composition of edge labels along a path from $QG_1$ to $RG_1$, and report that $R^{-1}CQ$ is a matrix in $G_1$ sending $A$ to $B$.
\end{proof}

Now we find a presentation for the stabilizer in $G_1$ of a matrix $A$ in $M_{n+k,m}(\Z)$.
\begin{proof}[Proof of Proposition~\ref{pr:matrixstabpres}]
Let $N$ be the $G_{\Q}$--normal form of $A$ and let $Q\in G_{\Q}$ be an element with $N=QA$, as found using Proposition~\ref{pr:normalformexists}.
Let $d$ be the least common denominator of the entries of $N$ and $Q$.
Let $S$ be a generating set for the stabilizer of $N$ in $G_d$, as given by Proposition~\ref{pr:normalformQstabpres}.
Let $\Delta$ be the Schreier graph of $G_1$ in $G_d$ with respect to $S$.
Let $S'$ be $\{ Q^{-1}CQ | C\in S\}$; note that $S'\subset (G_d)_A$.
Finally, let $\Delta'$ be the Schreier graph of $(G_1)_A$ in $(G_d)_A$ with respect to $S'$.
The proof of the proposition will follow from the claim:

\begin{claim*}
As a directed multigraph, $\Delta'$ is isomorphic to the connected component of $Q^{-1}G_1$ in $\Delta$, by an isomorphism that sends edge labels in $S'$ to their conjugates by $Q$ in $S$.
\end{claim*}

To prove the claim, we start by defining a map on vertices from $\Delta'$ to $\Delta$.
Let $B\in (G_d)_A$.
We send the vertex $B\cdot (G_1)_A$ of $\Delta'$ to the vertex $Q^{-1}BG_1$ of $\Delta$.
This map is well defined: if $C$ is also in $(G_d)_A$ with $B\cdot (G_1)_A=C\cdot (G_1)_A$, then $B^{-1}C\in (G_1)_A\subset G_1$ and therefore $Q^{-1}BG_1=Q^{-1}CG_1$.
Conversely, the map is injective: given $B,C\in (G_d)_A$ with $Q^{-1}BG_1=Q^{-1}CG_1$, we see that $B^{-1}C\in G_1$; since $(G_1)_A=(G_d)_A\cap G_1$, this means that $B^{-1}C\in (G_1)_A$ and therefore that $B\cdot (G_1)_A=C\cdot(G_1)_A$.
Suppose there is an edge labeled by $C$ in $S'$ from $B\cdot (G_1)_A$ to $B'\cdot (G_1)_A$ in $\Delta'$.
Then there is an edge labeled by $Q^{-1}CQ$ in $S'$ from $Q^{-1}BG_1$ to $Q^{-1}B'G_1$.
This is immediate from the definition of the Schreier graph.
The reverse implication also holds, so the map is an isomorphism of directed multigraphs and respects labels as described.

All that is left in the claim is to show that $\Delta'$ is connected.
Suppose $B\cdot (G_1)_A$ is a vertex of $\Delta'$.
Of course this means that $B\in (G_d)_A$.
Then $QBQ^{-1}\in (G_d)_N$, and by the definition of $S$, $QBQ^{-1}$ can be expressed as a product of elements of $S$.
Of course, this means that $B$ can be expressed as a product of elements of $S'$.
Since $\Delta'$ is a Schreier graph,
 we can trace out a unique edge path starting at $(G_1)_A$ using the labels from from the given expression for $B$ as a product of elements of $S'$.
By Lemma~\ref{le:travelSchreier}, the terminus of this path is $B\cdot (G_1)_A$.
Since $B$ was arbitrary, this means that $\Delta'$ is connected.  This proves the claim.

Now we use the claim to finish the proposition.
Let $Z$ be the presentation $2$--complex for $(G_d)_N$ using the finite presentation from Proposition~\ref{pr:normalformQstabpres}.
Of course, the generating set for this presentation is $S$.
The inner automorphism of $G_d$ given by conjugating by $Q^{-1}$ sends $(G_d)_N$ to $(G_d)_A$ and sends $S$ to $S'$.
Therefore we can also view $Z$ as a presentation $2$--complex for $(G_d)_A$ with generators $S'$.
Let $\widetilde Z$ be the cover of $Z$ with fundamental group $(G_1)_A$.
Since $\Delta'$ is a finite graph, $\widetilde Z$ is a finite-sheeted cover and is therefore a finite complex.
Then we use the Seifert--Van Kampen theorem to write down a finite presentation for the fundamental group of $\widetilde Z$, which is $(G_1)_A$.
Since every step in this construction can be done effectively, this is an algorithm to produce a finite presentation.
\end{proof}

\begin{example}
We consider a concrete example to illustrate the above algorithms.
It also happens that this is an example where the Schreier graph we describe is disconnected, and where a pair of matrices are in the same $G_d$--orbit but in different $G_1$--orbits.
Consider the following matrices:
\[
A=
\left(
\begin{array}{c}
1 \\ 0 \\ 2
\end{array}
\right)
,
\quad
N=
\left(
\begin{array}{c}
0 \\ 0 \\ 2
\end{array}
\right)
,\quad
Q=
\left(
\begin{array}{ccc}
1 & 0 & -\frac{1}{2} \\ 0 &  1 & 0\\ 0 & 0 & 1
\end{array}
\right).
\]
Then $A,N\in M_{3,1}(\Z)$ and $Q\in G_2=\GL(2,\Z)\ltimes M_{2,1}(\frac{1}{2}\Z)$.
We see that $A$ is not in $G_{\Q}$--normal form, but $N$ is, and $N=QA$.
The stabilizer of $N$ in $G_2$ is a copy of $\GL(2,\Z)$ generated by $S=\{a,b,c\}$ with
\[
a=
\left(
\begin{array}{ccc}
1 & 1 & 0 \\
0 & 1 & 0 \\
0 & 0 & 1 \\
\end{array}
\right),
\quad
b=
\left(
\begin{array}{ccc}
1 & 0 & 0 \\
1 & 1 & 0 \\
0 & 0 & 1 \\
\end{array}
\right),
\quad
c=
\left(
\begin{array}{ccc}
-1 & 0 & 0 \\
0 & 1 & 0 \\
0 & 0 & 1 \\
\end{array}
\right).\]
Using the fact that $\SL(2,\Z)$ is the amalgamated free product $(\Z/4\Z) *_{\Z/2\Z} (\Z/6\Z)$ (see for example, Serre~\cite[page~35]{Serre}) with the elements of order $4$ and $6$ given by $ba^{-1}b$ and $a^{-1}b$ respectively, it is straightforward to derive a presentation for $\GL(2,\Z)=(G_2)_N$ as follows:
\[
\genby{ a,b,c | c^2=1, (a^{-1}b)^3=(ba^{-1}b)^2, (a^{-1}b)^6=1, cac=a^{-1}, cbc=b^{-1}}.
\]

\begin{figure}
\labellist\small\hair 2.5pt
\pinlabel $\binom{0}{0}$ [l] at 32 138
\pinlabel $\binom{1}{0}$ [l] at 129 138
\pinlabel $\binom{1}{1}$ [l] at 129 39
\pinlabel $\binom{0}{1}$ [l] at 32 39
\pinlabel $c$ [l] at 39 9
\pinlabel $c$ [l] at 144 9
\pinlabel $c$ [l] at 148 118
\pinlabel $c$ [l] at 45 115
\pinlabel $a$ [r] at 0 124 
\pinlabel $a$ [r] at 100 124 
\pinlabel $a$ [l] at 70 44
\pinlabel $a$ [l] at 70 6
\pinlabel $b$ [r] at 16 142
\pinlabel $b$ [r] at 16 42
\pinlabel $b$ [r] at 111 75
\pinlabel $b$ [l] at 138 75
\endlabellist
\begin{center}
\includegraphics{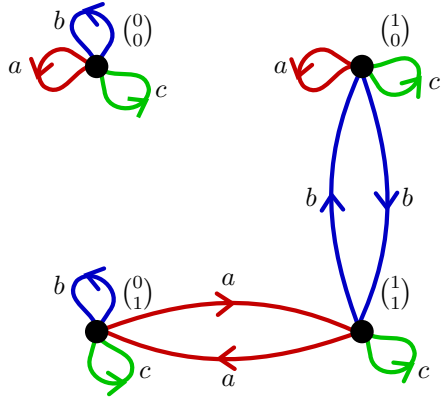}
\end{center}
\caption{The Schreier graph of $\GL(2,\Z)\ltimes M_{2,1}(\Z)$ in $\GL(2,\Z)\ltimes M_{2,1}(\frac{1}{2}\Z)$ with respect to $\{a,b,c\}$.
Each vertex is labeled by the image of its coset under $\rho$. 
The vertex $Q^{-1}G_1$ is labeled $\binom{1}{0}$.
}\label{fig:Schreiereg}
\end{figure}

As usual, $G_1=\GL(2,\Z)\ltimes M_{2,1}(\Z)$.
The Schreier graph $\Delta$ of $G_1$ in $G_2$ with respect to $S$ is displayed in Figure~\ref{fig:Schreiereg}.
By inspecting $\Delta$, we learn that the fundamental group of the connected component of $Q^{-1}G_1$, based at $Q^{-1}G_1$, is generated by the following seven elements:
\[\{a,c,b^2,bcb^{-1}, ba^2b^{-1}, baca^{-1}b^{-1}, baba^{-1}b^{-1}\}.\]
Therefore the stabilizer $(G_1)_A$ is generated by the conjugates of these seven elements by $Q^{-1}$:
\begin{multline*}
\Bigg\{
a,
c,
\left(
\begin{array}{ccc}
1 & 0 & 0 \\
2 & 1 & -1 \\
0 & 0 & 1 \\
\end{array}
\right),
\left(
\begin{array}{ccc}
-1 & 0 & 1 \\
-2 & 1 & 1 \\
0 & 0 & 1 \\
\end{array}
\right),
\left(
\begin{array}{ccc}
-1 & 2 & 1 \\
-2 & 3 & 1 \\
0 & 0 & 1 \\
\end{array}
\right),
\\
\left(
\begin{array}{ccc}
-3 & 2 & 2 \\
-4 & 3 & 2 \\
0 & 0 & 1 \\
\end{array}
\right),
\left(
\begin{array}{ccc}
3 & -1 & -1 \\
4 & -1 & -2 \\
0 & 0 & 1 \\
\end{array}
\right)
\Bigg\}.
\end{multline*}

There is a presentation $2$--complex for $(G_1)_A$ that has the connected component of $QG_1$ in $\Delta$ as its $1$--skeleton, and which is a three-sheeted cover of the presentation $2$--complex for $(G_2)_N$ corresponding to the presentation given above.
In particular, $(G_1)_A$ has a finite presentation with seven generators and fifteen relators.

There is another interesting observation we can make by looking at $\Delta$.
Even though $A$ and $N$ are both in $M_{3,1}(\Z)$ and $A$ and $N$ are in the same orbit under $G_2$, the vertices $G_1$ and $QG_1$ are in different connected components of $\Delta$ and therefore $A$ and $N$ are in different orbits under the action of $G_1$.
This shows that it is not enough simply to check whether matrices are in the same $G_\Q$--orbit.
\end{example}

\section{Peak reduction}\label{se:peakreduction}
\subsection{Preliminary notions}
Throughout this section we use the generalized Whitehead automorphisms $\Omega$ defined in the introduction.
The definition of peak here does not exactly match the definition used in Day~\cite{Day1}, but the current definition is more symmetric.
\begin{definition}
A \emph{peak} is a triple $(W,\alpha,\beta)$, where $W$ is a tuple of conjugacy classes and $\alpha$ and $\beta$ are automorphisms of $A_\Gamma$ and 
\[\abs{\alpha\cdot W}\leq \abs{W}\quad\text{and}\quad\abs{\beta\cdot W}\leq\abs{W},\]
with at least one of these inequalities being strict.
The \emph{height} of a peak is $\abs{W}$.
In this paper, the automorphisms $\alpha$ and $\beta$ in a peak are assumed to be in $\Omega$.

A \emph{lowering} of a peak $(W,\alpha,\beta)$ is a factorization
\[\beta\alpha^{-1}=\gamma_k\dotsc\gamma_1\]
of $\beta\alpha^{-1}$ by automorphisms $\gamma_1,\dotsc,\gamma_k$, such that all the lengths of the intermediate images of $\alpha\cdot W$ are strictly lower than that of $W$:
\[\abs{\gamma_i\dotsm\gamma_1\alpha\cdot W}<\abs{W}\]
for $i=1,\dotsc,k-1$.
In this paper, the automorphisms in a lowering factorization of a peak are always elements of $\Omega$.
\end{definition}

The goal of this section is to prove the following lemma:
\begin{mainlemma}\label{mle:peaklowering}
Suppose $(W,\alpha,\beta)$ is a peak and $\alpha$ and $\beta$ are generalized Whitehead automorphisms.
Then this peak can be lowered using a factorization of $\beta\alpha^{-1}$ by generalized Whitehead automorphisms.
\end{mainlemma}
Theorem~\ref{th:fullfeaturedpeakreduction} follows immediately from this lemma:
\begin{proof}[Proof of Theorem~\ref{th:fullfeaturedpeakreduction} from Main Lemma~\ref{mle:peaklowering}]
This is a standard argument; it appears in detail as the ``Proof of part (3) of Theorem~B" on pages~836--837 of~\cite{Day1}.
In short, we express $\alpha$ as a product of generalized Whitehead automorphisms (for example, as a product of Laurence--Servatius generators using Laurence's theorem) and then repeatedly alter the factorization using Main Lemma~\ref{mle:peaklowering}.
Whenever we see a peak of maximal height, we replace it with a peak-lowering factorization.
Each application of the Main Lemma reduces either the number of peaks of maximal height or reduces the maximum of the heights of the peaks in the factorization.
The procedure will terminate with a factorization of $\alpha$ as product of automorphisms from $\Omega$ such that no peaks appear in the sequence of lengths of the intermediate images of $W$.
Such a factorization satisfies the conclusion of the theorem.
\end{proof}

In our algorithm to lower peaks, we break into cases depending on the properties of the automorphisms $\alpha$ and $\beta$ in the given peak $(W,\alpha,\beta)$.
We often consider the support of automorphisms, defined in Definition~\ref{de:support}.
Using this, we can explain the structure of a generalized Whitehead automorphism.
\begin{lemma}\label{le:actoncomps}
Suppose $\alpha\in\whset{a}$ for some vertex $a$.
Then $\alpha$ is a product of dominated transvections and partial conjugations with multipliers in $[a]$ and inversions in elements of $[a]$.
In particular:
\begin{itemize}
\item if $b$ is adjacent to $a$ with $[b]\neq[a]$ and $\alpha$ does not fix $b$, then $a$ dominates $b$.
\item if $Y$ is a connected component of $\Gamma\setminus\st(a)$ with at least two elements, then either $Y^{\pm1}\cap\supp(\alpha)=\varnothing$ or $Y^{\pm1}\subset\supp(\alpha)$.
\end{itemize}
\end{lemma}
\begin{proof}
We know that $\whset{a}$ is isomorphic to a semidirect product $\GL(k,\Z)\ltimes M_{k,l}(\Z)$, where the short-range dominated transvections between elements of $[a]$ and inversions in $[a]$ generate the $\GL(k,\Z)$--factor and the other dominated transvections and partial conjugations with multipliers in $[a]$ generate the $M_{k,l}(\Z)$--factor.
This $M_{k,l}(\Z)$ is a free abelian group.
In particular, we can express $\alpha$ as $\alpha_2\alpha_1$, where $\alpha_1$ acts trivially on $\Gamma\setminus[a]$ (it is in the $\GL(k,\Z)$--factor) and $\alpha_2$ acts trivially on $[a]$ (it is in the $M_{k,l}(\Z)$--factor).

If $b$ is adjacent to $a$ with $[b]\neq[a]$, then the only possible generators of $\whset{a}$ that can alter $b$ are transvections, if they exist.
So if $\alpha$ does not fix $b$, there is a dominated transvection with multiplier $a$ acting on $b$, and therefore $a$ dominates $b$.

For $Y$ a connected component of $\Gamma\setminus\st(a)$, 
the only possible generators of $\whset{a}$ that do not leave $Y$ pointwise fixed are the partial conjugations that conjugate $Y$ by an element of $[a]$.
If one of these appears in a minimal-length factorization of $\alpha_2$, then $Y^{\pm1}\subset\supp(\alpha)$.
If none do, then $Y^{\pm1}\cap\supp(\alpha)=\varnothing$.
\end{proof}

We need some basic observations about the interactions of Whitehead automorphisms, which we gather together in the following lemma.
\begin{lemma}\label{le:basicobservations}
Suppose $a$ and $b$ are vertices of $\Gamma$ and $\alpha\in\whset{a}$ and $\beta\in\whset{b}$.
Then:
\begin{itemize}
\item If $a$ is adjacent to $b$ and $\alpha|_{[b]}$ is not the identity and $\beta|_{[a]}$ is not the identity, then $[a]=[b]$.
\item If $a$ is not adjacent to $b$ and $\alpha$ fixes each element of $[b]$ then $\alpha$ fixes each element of $\st(b)$.
\item If $a$ is not adjacent to $b$ then $\alpha$ and $\beta$ both fix all of $\st(a)\cap\st(b)$.
\item If $a$ is not adjacent to $b$ and $a$ dominates $b$, then $\beta$ is a long-range automorphism.
\item If $a$ is not adjacent to $b$ and $\alpha|_{[b]}$ is not the identity and $\supp(\alpha)\cap \supp(\beta)=\varnothing$ then $\beta$ is a long-range automorphism.
\end{itemize}
\end{lemma}
\begin{proof}
For the first item, suppose for contradiction that $[a]\neq[b]$.
We factor $\alpha$ and $\beta$ as products of Laurence generators.
Since $\alpha|_{[b]}$ is not the identity, some factor of $\alpha$ is a dominated transvection that replaces some $c\in[b]$ with $cd$ for some $d\in[a]$.
In particular, this means that $b$ dominates $a$.
Reversing the roles of $\alpha$ and $\beta$, we see that $a$ dominates $b$ as well.
Then by definition $[a]=[b]$, a contradiction.

For the second item, we write $\alpha=\alpha_1\alpha_2$ where $\alpha_1$ is a product of partial conjugations and $\alpha_2$ is a product of transvections and of inversions of elements of $[a]$.
Let $x\in\st(b)\setminus[b]$.
If $\alpha_1$ does not fix $x$, then $x$ is not adjacent to $a$ and $\alpha_1$ conjugates the entire connected component of $x$ in $\Gamma\setminus\st(a)$ by the same nontrivial element of $\genby{[a]}$.
Necessarily $b$ is in the same connected component and is therefore not fixed by $\alpha_1$.
But since $x$ is adjacent to $b$ but not to $a$, we know that $a$ does not dominate $b$, and therefore $\alpha_2$ fixes $b$.
This implies that $\alpha$ does not fix $b$, a contradiction.
If $\alpha_2$ does not fix $x$, then there is a transvection multiplying an element of $[a]$ onto $x$. 
Then $a$ dominates $x$.
But since $x$ is adjacent to $b$, this implies $a$ is adjacent to $b$, a contradiction.
So neither $\alpha_1$ nor $\alpha_2$ can alter $x$, which was arbitrary, and therefore $\alpha$ fixes $\st(b)$.

The third item is similar.
If $\alpha$ does not fix an element $x$ of $\st(a)\cap\st(b)$, then there must be a transvection multiplying an element of $[a]$ onto $x$, so that $a$ dominates $x$.
But if $a$ dominates $x$ and $x$ is adjacent to $b$, then $a$ is adjacent to $b$, a contradiction.
Similarly $\beta$ must fix every element of $\st(a)\cap\st(b)$.

To show the fourth item, we suppose that $a$ dominates $b$.
If $x$ is adjacent to $b$ and $\beta$ does not fix $x$, then there is a transvection factor of $\beta$ that multiplies an element of $[b]$ onto $x$.
Then $b$ dominates $x$.
Since $a$ dominates $b$, this implies that $a$ dominates $x$, and therefore that $a$ is adjacent to $b$, because $x$ is.
This is a contradiction, so therefore $\beta$ fixes every vertex adjacent to $b$.

Now we explain the fifth item.
If $\alpha$ does not fix $b$, then either $a$ dominates $b$, or the component $Y$ of $a$ in $\Gamma\setminus\st(b)$ has at least two elements and $\alpha$ conjugates every element of $Y$ by the same nontrivial word in $[a]$.
If $a$ dominates $b$, then we have already shown that $\beta$ must be long-range.
So we suppose $Y$ has at least two elements.
Suppose $x$ is adjacent to $b$.
If $\beta$ does not fix $x$, then $b$ dominates $x$ (as explained in the previous item).
If $x$ were adjacent to $a$, then $b$ would be adjacent to $a$, which it is not.
So $x$ is not adjacent to $a$, and therefore $x$ is an element of $Y$.
However, if $x$ is an element of $Y$, then $x$ is in $\supp(\beta)$, a contradiction since $\supp(\beta)\cap\supp(\alpha)=\varnothing$ and $x\in\supp(\alpha)$.
So $\beta$ fixes each vertex adjacent to $b$.
\end{proof}

The following observation is easy and needs no proof.
It greatly reduces the cases we need to consider.
\begin{lemma}
If $(W,\alpha,\beta)$ is a peak, then so is $(W,\beta,\alpha)$.
If one of these peaks can be lowered, then so can the other.
\end{lemma}

Now we outline the proof of Main Lemma~\ref{mle:peaklowering}, which is broken into lemmas filling out the rest of this section.
\begin{proof}[Proof of Main Lemma~\ref{mle:peaklowering}]
Since $(W,\alpha,\beta)$ is a peak, at least one of $\alpha$ or $\beta$ shortens $W$.
Since permutation automorphisms do not shorten $W$, at least one of $\alpha$ or $\beta$ is not a permutation automorphism.
Since we may swap $\alpha$ and $\beta$, we assume that $\alpha$ is not a permutation, so $\alpha\in\whset{a}$ for some vertex $a$ of $\Gamma$.
If $\beta$ is a permutation automorphism, then conjugate $\alpha$ by $\beta$ to lower the peak as described in Lemma~\ref{le:permlower}.

So we assume that neither $\alpha$ nor $\beta$ is a permutation automorphism; then $\beta\in\whset{b}$ for some vertex $b$ in $\Gamma$.
If $[a]=[b]$, then we lower the peak by replacing the length-two factorization $\beta\alpha^{-1}$ with the length-one factorization given by its product, as explained in Lemma~\ref{le:samemultset}.

Next we suppose that $a$ and $b$ are adjacent to each other in $\Gamma$, but that $[a]\neq[b]$.
By Lemma~\ref{le:basicobservations}, this implies that $\alpha|_{[b]}$ is the identity or $\beta|_{[a]}$ is the identity.
We suppose that $\alpha|_{[b]}$ is the identity; then Lemma~\ref{le:adjsteinberg} shows how to lower the peak.
This uses a Steinberg relation explained in Lemma~\ref{le:steinbergrel}.

So we assume that $a$ and $b$ are not adjacent in $\Gamma$.
If $a$ dominates $b$ and $b$ dominates $a$, then by Lemma~\ref{le:basicobservations}, both $\alpha$ and $\beta$ are long-range automorphisms.
In this case, we can lower the peak (in fact, fully reduce it) using Theorem~\ref{th:longrangepeakreduction}.
So we suppose that $b$ dominates $a$ but $a$ does not dominate $b$.
Again by Lemma~\ref{le:basicobservations}, $\alpha$ is a long-range automorphism.
Proposition~\ref{pr:nonadjacentasymmetric} explains how to lower such a peak.
Finally, we suppose that $b$ does not dominate $a$ and $a$ does not dominate $b$.
Then Proposition~\ref{pr:nonadjacentnodom} shows that the peak can be lowered.
\end{proof}

\subsection{Change of length under generalized Whitehead automorphisms}
To show that the algorithm for lowering peaks works as desired, we need to show that certain factorizations really are peak lowering.
To do this we need a good understanding of the effect that generalized Whitehead automorphisms have on lengths of tuples of cyclic words.
Next we prove some results about this.

We recall the action of $\whset{a}$ on syllables with respect to $a$ given in Definition~\ref{de:syllableaction}.
\begin{lemma}\label{le:syllablesandlengthchange}
Suppose $W$ is a tuple of cyclic words and $T$ is a syllable decomposition of $W$ with respect to $a$ and $\alpha\in\whset{a}$.
Then
\[\abs{\alpha\cdot W}-\abs{W}=\abs{\alpha\cdot T}-\abs{T}.\]
In particular, we can compute the difference $\abs{\alpha\cdot W}-\abs{W}$ by computing the differences $\abs{\alpha\cdot t}-\abs{t}$ over all syllables $t$ of $T$ and summing.
\end{lemma}

\begin{proof}
This is an easy corollary of Lemma~\ref{le:decompositionequivariance}; the details are left to the reader.
\end{proof}

In the following, when we talk about an endpoint of a syllable $cud$ being in the support of an automorphism, we are only talking about whether $c$ and $d^{-1}$ are in the support, but not about $c^{-1}$ or $d$.
This is because the  action of automorphisms on syllables has nothing to do with the action on the far sides of the endpoints.
The following proposition is used to show that factorizations coming from the relations in Lemma~\ref{le:steinbergrel} can lower peaks.
\begin{proposition}\label{pr:steinberglengthchange}
Let $a$ and $b$ be vertices of $\Gamma$ with $[a]\neq[b]$.
Suppose $\alpha\in\whset{a}$ and $\beta\in\whset{b}$ satisfy the hypotheses of Lemma~\ref{le:steinbergrel} ($\alpha$ fixes $[b]$, and either $a$ is adjacent to $b$, or $\supp(\alpha)\cap\supp(\beta)=\varnothing$ and $\beta$ fixes $[a]$).
Let $\gamma=\alpha\beta\alpha^{-1}$, which is in $\whset{b}$ by Lemma~\ref{le:steinbergrel}.
Suppose $W$ is a tuple of conjugacy classes.
Then
\[\abs{W}-\abs{\beta\cdot W}=\abs{\alpha\cdot W}-\abs{\alpha\beta\cdot W}.\]
Further, if $(W,\alpha,\beta)$ is a peak, then $\abs{\gamma\alpha\cdot W} < \abs{W}$.
\end{proposition}

Since $\gamma\alpha=\alpha\beta$, we can also express $\alpha\beta\cdot W$ as $\gamma\alpha\cdot W$ in the equation above.

\begin{proof}[Proof of Proposition~\ref{pr:steinberglengthchange}]
Let $T=(t_1,\dotsc,t_N)$ be a syllable decomposition of $W$ with respect to $[b]$.
We form a syllable decomposition $T'$ of $\alpha\cdot W$ with respect to $[b]$ by the following process:
first we form the representative of $W$ associated to $T$, then we apply $\alpha$ to this representative by inserting elements from $[a]$ where appropriate, then we delete inverse pairs of elements of $[a]$ that commute with all intervening letters, and then we split the resulting representative into syllables.
We note that no collapsing can happen in this process: if $c\notin[a]$ cancels with its inverse after we apply $\alpha$, then it could have canceled before applying $\alpha$, thus contradicting the stipulation from the definition of syllable decomposition that the associated representative be graphically reduced.
We prove the proposition by defining a relation $f$ between the syllables in $T$ and the syllables $T'$.
This $f$ will be an injective partial function.

Consider syllable $t_i$ from $T$.  
Either $t_i$ is a linear syllable $t_i=c_iv_i u_i d_i$  or a cyclic syllable $t_i=v_i u_i$,  with $v_i\in\genby{[b]}$, $u_i\in\genby{\st(b)\setminus[b]}$ and $c_i,d_i\in(\Gamma\setminus\st(b))^{\pm1}$.
We consider what happens to $t_i$ in passing from $W$ to $\alpha\cdot W$.
The hypotheses of Lemma~\ref{le:steinbergrel} give us two cases for $\alpha$ and $\beta$.
In the first case, $a$ is adjacent to $b$, so acting by $\alpha$ cannot alter syllable boundaries, and we define $f$ to be the map that sends each syllable in $T$ to the corresponding syllable with the same boundaries in $T'$.

In the other case, $a$ is not adjacent to $b$ and $\supp(\alpha)\cap\supp(\beta)=\varnothing$.
So by Lemma~\ref{le:basicobservations},  $\alpha$ fixes $v_iu_i$.
In this case, $\beta|_{[a]}$ is the identity.
Also, if $c_i$ or $d_i^{-1}$ is in $\supp(\beta)$, it is not in $[a]$ and is not in $\supp(\alpha)$, so action by $\alpha$ cannot cancel it away or insert an element of $[a]$ between it and $v_iu_i$.
So in the second case, if $c_i$ or $d_i^{-1}$ or any element of $u_i$ is in $\supp(\beta)$ then all of these land in the same syllable of $T'$; we declare $f$ to send $t_i$ to this syllable.

In the second 
case, if $c_i$ and $d_i^{-1}$ are not in $\supp(\beta)$ and $u_i$ does not contain any elements of $\supp(\beta)$, then we leave $t_i$ out of the domain of $f$.
This completes the definition of $f$.
We note that $f$ is injective.
This is obvious in the first case; in the second 
case, if two or more syllables merge to form the syllable $t'_j$, then at most one of them can be in the domain of $f$.
This is because $\supp(\alpha)\cap\supp(\beta)=\varnothing$ and because elements of $\st(a)\cap\st(b)$ cannot be in $\supp(\beta)$: to break a syllable boundary, we must have the far endpoint being in $\supp(\alpha)$ and everything in $\st(b)$ on one side of the boundary in $\st(b)\cap\st(a)$.

\begin{claim*}
If $t_i$ is in the domain of $f$, then $\abs{t_i}-\abs{\beta\cdot t_i}=\abs{f(t_i)}-\abs{\gamma\cdot f(t_i)}$.
\end{claim*}

First we consider the case where $a$ is adjacent to $b$.
Then $f(t_i)$ differs from $t_i$ in that some elements of $[a]$ have been added and deleted from $u_i$.
We use $u'_i$ to denote the part of $f(t_i)$ that is a word from $\st(b)\setminus[b]$.
The syllable $\beta\cdot t_i$ differs from $t_i$ in that elements of $[b]$ have been added and deleted from $v_i$.
Let $v'_i$ denote the part of $\beta\cdot t_i$ that is a word from $[b]$.
Since $\gamma\alpha=\alpha\beta$, the syllable $\gamma\cdot f(t_i)$ can be computed by applying $\alpha$ to the syllable $\beta\cdot t_i$ (in the same manner that we obtained $f(t_i)$ by applying $\alpha$ to $t_i$).
Since the elements of $\supp(\alpha)$ in $t_i$ are in $c_i$, $d_i^{-1}$ and $u_i$ (since $\alpha$ fixes $[b]$), these are unchanged in $\beta\cdot t_i$, and therefore $\gamma\cdot f(t_i)$ differs from $\beta\cdot t_i$ in that $u_i$ is replaced by $u'_i$ (the same $u'_i$ as above).
So the $\st(b)$-parts of $t_i$, $\beta\cdot t_i$, $f(t_i)$ and $\gamma\cdot f(t_i)$ are $u_iv_i$, $u_iv'_i$, $u'_iv_i$ and $u'_iv'_i$ respectively.
Since these are reduced products, this proves the claim in this case.

Next we consider the case where $a$ is not adjacent to $b$, the supports of $\alpha$ and $\beta$ are disjoint, and $\beta$ fixes $[a]$.
In this case, $\gamma=\beta$.
The elements of $f(t_i)$ in $\supp(\beta)$ are exactly the same as those in $t_i$, since, as noted above, any elements merged in from other syllables cannot contain anything from $\supp(\beta)$.
Similarly, parts merged in from other syllables cannot contain elements of $[b]$.
By Lemma~\ref{le:basicobservations}, $\alpha$ fixes $v_iu_i$.
Of course, $\beta\cdot t_i$ differs from $t_i$ in that $v_i$ is replaced by $v'_i$, where $v'_i$ is determined by $v_i$ and the elements of $t_i$ in $\supp(\beta)$.
Since $f(t_i)$ contains the same $[b]$-part $v_i$ and the same elements of $\supp(\beta)$, it follows that $\beta\cdot t_i$ will be the same as $f(t_i)$, but with $v_i$ replaced by the same $v'_i$ just mentioned.
In particular, the claim follows in this case.

\begin{claim*}
If $t_i$ is not in the domain of $f$, then $t_i=\beta\cdot t_i$, and if $t'_j$ is not in the range of $f$, then $t'_j=\gamma\cdot t'_j$.
\end{claim*}

In the first case, $f$ is a bijective total function and there is nothing to show.
In the second case, the statement about the domain of $f$ is obvious: if $t_i$ is not in the domain, then it contains nothing from $\supp(\beta)$, and therefore $\beta$ does nothing to it.
In the second case, if $t'_j$ is some syllable in $T'$ not in the range of $f$, then no part of it is in $\supp(\beta)$.
This is because $\alpha$ did not create new elements of $\supp(\beta)$ in acting on $W$, and if some $t_i$ has some elements of $\supp(\beta)$, all these elements of $\supp(\beta)$ end up in its correspondent under $f$.
The action of $\gamma$ is to place cancelling copies of $v$ and $v^{-1}$ in the syllable, proving the claim.

Now we finish the proof of the proposition.
Using Lemma~\ref{le:syllablesandlengthchange}, we have shown that
\[\abs{W}-\abs{\beta\cdot W}=\abs{\alpha\cdot W}-\abs{\gamma\alpha\cdot W},\]
since we have matched up all the syllables that change length on one side with the syllables that change length on the other side, and shown that they change length by the same amount.
If we suppose that $(W,\alpha,\beta)$ is a peak, then by definition, we know that $2\abs{W}<\abs{\beta\cdot W}+\abs{\alpha\cdot W}$ (this is from summing the two inequalities in the definition of a peak and using the stipulation that one of them is strict).
Combining these two statements, we obtain
 $\abs{W}>\abs{\gamma\alpha\cdot W}$.
\end{proof}

For the next proposition, we recall the technique from Day~\cite{Day1} for computing the change in length of classic long-range Whitehead automorphisms.
We defined classic Whitehead automorphisms in Section~\ref{ss:relations}.
Fix a tuple of conjugacy classes $W$.
For $c$ in a vertex of $\Gamma$ and $A$ and $B$ subsets of $X^{\pm1}$, we use the notation
$\pcount{A}{B}{W,c}$ to denote the number of instances of subwords of the forms $due^{-1}$ or $eud^{-1}$ in a graphically reduced representative for $W$, where $d\in A\setminus\st(c)$, $e\in B\setminus\st(c)$, and $u$ is a word in $\st(c)$.
This number is nonnegative and does not depend on the choices involved.
As a function of $A$ and $B$, it is additive over disjoint sets in both inputs.
Generally we will work with a specific $W$ and suppress $W$ from the notation.
Let $\gamma$ be a long-range classic Whitehead automorphism with multiplier $c$ and write $C=\supp(\gamma)+c$.
We write $c$ for $\{c\}$, ``-" for differences of sets, ``$+$" for disjoint unions and $C'$ for $X^{\pm1}\setminus C$ in the following computations.
This bracket has the following useful property (see Day~\cite[Lemma~3.16, Lemma~3.17]{Day1}):
\[\abs{W}-\abs{\gamma\cdot W}=\pcount{c}{C-c}{c}-\pcount{C'}{C-c}{c}.\]
Since $\pcount{c}{c}{c}=0$ ($W$ is graphically reduced), we can rewrite this as
\[\abs{W}-\abs{\gamma\cdot W}=\pcount{c}{X^{\pm1}}{c}-\pcount{C'}{C}{c}.\]

We use this bracket to detect sets of vertices for constructing Whitehead automorphisms for factorizations used in the peak reduction algorithm.
The next result helps us do this.
\begin{proposition}\label{pr:shorterfactors}
Suppose $\alpha$ and $\beta$ are classic long-range Whitehead automorphisms with multipliers $a$ and $b$ respectively.
Suppose that $a$ is not adjacent to $b$ and $b$ dominates $a$.
Finally, we suppose that $\alpha(b)=b$ and $a$ is in $\supp(\beta)$.
Let $W$ be a tuple of conjugacy classes.
Let $A$ denote $\supp(\alpha)+a$ and let $B$ denote $\supp(\beta)+b$.
Then:
\begin{itemize}
\item there is a classic long-range Whitehead automorphism $\beta_1$ with multiplier $b^{-1}$ and support $\supp(\beta_1)=A'\cap B'-b^{-1}$;
\item there is a classic long-range Whitehead automorphism $\alpha_1$ with multiplier $a$ and support $\supp(\alpha_1)=A\cap B-a$; and
\item we have the following inequality for the changes in length of $W$ under these automorphisms:
\[(\abs{W}-\abs{\beta_1\cdot W})+(\abs{W}-\abs{\alpha_1\cdot W})\geq (\abs{W}-\abs{\beta\cdot W})+(\abs{W}-\abs{\alpha\cdot W}).\]
\end{itemize}
\end{proposition}

\begin{proof}
First we explain why the automorphisms $\beta_1$ and $\alpha_1$ exist.
Since $b$ dominates $a$, $b$ dominates every vertex that $a$ dominates, and if $Y$ is a connected component of $\Gamma\setminus\st(a)$, $Y$ is a union of vertices adjacent to $b$, vertices dominated by $b$, and connected components of $\Gamma\setminus\st(b)$.
Therefore $\beta_1$ can be expressed as  a product of Laurence generators and is a well defined automorphism.

If $a$ dominates $b$, then $a$ and $b$ dominate the same vertices non-adjacently and $\Gamma\setminus\st(a)$ and $\Gamma\setminus\st(b)$ have the same connected components with two or more vertices.
Then $A\cap B-a$ is a union of vertices that $a$ dominates non-adjacently and components of $\Gamma\setminus\st(a)$ with two or more vertices, and therefore $\alpha_1$ can be expressed as a product of Laurence generators with multiplier $a$ and is well defined.
Now suppose that $a$ does not dominate $b$.
Then since $\alpha$ fixes $b$, $\alpha$ must fix the entire connected component of $b$ in $\Gamma\setminus\st(a)$. 
If $c$ is a vertex in $A\cap B$ and $a$ does not dominate $c$, then the entire connected component $Y$ of $c$ in $\Gamma\setminus\st(a)$ is in $A$.
If $b$ dominates $c$, then there is a path from $c$ to $b$ outside of $\st(a)$, and therefore $b$ and $c$ are both in $Y$.
However, this is a contradiction, since $\alpha$ must conjugate every element of $Y$ by $a$.
So $b$ does not dominate $c$, which means that the entire connected component $Z$ of $c$ in $\Gamma\setminus\st(b)$ is in $B$.
Since $Y\subset Z$ (since $b$ dominates $a$), this means that $Y\subset A\cap B$.
In particular, $\alpha_1$ can be expressed as a product of Laurence generators and is well defined.

Now we consider lengths of images of $W$.
From the comments preceding the statement of Proposition~\ref{pr:shorterfactors}, we know
\[\abs{W}-\abs{\alpha\cdot W}=\pcount{a}{X^{\pm1}}{a}-\pcount{A'}{A}{a},\] 
\[\abs{W}-\abs{\beta\cdot W}=\pcount{b}{X^{\pm1}}{b}-\pcount{B'}{B}{b},\] 
and
\[\abs{W}-\abs{\beta_1\cdot W}=\pcount{b^{-1}}{X^{\pm1}}{b}-\pcount{A'\cap B'}{(A'\cap B')'}{b}.\]
We expand:
\[
\begin{split}
\pcount{A'\cap B'}{(A'\cap B')'}{b}=&\pcount{A'\cap B'}{A\cap B}{b}+\pcount{A'\cap B'}{A'\cap B}{b}\\
&+\pcount{A'\cap B'}{A\cap B'}{b}
\end{split}
\]
and:
\[
\begin{split}
\pcount{B'}{B}{b}=&\pcount{A\cap B'}{A\cap B}{b}+\pcount{A\cap B'}{A'\cap B}{b}+\pcount{A'\cap B'}{A\cap B}{b}\\&+\pcount{A'\cap B'}{A'\cap B}{b}.
\end{split}
\]
In particular, we see
\[
\begin{split}
(\abs{W}-\abs{\beta_1\cdot W})-(\abs{W}-\abs{\beta\cdot W})=&\pcount{A\cap B}{A\cap B'}{b}-\pcount{A'\cap B'}{A\cap B'}{b}\\
&+\pcount{A\cap B'}{A'\cap B}{b}.
\end{split}
\]
Let $C_W$ be the number of subwords of our graphically reduced representative for $W$ of the form $auc^{-1}$, where $c$ is in $A-a$ and $u$ is in $\genby{\st(b)}$ but not in $\genby{\st(a)}$ (so $u$ contains some letter not commuting with $a$).
Then 
\[\pcount{A\cap B}{A\cap B'}{b}=\pcount{A\cap B}{A\cap B'}{a}+C_W.\]
On the other hand, if $cud^{-1}$ is counted by $\pcount{A'\cap B'}{A\cap B'}{b}$, meaning that $c\in A'\cap B'$, $u\in\genby{\st(b)}$ and $d\in A-a$, then either $u\in\genby{\st(a)}$ and $cud^{-1}$ is counted by $\pcount{A'\cap B'}{A\cap B'}{a}$, or else $u=u_1eu_2$ with $u_1\in\genby{\st(a)}$ and $e\in\st(b)\setminus\st(a)$.
In the latter case, $cu_1e$ is counted by  $\pcount{A'\cap B'}{A\cap B'}{a}$, since $e$ is in $A'\cap B'$ ($e$ is fixed by $\beta$ because $\beta$ is long range and by $\alpha$ because otherwise $\alpha$ would not fix $b$).
So each subword counted by $\pcount{A'\cap B'}{A\cap B'}{b}$ is counted exactly once by $\pcount{A'\cap B'}{A\cap B'}{a}$, so 
\[\pcount{A'\cap B'}{A\cap B'}{b}=\pcount{A'\cap B'}{A\cap B'}{a}.\]
Then
\[\begin{split}
(\abs{W}-\abs{\beta_1\cdot W})-(\abs{W}-\abs{\beta\cdot W})=&\pcount{A\cap B}{A\cap B'}{a}-\pcount{A'\cap B'}{A\cap B'}{a}\\&+C_W+\pcount{A\cap B'}{A'\cap B}{b}.
\end{split}\]
By an argument parallel to the one for $\beta_1$, we deduce
\[\begin{split}
(\abs{W}-\abs{\alpha_1\cdot W})-(\abs{W}-\abs{\alpha\cdot W})=&\pcount{A\cap B'}{A'\cap B'}{a}-\pcount{A\cap B}{A\cap B'}{a}\\&+\pcount{A'\cap B}{A\cap B'}{a}.
\end{split}\]
Summing the two equations, we see that
\[(\abs{W}-\abs{\beta_1\cdot W})-(\abs{W}-\abs{\beta\cdot W})+(\abs{W}-\abs{\alpha_1\cdot W})-(\abs{W}-\abs{\alpha\cdot W})\]
is 
\[C_W+\pcount{A\cap B'}{A'\cap B}{b}+\pcount{A\cap B'}{A'\cap B}{a},\]
a nonnegative number.
\end{proof}

In the following, when we write $\abso{\alpha}$, we mean the minimum length of a representative of the class of $\alpha$ in $\Out(A_\Gamma)$ as a product of Laurence generators.
\begin{corollary}\label{co:shorterfactors}
Suppose $(W,\alpha,\beta)$ is a peak with $\alpha\in \whset{a}$, $\beta\in\whset{b}$ classic long-range Whitehead automorphisms, with $a$ not adjacent to $b$, and $b$ dominating $a$, but $a$ not dominating $b$.
Suppose $a\in\supp(\beta)$ and $\alpha(b)=b$.
Then either
\begin{itemize}
\item there is a classic long-range Whitehead automorphism $\beta_1$ with multiplier $b^{-1}$, with $\abso{\beta_1^{-1}\beta}<\abso{\beta}$ as a product of Laurence generators, with $\beta_1(a)$ equal to $ba$ or $a$, and with $\abs{\beta_1\cdot W}<\abs{W}$, or
\item there is a classic long-range Whitehead automorphism $\alpha_1$ with multiplier $a$, 
with $\supp(\alpha_1)\subset\supp(\beta)$, and with $\abs{\alpha_1\cdot W}<\abs{W}$.
\end{itemize}
Further, if $\supp(\alpha)\cap\supp(\beta)=\varnothing$, then the first option holds.
\end{corollary}

\begin{proof}
We apply Proposition~\ref{pr:shorterfactors}.
Using the definition of  a peak,
\[(\abs{W}-\abs{\beta_1\cdot W})+(\abs{W}-\abs{\alpha_1\cdot W})\geq (\abs{W}-\abs{\beta\cdot W})+(\abs{W}-\abs{\alpha\cdot W}) > 0.\]
So one of $\alpha_1$ or $\beta_1$ strictly shortens $W$.
Since $\alpha_1$ and $\beta_1$ satisfy the other properties in the lemma, this proves the main statement of the lemma.
If $\supp(\alpha)\cap\supp(\beta)=\varnothing$, then $\alpha_1$ is defined to be the trivial automorphism and cannot shorten $W$.
Therefore in this case, it is $\beta_1$ that shortens $W$.
\end{proof}

\subsection{The algorithm: basic cases}\label{ss:algbasic}

Now we begin considering cases for lowering peaks.
For the rest of this section, we suppose we have peak $(W,\alpha,\beta)$ with $\alpha$ and $\beta$ generalized Whitehead automorphisms.
\begin{lemma}\label{le:permlower}
If $\alpha$ is a permutation automorphism, then the peak $(W,\alpha,\beta)$ can be lowered.
\end{lemma}
\begin{proof}
Since $\alpha$ is a permutation automorphism, $\alpha\cdot W$ is the same length as $W$.
Then the definition of a peak demands that $\abs{\beta\cdot W}<\abs{W}$, and therefore $\beta\in\whset{b}$ for some vertex $b$ of $\Gamma$.
Then $\beta\alpha^{-1}=\alpha^{-1}\cdot \alpha\beta\alpha^{-1}$ is a peak-lowering factorization: $\alpha\beta\alpha^{-1}$ is in $\whset{\alpha(b)}$ and $\abs{\alpha\beta\cdot W}=\abs{\beta\cdot W}<\abs{W}$ (again because $\alpha$ is a permutation automorphism). 
\end{proof}

\begin{lemma}\label{le:samemultset}
If $\alpha$ and $\beta$ are both in $\whset{a}$ for $a$  vertex of $\Gamma$, then $(W,\alpha,\beta)$ can be lowered.
\end{lemma}
\begin{proof}
Set $\gamma=\beta\alpha^{-1}$.
Then $\beta\alpha^{-1}=\gamma$ is a peak-lowering factorization: $\gamma\in\whset{a}$ and the condition on intermediate lengths of images of $W$ is vacuous for factorizations of length one.
\end{proof}

For the rest of this section we assume that $(W,\alpha,\beta)$ is a peak with $\alpha\in\whset{a}$ and $\beta\in\whset{b}$ for distinct vertices $a$ and $b$ in $\Gamma$.
\begin{lemma}\label{le:adjsteinberg}
Suppose $a$ is adjacent to $b$ and $\alpha_{[b]}$ is the identity.
Then the peak $(W,\alpha,\beta)$ can be lowered.
\end{lemma}
\begin{proof}
By Lemma~\ref{le:steinbergrel}, we know that the element $\gamma=\alpha\beta\alpha^{-1}$ is in $\whset{b}$.
We use the factorization $\beta\alpha^{-1}=\alpha^{-1}\gamma$.
To show that this is peak-lowering, it is enough to show that $\abs{\gamma\alpha\cdot W}<\abs{W}$.
However, we have already done this in Proposition~\ref{pr:steinberglengthchange}
\end{proof}

\begin{lemma}\label{le:nonadjsteinbergidentity}
Suppose $a$ is not adjacent to $b$, $\supp(\alpha)\cap\supp(\beta)=\varnothing$ and $\alpha|_{[b]}$ and $\beta|_{[a]}$ are both identity maps.
Then the peak $(W,\alpha,\beta)$ can be lowered.
\end{lemma}
\begin{proof}
Again we use $\gamma=\alpha\beta\alpha^{-1}$, which is in $\whset{b}$ by Lemma~\ref{le:steinbergrel}.
Our hypotheses imply that $\alpha$ and $\beta$ commute, so that $\gamma=\beta$.
Our peak-lowering factorization is 
\[\beta\alpha^{-1}=\alpha^{-1}\beta.\]
To show this, of course we need to show that $\abs{\beta\alpha\cdot W}<\abs{W}$.
We have already shown this in Proposition~\ref{pr:steinberglengthchange}.
\end{proof}

\subsection{General cases with non-adjacent multipliers}\label{ss:alggen}
\begin{lemma}\label{le:lowerconj}
Suppose $(W,\alpha,\beta)$ is a peak with $\alpha\in\whset{a}$ and $\beta\in\Omega$, and $\gamma\in\whset{a}$ is an inner automorphism of $A_\Gamma$.
Then $\gamma\alpha$ and $\alpha\gamma$ are both in $\whset{a}$ and $(W,\gamma\alpha,\beta)$ and $(W,\alpha\gamma,\beta)$ are both peaks and can be lowered if and only if $(W,\alpha,\beta)$ can be lowered.
\end{lemma}
\begin{proof}
The claim that $\gamma\alpha$ and $\alpha\gamma$ are in $\whset{a}$ is obvious.
We note that $\gamma\alpha=\alpha\gamma'$ for a possibly different inner automorphism $\gamma'\in\whset{a}$, so it is enough to prove the lemma for $(W,\gamma\alpha,\beta)$.
Also, it is enough to show only the ``if" direction of the statement.

First of all, if $(W,\alpha,\beta)$ is a peak, then $(W,\gamma\alpha,\beta)$ is as well, since action by inner automorphisms has no effect on cyclic words.
Now we suppose that $\delta_1,\dotsc,\delta_k\in\Omega$ and 
\[\beta\alpha^{-1}=\delta_k\dotsm\delta_1\]
is a lowering of $(W,\alpha,\beta)$.
Let $l$ be one of $1,\dotsc, k$ such that $\delta_l\dotsm\delta_1\alpha\cdot W$ is minimal length among all $\delta_i\dotsm\delta_1\alpha\cdot W$.
Since a conjugate of an inner automorphism by another automorphism is still an inner automorphism, and since every inner automorphism is a product of inner automorphisms from $\Omega$, we can find inner automorphisms $\gamma_1,\dotsc,\gamma_m$ in $\Omega$ such that
\[\beta\alpha^{-1}\gamma^{-1}=\delta_k\dotsm\delta_{l+1}\gamma_m\dotsm\gamma_1\delta_l\dotsm\delta_1.\]
Since inner automorphisms do not change the length of cyclic words, this is a lowering of the peak $(W,\gamma\alpha,\beta)$.
\end{proof}

If $\alpha$ is a classic long-range Whitehead automorphism with multiplier $a$, we will sometimes consider the \emph{complement} of $\alpha$.
This is the classic long-range Whitehead automorphism with multiplier $a$, with the property that the union of the support of $\alpha$ and the support of its complement is $(\Gamma\setminus\st(a))^{\pm1}$.
It follows that the product of $\alpha$ and its complement is inner, so Lemma~\ref{le:lowerconj} implies that we may freely replace an automorphism in a peak with its complement.

\begin{proposition}
\label{pr:nonadjacentnodom}
Suppose $a$ is not adjacent to $b$, $b$ does not dominate $a$ and $a$ does not dominate $b$.
Then the peak $(W,\alpha,\beta)$ can be lowered.
\end{proposition}

\begin{proof}
We prove the proposition by proving it in successively more general cases, with each case building on the last.
First we prove:
\begin{case*}
The proposition is true if $\beta$ is a classic long-range Whitehead automorphism with multiplier $b$, $\beta|_{[a]}$ is the identity, $\alpha|_{[b]}$ is not the identity and $\supp(\alpha)\cap\supp(\beta)=\varnothing$.
\end{case*}
Since $a$ does not dominate $b$, there is some nontrivial element $v\in\genby{[a]}$ such that $\alpha(b)=vbv^{-1}$.
We define a new automorphism $\alpha_1\in\whset{a}$ by $\alpha_1(c)=v^{-1}\alpha(c)v$ for every $c\in\Gamma$.
Then $\alpha_1$ and $\alpha$ differ by an inner automorphism, and to prove the claim, it is enough to show that the peak $(W,\alpha_1,\beta)$ can be lowered.

\begin{claim*}
If $c^{\pm1}\in\supp(\beta)$, then either $a$ dominates $c$ or $Y^{\pm1}\subset\supp(\beta)$, where $Y$ is the connected component of $c$ in $\Gamma\setminus\st(a)$.
\end{claim*}
Suppose that $c$ is a vertex of $\Gamma$ with $c$ or $c^{-1}\in\supp(\beta)$ and $a$ does not dominate $c$.
Let $Y$ be the connected component of $c$ in $\Gamma\setminus\st(a)$.
We want to show that $Y$ is a subset of a connected component of $\Gamma\setminus\st(b)$ with at least two vertices.
This is enough, since by Lemma~\ref{le:actoncomps} this implies that $Y^{\pm1}\subset\supp(\beta)$.
Then it is enough to show that $b$ does not dominate $c$ and that $Y$ does not contain an element of $\st(b)$.

If $b$ dominates $c$, then since $a$ does not dominate $c$, either $b$ is adjacent to $c$ or there is a vertex adjacent to $b$ and $c$ but not adjacent to $a$.
If $Y$ contains an element of $\st(b)$, then there is a path from $c$ to $b$ outside of $\st(a)$.
In any of these three cases, it follows that $b$ and $c$ are in the same component of $\Gamma\setminus\st(a)$.
However, this contradicts Lemma~\ref{le:actoncomps}: $\alpha$ fixes $c$ and $c^{-1}$ since $\beta$ does not and $\supp(\alpha)\cap\supp(\beta)=\varnothing$, but $\alpha$ does not fix $b$.
This proves the claim.

By the claim, $\supp(\beta)$ is a union of elements $c^{\pm1}$ where $a$ dominates $c$, and subsets $Y^{\pm1}$ where $Y$ is an entire connected component of $\Gamma\setminus\st(a)$.
In particular, the following automorphisms are products of Laurence generators and are well defined.
Let $\alpha_2(c)$ be in $\{c,cv^{-1},vc,vcv^{-1}\}$ for each vertex $c\in \Gamma$, with $\supp(\alpha_2)=\supp(\beta)$,
and let $\alpha_3=\alpha_1\alpha_2^{-1}$.

Now our goal is to show that $\abs{\alpha_3\cdot W}<\abs{W}$.
If this is true, then we will have that $(W,\alpha_3,\beta)$ is a peak with $\beta|_{[a]}$ and $\alpha_3|_{[b]}$ being identity maps and $\supp(\alpha_3)\cap\supp(\beta)=\varnothing$.
In particular, Lemma~\ref{le:nonadjsteinbergidentity} applies to lower $(W,\alpha_3,\beta)$.
Appending $\alpha_2^{-1}$ to this factorization will give us a peak-lowering factorization for $\beta\alpha^{-1}$.

In fact we show that 
\[\abs{\alpha_1\cdot W}-\abs{\alpha_2\cdot W}\geq \abs{v}(\abs{W}-\abs{\beta\cdot W}).\]
Let $T$ be a syllable decomposition of $W$ with respect to $[b]$.
For each syllable $t_i$ in $T$, either $\beta\cdot t_i$ is longer that $t_i$ or shorter than $t_i$ or the same length.
Let $L_1$ be the sum over syllables $t_i$ that $\beta$ lengthens of $\abs{\beta\cdot t_i}-\abs{t_i}$, and let $S_1$ be the sum over syllables $t_i$ that $\beta$ shortens of $\abs{t_i}-\abs{\beta\cdot t_i}$.
Then
\[\abs{W}-\abs{\beta\cdot W}= S_1-L_1.\]
Let $T'$ be a syllable decomposition of $\alpha_1\cdot W$ with respect to $[a]$, and similarly let $L_2$ be the total increase in length of syllables $\alpha_2^{-1}$ lengthens and let $S_2$ be the total decrease in length of syllables that $\alpha_2^{-1}$ shortens.
Then 
\[\abs{\alpha_1\cdot W}-\abs{\alpha_3\cdot W}=S_2-L_2.\]
We assume without loss of generality that $\alpha_1\cdot T$ and $T'$ have the same associated representative of $\alpha_1\cdot W$.

Suppose $t_i$ is a syllable in $T$ that decreases in length under $\beta$.
Then $t_i$ contains $b$ to a nonzero power and exactly one endpoint of $t_i$ is in $\supp(\beta)$.
Suppose $t_i=cb^kud$ and where $u\in\genby{\st(b)\setminus\{b\}}$ and without loss of generality assume $c$ is in $\supp(\beta)$.
Then the initial part $cb^{\pm1}$ of $t_i$ is a syllable of $W$ with respect to $[a]$.
The image of this syllable under $\alpha_1$ is a syllable of $\alpha_1\cdot W$ with respect to $[a]$; it is of the form $c\alpha_1(b)^{\pm1}$.
Clearly $\alpha_3\cdot cb^{\pm1}$ (acting on the syllable) is $cv^{-1}\alpha_1(b)^{\pm1}$.
This is a syllable of $\alpha_1\cdot W$ in $T'$ that shortens by $\abs{v}$ under the action of $\alpha_2^{-1}$.
In particular, summing over all syllables in $T$ that decrease in length under $\beta$, we have
\[\abs{v} S_1\leq S_2.\]

Suppose $t'_j$ is a syllable in $T'$ that increases in length under $\alpha_2^{-1}$.
Suppose $t'_j=cwud$ where $u\in\genby{\st(a)\setminus[a]}$ and $w\in\genby{[a]}$.
We note that $\alpha_2$ is like a classic Whitehead automorphism in that $\alpha_2(c)$ must be one of $\{c,cv^{-1},vc,vcv^{-1}\}$.
In particular, since $\alpha_2^{-1}\cdot t'_j$ is longer than $t_j$, we must have exactly one of $c$ and $d^{-1}$ in $\supp(\alpha_2)$, and it must be the case $w\neq 1$ that the $v$ or $v^{-1}$ that is inserted by $\alpha_2^{-1}$ does not cancel away completely into $w$.
We assume without loss of generality that $c$ is in $\supp(\alpha_2)=\supp(\beta)$ but $d^{-1}$ is not.
Clearly $t'_j$ increases in length by at most $\abs{v}$ under $\alpha_2^{-1}$.
We consider $\alpha_1^{-1}\cdot t'_j$, which is a syllable of $W$ with respect to $[a]$.
Then $\alpha_1^{-1}\cdot t'_j=cw'ud$, where $w'\in\genby{[a]}$.
Since $c\in\supp(\beta)$, we know $c\neq b^{\pm1}$.
We note that if $d$ is $b^{\pm1}$, then it must be that $w'u$ contains a generator not in $\st(b)$.
If $w'$ is trivial, then $u$ contains elements adjacently dominated by $a$, which cannot be in $\st(b)$ since $a$ is not adjacent to $b$.
If $w'$ is nontrivial, then $w'$ contains elements of $[a]$ which are not adjacent to $b$.
Further $w'u$ contains no elements of $\supp(\beta)$ because $\beta$ fixes $\st(a)$.
So whether $d=b^{\pm1}$ or not, the initial segment of $cw'ud$ contains a syllable with respect to $[b]$ starting with $c$ (in $\supp(\beta)$) and ending with an element not in $\supp(\beta)$.
Such a syllable increases in length by one under $\beta$.
Then summing over all syllables in $T'$ that increase in length under $\alpha_2^{-1}$, we have
\[L_2\leq \abs{v} L_1.\]
Summing these gives us
\[\abs{v}S_1+L_2\leq \abs{v}L_1+S_2,\]
in other words 
\[\abs{v}(\abs{W}-\abs{\beta\cdot W})=\abs{v}(S_1-L_1)\leq S_2-L_2 = \abs{\alpha_1\cdot W}-\abs{\alpha_3\cdot W}.\]
We can rewrite this as
\[\abs{W}-\abs{\alpha_3\cdot W} \geq \abs{v}(\abs{W}-\abs{\beta\cdot W}) + \abs{W}-\abs{\alpha_1\cdot W}.\]
Since $(W,\alpha_1,\beta)$ is a peak, we know $\abs{W}-\abs{\beta\cdot W}\geq 0$ and $\abs{W}-\abs{\alpha_1\cdot W} \geq0$ with at least one strict.
In either case, it follows that $\abs{\alpha_3\cdot W}<\abs{W}$.
As explained above, this means $(W,\alpha_3,\beta)$ is a peak that Lemma~\ref{le:nonadjsteinbergidentity} lowers, and by attaching $\alpha_2^{-1}$ to the lowering factorization, we get a peak-lowering factorization for the peak $(W,\alpha_1,\beta)$.

This proves the first case.
Now we move to a slightly more general case.
\begin{case*}
The proposition is true for general $\alpha\in\whset{a}$, $\beta\in\whset{b}$ with $\supp(\alpha)\cap\supp(\beta)=\varnothing$.
(We are still assuming $a$ does not dominate $b$ and $b$ does not dominate $a$.)
\end{case*}
If $\alpha|_{[b]}$ and $\beta|_{[a]}$ are both the identity, then Lemma~\ref{le:nonadjsteinbergidentity} lowers the peak.
So we assume without loss of generality that $\alpha|_{[b]}$ is not the identity.
Lemma~\ref{le:basicobservations} implies that $\beta$ is a long-range automorphism in this case.
If we also assume that $\beta|_{[a]}$ is not the identity, then $\alpha$ is also a long-range automorphism, and therefore Theorem~\ref{th:longrangepeakreduction} applies to lower the peak.
So we assume that $\beta|_{[a]}$ is the identity.

Since $\beta$ is a long-range automorphism, we peak-reduce $\beta$ with respect to $W$ to get a factorization $\beta=\beta_k\dotsm\beta_1$ by classic long-range Whitehead automorphisms.
It is possible that $\abs{\beta_i\dotsm\beta_1\cdot W}<\abs{W}$ for some $i$  in $1,\dotsc,k$.
If this is the case, we lower the peak $(W,\alpha,\beta_1)$ by the method above and concatenate $\beta_k\dotsm\beta_2$ to this factorization to get a peak-lowering factorization for $\beta\alpha^{-1}$.

Otherwise, since $\beta_k\dotsm\beta_1$ is peak reduced, we have $\abs{\beta\cdot W}=\abs{W}$, $\abs{\alpha\cdot W}<\abs{W}$ and $\abs{\beta_i\dotsm\beta_1\cdot W}=\abs{W}$ for $i=1,\dotsc,k$.
We prove the claim by induction on the length $k$ of the factorization.
We let $\alpha_1,\alpha_2$ and $\alpha_3$ have the same meanings as above, with $\beta_1$ taking the role of $\beta$.
Then $\beta_1\alpha_1^{-1}=\alpha_3^{-1}\beta_1\alpha_2^{-1}$ is a peak-lowering factorization for the peak $(W,\alpha_1,\beta_1)$, using Lemma~\ref{le:nonadjsteinbergidentity} as explained above. 
However, if $k>1$, then we have a new peak $(\beta_1\cdot W, \alpha_3, \beta_k\dotsm\beta_2)$.
The inductive hypothesis implies that this new peak can be lowered.
Concatenating $\alpha_3^{-1}\beta_1$ onto the peak-lowering factorization for $(\beta_1\cdot W,\alpha_3,\beta_k\dotsm\beta_2)$ gives us a peak-lowering factorization for $(W,\alpha_1,\beta)$, which gives us a peak-lowering factorization for the original peak using Lemma~\ref{le:lowerconj}.
This proves the second case.

\paragraph{General case.} Finally we prove the proposition as stated.
We have $\alpha\in\whset{a}$ and $\beta\in\whset{b}$.
If $\alpha$ does not fix $b$, then since $a$ does not dominate $b$, we know that $\alpha$ conjugates $b$ by a nontrivial element of $\genby{[a]}$.
Then by Lemma~\ref{le:lowerconj}, we may replace $\alpha$ by its product with an inner automorphism and assume that $\alpha$ fixes $b$.
Similarly, we assume that $\beta$ fixes $a$.

\begin{claim*}
If $c^{\pm1}$ is in $\supp(\alpha)\cap\supp(\beta)$, then either both $a$ and $b$ non-adjacently dominate $c$, or else the connected component $Y$ of $c$ in $\Gamma\setminus\st(a)$ is also a connected component of $\Gamma\setminus\st(b)$ and $Y\cup Y^{-1}\subset\supp(\alpha)\cap\supp(\beta)$.
\end{claim*}

If $c^{\pm1}$ is in $\supp(\alpha)\cap\supp(\beta)$, then one possibility is that $a$ dominates $c$.
If $b$ were adjacent to $c$, it would imply that $a$ is adjacent to $b$, which is not the case.
So $b$ is not adjacent to $c$.
If $b$ does not also dominate $c$, then there is a vertex adjacent to $a$ and $c$ but not adjacent to $b$.
This means that $a$ and $c$ are in the same component of $\Gamma\setminus\st(b)$; then $\beta(a)=a$ implies that $c,c^{-1}\notin\supp(\beta)$ by Lemma~\ref{le:actoncomps}.
This is a contradiction, so we see that both $a$ and $b$ dominate $c$.
Then if $c$ were adjacent to $a$, it would imply that $a$ and $b$ are adjacent, so $c$ is also not adjacent to $a$.

Next we suppose that $c$ in $\supp(\alpha)\cap\supp(\beta)$ is in a connected component $Y$ of $\Gamma\setminus\st(a)$ with at least two vertices.
First of all, $b$ is not adjacent to $c$, since otherwise $\alpha$ would not fix $b$ (by Lemma~\ref{le:actoncomps}, since $b$ would be in $Y$).
Further, it cannot be the case that $b$ dominates $c$---if it did, $b$ would be adjacent to a vertex of $Y$ other than $c$ and therefore $b$ would be in $Y$ and $\alpha$ would not fix $b$.
So let $Y'$ be the component of $c$ in $\Gamma\setminus\st(b)$.
Suppose $d\in Y$. 
Then there is a path from $d$ to $c$ outside of $\st(a)$.
If this path intersects $\st(b)$, then $b$ would be in $Y$, which is impossible.
So the path from $d$ to $c$ is outside $\st(b)$, meaning that $d$ is in $Y'$.
Since $d$ was arbitrary $Y\subset Y'$.
By a parallel argument, $Y'\subset Y$, and therefore they are equal.
This proves the claim.

We prove the proposition by induction on the cardinality of $\supp(\alpha)\cap\supp(\beta)$.
The base case is that $\supp(\alpha)\cap\supp(\beta)=\varnothing$, which we covered in a previous case.

Now we work the inductive step.
Since $\supp(\alpha)\cap\supp(\beta)$ is nonempty, we select an element $c$ from it.
First we assume that both $a$ and $b$ non-adjacently dominate $c$.
Fix a graphically reduced representative for $W$.
Let $n_a$ denote the number of instances of $c$ or $c^{-1}$ in the representative in a subword $cud$ or $duc^{-1}$ with $u\in\genby{\st(a)}$ and $d\in[a]^{\pm1}$ and let $n_b$ be defined similarly with $a$ replaced by $b$.
Let $n'$ count the number of remaining instances of $c$ or $c^{-1}$ in the representative (those that are counted by neither $n_a$ or $n_b$).
Let $v_1$ and $v_2$ be in $\genby{[a]}$ with $\alpha(c)=v_1cv_2$ and let
$u_2$ and $u_2$ be in $\genby{[b]}$ with $\beta(c)=u_1cu_2$.
Let $\alpha_c$  be the element of $\whset{a}$ with $\alpha_c(d)=\alpha(d)$ for all $d\neq c$, 
and $\alpha_c(c)=v_1c$.
Let $\beta_c$ be defined similarly, with $\beta_c(c)=u_1c$.
We note that $\alpha_c$ and $\beta_c$ can be expressed as products of Laurence generators and are well-defined automorphisms.

The decomposition of $W$ into syllables with respect to $[a]$ coming from our representative, with Lemma~\ref{le:syllablesandlengthchange}, tells us the following: at each instance of $c$ counted by $n_b$ or $n'$, $\alpha$ increases the length of $W$ by $\abs{v_2}$ more than $\alpha_c$ does, and at each instance of $c$ counted by $n_a$, $\alpha$ may decrease the length of $W$ by up to $\abs{v_2}$ more than $\alpha$ does or increase it by up to $\abs{v_2}$ more.
Specifically,
\[\abs{v_2}(n'+n_b-n_a)\leq \abs{\alpha\cdot W}-\abs{\alpha_c\cdot W}\leq \abs{v_2}(n'+n_b+n_a).\]
Similarly for $\beta$:
\[\abs{u_2}(n'+n_a-n_b)\leq \abs{\beta\cdot W}-\abs{\beta_c\cdot W}\leq \abs{u_2}(n'+n_a+n_b).\]
Of course, since $c\in\supp(\alpha)\cap\supp(\beta)$, we know $\abs{v_2}>0$ and $\abs{u_2}>0$.

Suppose both $\abs{\alpha_c\cdot W}\geq\abs{\alpha\cdot W}$ and $\abs{\beta_c\cdot W}\geq\abs{\beta\cdot W}$.
Then $n'+n_b-n_a\leq 0$ and $n'+n_a-n_b\leq 0$.
Summing these, we see that $n'\leq 0$, but since $n'$ is a nonnegative integer, $n'=0$.
Then $n_b=n_a$, and in this case we have both $\abs{\alpha_c\cdot W}=\abs{\alpha\cdot W}$ and $\abs{\beta_c\cdot W}=\abs{\beta\cdot W}$.
By the definition of a peak, either $\abs{\alpha\cdot W}<\abs{W}$ or $\abs{\beta\cdot W}<\abs{W}$.
So in this case we have either $\abs{\alpha_c\cdot W}<\abs{W}$ or $\abs{\beta_c\cdot W}<\abs{W}$.

Otherwise, we have either $\abs{\alpha_c\cdot W}<\abs{\alpha\cdot W}$ or $\abs{\beta_c\cdot W}<\abs{\beta\cdot W}$.
Since the definition of a peak ensures that both $\abs{\alpha\cdot W}\leq \abs{W}$ and $\abs{\beta\cdot W}\leq \abs{W}$,
in this case we also have either $\abs{\alpha_c\cdot W}<\abs{W}$ or $\abs{\beta_c\cdot W}<\abs{W}$.

Without loss of generality we suppose $\abs{\alpha_c\cdot W}<\abs{W}$.
Then $(W,\alpha_c,\beta)$ is a peak that can be lowered by the inductive hypothesis.
So we append the element $(\alpha_c\alpha^{-1})\in\whset{a}$ onto a peak-lowering factorization for $(W,\alpha_c,\beta)$ to get a peak-lowering factorization for $(W,\alpha,\beta)$.

The case where $c\in\supp(\alpha)\cap\supp(\beta)$ is in a connected component of $\Gamma\setminus\st(a)$ is similar, with $n_a$, $n_b$ and $n'$ defined similarly.
\end{proof}

\begin{lemma}\label{le:classicsteinberg}
Suppose $a$ and $b$ are vertices of $\Gamma$ with $[a]\neq[b]$ and $\alpha\in\whset{a}$ and $\beta\in\whset{b}$ are classic long-range Whitehead automorphisms such that: $\alpha(b)=b$, $\supp(\alpha)\cap\supp(\beta)=\varnothing$ and
$\beta(a)$ is $a$ or $ab$.
Then $\alpha\beta\alpha^{-1}$ is a classic long-range Whitehead automorphism in $\whset{b}$ and for any tuple $W$ of cyclic words, we have
\[\abs{W}-\abs{\beta\cdot W}=\abs{\alpha\cdot W}-\abs{\alpha\beta\cdot W}.\]
\end{lemma}

\begin{proof}
The fact that $\alpha\beta\alpha^{-1}\in\whset{b}$ is relation R4 from Section~2.4 in Day~\cite{Day1}, with $\alpha$ being $(A,a)^{-1}$ and $\beta$ being $(B,b)$ in that statement.
The equation of differences of lengths is essentially Sublemma~3.21 from Day~\cite{Day1}, applied twice.
Near the end of the proof of that statement, the equation above appears as an inequality ($\alpha$ and $\alpha^{-1}$ are switched and the terms rearranged, but it is the same assertion).
The same inequality with $W$ replaced by $\alpha\cdot W$ and $\alpha$ replaced by $\alpha^{-1}$ is the inequality in the reverse direction, proving the equation.
\end{proof}

\begin{proposition}\label{pr:nonadjacentasymmetric}
Suppose $a$ is not adjacent to $b$, $b$ dominates $a$ and $a$ does not dominate $b$.
Then the peak $(W,\alpha,\beta)$ can be lowered.
\end{proposition}

\begin{proof}
More precisely, we prove:
\begin{claim*}
Suppose $a$ is not adjacent to $b$, $b$ dominates $a$ and $a$ does not dominate $b$.
Then the peak $(W,\alpha,\beta)$ has a peak-lowering factorization
such that each automorphism in the factorization is in $\whset{b}$ or is a long-range automorphism in $\whset{a}$.
\end{claim*}
Since $a$ does not dominate $b$, we know that $\alpha$ acts on $[b]$ by conjugating it by a fixed element of $\genby{[a]}$.
Then by Lemma~\ref{le:lowerconj}, we can replace $\alpha$ by its composition with an inner automorphism from $\whset{a}$ and assume that $\alpha$ fixes $[b]$.
We also note that since $b$ non-adjacently dominates $a$, $a$ does not adjacently dominate anything and therefore $\alpha$ is long-range and $[a]=\{a\}$ (see Lemma~\ref{le:basicobservations}).

We induct on $\abso{\alpha}$, the length of the outer class of $\alpha$ as a product of Laurence generators.
We use as a base case the case that $\alpha$ is a classic long-range Whitehead automorphism, which will certainly be true if $\abso{\alpha}$ is one.
Before proving this case, we prove the inductive step.
We peak-reduce $\alpha$ with respect to $W$ by Theorem~\ref{th:longrangepeakreduction}.
We select the first automorphism $\alpha_1$ out of such a peak-reducing factorization.
Let $\alpha_2=\alpha\alpha_1^{-1}$.
Then $\alpha_1$ will be a classic long-range Whitehead automorphism with multiplier $a^{\pm1}$, $\abso{\alpha_2}<\abso{\alpha}$, and $\abs{\alpha_1\cdot W}\leq \abs{W}$, with the inequality being strict if $\abs{\alpha\cdot W}<\abs{W}$.
(Specifically, we factor  $\alpha$ as a product of transvections and partial conjugations with multipliers $a^{\pm1}$ and apply the algorithm from Day~\cite{Day1}; at each step, the algorithm merges, commutes or splits the automorphisms in the factorization, and the resulting peak-reduced factorization consists entirely of long-range automorphisms with multipliers in $\{a,a^{-1}\}$.)
Then $(W,\alpha_1,\beta)$ is a peak.
We apply the claim and get a peak-lowering factorization of $\beta\alpha_1^{-1}$.
If the last term in the factorization is a long-range automorphism $\alpha'\in\whset{a}$, then $\abs{\alpha'^{-1}\alpha_1\cdot W}<\abs{W}$ by the definition of a peak-lowering factorization.
We peak-reduce $\alpha_2\alpha'$ with respect to $\abs{\alpha'^{-1}\alpha_1\cdot W}$ using Theorem~\ref{th:longrangepeakreduction}.
If we replace the element $\alpha'$ in our factorization of $\beta\alpha_1^{-1}$ with this peak-reduced factorization of $\alpha_2\alpha'$, then the result is a peak-lowering factorization of $\beta\alpha^{-1}$.
Otherwise the last term in the factorization of $\beta\alpha_1^{-1}$ is an automorphism $\beta'\in\whset{b}$.
Then $(\alpha_1\cdot W,\alpha_2,\beta'')$ is a peak satisfying the hypotheses of the claim, but with $\abso{\alpha_2}<\abso{\alpha}$.
Then by induction, we apply the claim and get a peak-lowering factorization of $\beta'\alpha_2^{-1}$.
We replace $\beta'$ in the factorization of $\beta\alpha_1^{-1}$ with the entire factorization of $\beta'\alpha_2^{-1}$ to get a peak-lowering factorization of $\beta\alpha^{-1}$.

Now we prove the claim in the base case: with the additional hypothesis that $\alpha$ is a classic long-range Whitehead automorphism with multiplier $a$ (the case where the multiplier is $a^{-1}$ is similar).
Let $\beta''$ be the short-range part of $\beta$; in other words let $\beta''$ equal $\beta$ on $\st(b)$ and let $\beta$ fix all other vertices.
Then $\beta''$ will be a product of short-range transvections and inversions and will be a well defined automorphism.
Let $\beta'$ be the difference, the long-range part, so that $\beta=\beta'\beta''$.
Everything in $\supp(\beta'')$ is adjacent to $b$ and dominated by $b$, or equal to $b$.
Since $a$ is not adjacent to $b$, $a$ cannot be adjacent to anything in $\supp(\beta'')$ (otherwise domination would force $a$ to be adjacent to $b$).
Since $\alpha$ fixes $b$, $\alpha$ fixes the entire connected component of $b$ in $\Gamma\setminus\st(a)$.
Therefore $\supp(\beta'')\cap\supp(\alpha)=\varnothing$.
Since $a$ is not adjacent to $b$, by definition, $\beta''$ fixes $[a]$.
Then by Lemma~\ref{le:steinbergrel} and Proposition~\ref{pr:steinberglengthchange}, we know $\alpha$ commutes with $\beta''$ and 
\[\abs{W}-\abs{\alpha\cdot W}=\abs{\beta''\cdot W}-\abs{\alpha\beta''\cdot W}.\]
So $\beta\alpha^{-1}=\beta'\alpha^{-1}\beta''$.

Since we are about to start a step that we will repeat, we relabel $\alpha$ as $\alpha'$ and $\beta'^{-1}\beta\cdot W$ as $W'$.
Let $\alpha''$ denote the trivial automorphism at first.
Then we have $\beta\alpha^{-1}=\beta'\alpha'^{-1}\beta''\alpha''^{-1}$.
We process this factorization into a better one using an algorithm with a loop.
If $\abs{W'}<\abs{W}$, then we do not enter the loop and instead skip ahead to the next step.
Otherwise $(W',\alpha',\beta')$ is a peak:
\[\abs{\beta'\cdot W'}=\abs{\beta\cdot W}\leq \abs{W}\leq\abs{W'}\]
and
\[\abs{W'}-\abs{\alpha'\cdot W'}=\abs{\beta''\cdot W}-\abs{\alpha'\beta''\cdot W}=\abs{W}-\abs{\alpha'\cdot W}\geq 0,\]
with one of these inequalities strict because $(W,\alpha,\beta)$ is a peak.

\begin{claim*}
There is an algorithm to iteratively process the factorization $\beta\alpha^{-1}=\alpha''^{-1}\beta''\alpha'^{-1}\beta'$, such that 
at the beginning of each loop we have
\begin{itemize}
\item $(W',\alpha',\beta')$ is a peak, where $W'=\beta'^{-1}\beta\cdot W$,
\item $\alpha'$ is a non-inner long-range classic Whitehead automorphism,
\item $\beta'$ is a non-inner automorphism,
\item $\alpha'^{-1}\beta''\alpha'$ is in $\whset{b}$ and 
\[\abs{W'}-\abs{\alpha'\cdot W'}=\abs{\beta''^{-1}\cdot W'}-\abs{\beta''^{-1}\alpha'\cdot W'}.\]
\end{itemize}
With each repeat of the loop, we strictly decrease $\abs{W'}$, or we do not increase $\abs{W'}$ but strictly decrease one of $\abso{\alpha'}$ or  $\abso{\beta'}$ while leaving the other fixed.
At the end of an iteration, either the new $\abs{W'}$ satisfies $\abs{W'}<\abs{W}$ and we exit the loop, or we repeat the loop (and satisfy the conditions for beginning the loop with the relabeled terms).
\end{claim*}

We have already shown that if we enter the loop at all, then the beginning conditions are satisfied for the first iteration.
Now we explain the algorithm.
By construction, $\beta'$ is a long-range automorphism.
Using Theorem~\ref{th:longrangepeakreduction}, we peak-reduce $\beta'$ with respect to $W'$ and consider the first automorphism $\beta_0$ in the resulting peak-reduced factorization of $\beta'$.
This $\beta_0$ is a classic long-range Whitehead automorphism with multiplier $d^{\epsilon}$ for some $d\in[b]$ and $\epsilon=\pm1$.
Since $\abs{\beta'\cdot W'}\leq\abs{W'}$ (because $\abs{\beta'\cdot W'}=\abs{\beta\cdot W}\leq\abs{W}$), it follows from the definition of a peak-reduced factorization that $\abs{\beta_0\cdot W'}\leq\abs{W'}$.
Then $(W',\alpha',\beta_0)$ is a peak.

If $\supp(\alpha')\cap\supp(\beta_0)=\varnothing$ and $a\notin\supp(\beta_0)$, then Lemma~\ref{le:classicsteinberg} applies.
In particular, $\alpha'\beta_0\alpha'^{-1}\in\whset{b}$ and
\[\abs{W'}-\abs{\alpha'\cdot W'}=\abs{\beta_0\cdot W'}-\abs{\alpha'\beta_0\cdot W'}.\]
We replace $\beta'$ with $\beta_0^{-1}\beta'$ and replace $\beta''$ with $\beta''\alpha'\beta_0\alpha'^{-1}$.
The new $W'$ is $\beta_0\cdot W'$.
We go back to the beginning of the loop since the conditions to continue the loop are satisfied.
If $\supp(\alpha')\subset\supp(\beta_0)$ and $a\in\supp(\beta_0)$, we replace $\beta_0$ with its complement so that we have $\supp(\alpha')\cap\supp(\beta_0)=\varnothing$ and $a\notin\supp(\beta_0)$.
Then we apply the previous case.
In these cases we have not increased $\abs{W'}$ or the length of $\alpha'$, but we have decreased the length of $\beta'$.
If the new $\beta'$ is inner, then the changes in length imply that the new $W'$ is shorter than $\beta\cdot W$ and is therefore shorter than $W$.

Now we want $a\in\supp(\beta_0)$; if this is not the case, we replace $\beta_0$ with its complement.
This replacement does not change $\beta_0\cdot W$, so $(W',\alpha',\beta_0)$ is still a peak.
After performing such a replacement if necessary, we definitely have $a\in\supp(\beta_0)$.
Further, since we have just considered this case separately, we can assume that $\supp(\alpha')\not\subset\supp(\beta_0)$.
Then we apply Corollary~\ref{co:shorterfactors} to $(W',\alpha',\beta_0)$ and one of the following holds:
\begin{itemize}
\item
there is a classic long-range Whitehead automorphism $\beta_1$ with multiplier $d^{-\epsilon}$ such that $\abs{\beta_1\cdot W'}<\abs{W'}$ and $\supp(\beta_1)\cap\supp(\alpha)=\varnothing$ and $\beta_1$ fixes $[a]$, or
\item
there is a classic long-range Whitehead automorphism $\alpha_1$ with multiplier $a$ such that $\abs{\alpha_1\cdot W'}<\abs{W'}$ and $\supp(\alpha_1)\subset\supp(\alpha')\cap\supp(\beta_0)$.
\end{itemize}
If $\supp(\alpha')\cap\supp(\beta_0)=\varnothing$, then the first case holds.

If the second case holds,  $\alpha'\alpha_1^{-1}$ commutes with $\beta''$ since $\supp(\alpha'\alpha_1^{-1})\subset\supp(\alpha')$ and $\supp(\alpha')\cap\supp(\beta'')=\varnothing$, and since $\alpha'\alpha_1^{-1}$ fixes $[b]$ and $\beta''$ fixes $[a]$.
Note that $\alpha_1\neq\alpha$ since we assumed $\supp(\alpha')\not\subset\supp(\beta_0)$.
So we relabel $\alpha_1$ as $\alpha'$ and relabel $\alpha''$ as $\alpha''\alpha'\alpha_1^{-1}$.
Then we still have the conditions we expect at the beginning of the loop.
The new $\alpha'$ is nontrivial because $\alpha_1\cdot W'\neq W'$ and $(W',\alpha',\beta')$ is still a peak because $\abs{W'}$ is now strictly greater than $\abs{\alpha'\cdot W'}$ and $\abs{W'}-\abs{\beta'\cdot W'}$ is unchanged.
Then we repeat the loop.
In this case we have not increased $\abs{W'}$ or the length of $\beta'$ but we have decreased the length of $\alpha'$.

In the first case,
Lemma~\ref{le:steinbergrel} and Proposition~\ref{pr:steinberglengthchange} imply that $\alpha'$ and $\beta_1$ commute and that
\[\abs{W'}-\abs{\alpha'\cdot W'}=\abs{\beta_1\cdot W'}-\abs{\alpha'\beta_1\cdot W'}.\]
We relabel $\beta_1^{-1}\beta'$ as $\beta'$, and relabel $\beta_1\beta''$ as $\beta''$.
Note that we still have
$\beta\alpha^{-1}=\beta'\alpha'^{-1}\beta''\alpha''^{-1}$.
Further, we still have $\beta''$ fixing $[a]$ and $\supp(\beta'')\cap\supp(\alpha')=\varnothing$, since this was true of $\beta_1$ and the old $\beta''$.
We relabel $W'$ as $\beta'^{-1}\beta\cdot W=\beta_1\cdot W'$.
We also note that the new $\abs{W'}$ is strictly less than the old $\abs{W'}$, since it is the image of the old $W'$ under $\beta_1$.
If $\abs{W'}<\abs{W}$, then we exit the loop.
Otherwise $(W',\alpha',\beta')$ is still a peak since $\alpha'$ shortens the new and old $W'$ by the same amount.
We repeat the loop with the newly relabeled automorphisms.
If the new $\beta'$ is inner, then the changes in length imply that the new $W'$ is shorter than $\beta\cdot W$ and is therefore shorter than $W$.

So we leave the loop when $\abs{W'}<\abs{W}$.
The sum of differences 
\[\abs{\alpha''^{-1}\alpha\cdot W}-\abs{\alpha\cdot W}+\abs{W'}-\abs{\alpha'\cdot W'}\]
equals $\abs{W}-\abs{\alpha\cdot W}$.
These two facts together are enough to deduce  that $\beta\alpha^{-1}=\alpha''^{-1}\beta''\alpha'^{-1}\beta'$  is a peak-lowering factorization.
\end{proof}

\subsection{Peak reduction and stabilizer presentations}
Finally, we use peak reduction to homotope a path in the complex $Z$ from Section~\ref{ss:stabpres} in order to prove Proposition~\ref{pr:homotopeinZ}.
We need the following.
\begin{lemma}\label{le:mastertuple}
Consider the set of conjugacy classes of $A_\Gamma$ consisting of
all conjugacy classes of length one represented by a positive generator and
all conjugacy classes of length  two represented by a product of two non-commuting generators.
Let $U$ be a tuple of conjugacy classes whose entries are this set of classes in some order.
Then 
$U$ is minimal length in its automorphism orbit.

Further, if $\alpha$ is a generalized Whitehead automorphism sending one minimal-length representative of  the orbit of $U$ to another,  then $\alpha$ is a permutation automorphism or an inner automorphism.
\end{lemma}

\begin{proof}
Without loss of generality, we suppose that $U$ is ordered 
with the conjugacy classes of length one first and the classes of length two following.
By peak reduction, to show that $U$ is minimal length in its orbit, it is enough to show that no element of $\whset{a}$ can shorten it, for any $a$ in $\Gamma$.
Let $T$ be a syllable decomposition of $U$; we claim that the matrix $\nu(T)$ can be put in $G_{\Q}$--normal form by permuting the rows.
To see this we examine the matrix $\nu(T)$.

In the following, $\epsilon$ and $\delta$ are always in $\{1,-1\}$.
For each $b$ adjacent to or equal to $a$, the conjugacy class represented by $b^\epsilon$  maps to $\epsilon r_b$ under $\nu$.
For each $b$ not adjacent to $a$, with $a$ dominating $b$, the conjugacy class represented by $b^\epsilon$ is a single syllable that maps to $\epsilon r_b+\epsilon l_b$ under $\nu$.
For $b$ not adjacent to $a$, with $a$ not dominating $b$, the conjugacy class represented by $b^\epsilon$ is a single syllable that maps to $0$ under $\nu$ (it maps to $r_Y-r_Y$ where $Y$ is the component of $b$ in $\Gamma\setminus\st(a)$).
For $b$ and $c$ both adjacent to or equal to $a$ but not adjacent to each other, the conjugacy class of $b^\epsilon c^\delta$ maps to $\epsilon r_b+\delta r_c$.
For $b$ adjacent to or equal to $a$ and $c$  not adjacent to $a$ or $b$, but with $a$ dominating $c$, 
 the conjugacy class of $b^\epsilon c^\delta$ maps to $\epsilon r_b +\delta r_c+\delta l_c$.
For $b$ adjacent to or equal to $a$ and $c$ in the component $Y$ of $\Gamma\setminus\st(a)$ (which has at least two vertices), 
 the conjugacy class of $b^\epsilon c^\delta$ maps to $\epsilon r_b$.
For $b$ and $c$ not adjacent to $a$, but with $a$ dominating both $b$ and $c$, the conjugacy class of $(b c)^\epsilon$ splits into two syllables, which map to
$\epsilon(r_b+l_c)$ and $\epsilon(r_c+l_b)$,
and the conjugacy class of $(bc^{-1})^\epsilon$ splits into two syllables, which map to $\epsilon(r_b-r_c)$ and $\epsilon(l_b-l_c)$.
If $b$ is not adjacent to $a$, but $a$ dominates $b$, and $c$ is in the connected component $Y$ of $\Gamma\setminus\st(a)$ (which has at least two elements), then the conjugacy class of $(b c^\delta)^\epsilon$ splits into two syllables, which map to $\epsilon(r_b -r_Y)$ and $\epsilon(l_b+r_Y)$.
Finally, if $b$ is in the component $Y$ of $\Gamma\setminus\st(a)$ and $c$ is in the component $Z$ of $\Gamma\setminus\st(a)$ (with both having at least two elements), then the conjugacy class of $b^\epsilon c^\delta$ splits into two syllables that map to $r_Y-r_Z$ and $r_Z-r_Y$.

Let $n$ denote the cardinality of $[a]$ and let $k$ denote the number of basis elements of $Z_{[a]}$ other than the $r_b$ for $b\in[a]$.
First we verify that adding any linear combination of the last $k$ rows of $\nu(T)$ to any of the first $n$ does not simplify the matrix.
If $c$ does not commute with $a$ and $b$ is in $[a]$, then adding combinations of the row for $r_c$ and $l_c$ to the row for $r_b$ may simplify the columns for $b^\epsilon c^\delta$, but any such action will make the column for $c$ more complicated.
By our hypothesis on $U$, the column for $c$ precedes all the $b^\epsilon c^\delta$ columns of $\nu(T)$, and therefore none of these row moves can simplify $\nu(T)$.
However, none of the columns for the other kinds of conjugacy classes in $U$ (enumerated above) have a nonzero element in the first $n$ rows and a nonzero element in the last $k$ rows.
This implies that adding linear combinations of the last $k$ rows to the first $n$ rows cannot simplify any column of $\nu(T)$ without making a previous column worse.
Since the first $\abs{X}$ columns of $\nu(T)$ include the $n$ columns coming from the classes $b$ for $b$ in $[a]$, which map to $r_b$ in each case, and since all other columns with a nonzero entry in the first $n$ rows appear after the first $\abs{X}$ columns, the matrix $\nu(T)$ can be put in $G_{\Q}$--normal form by a permutation of the first $n$ rows.

If an element of $\whset{a}$ shortens $U$, then all it can do is to delete some elements of $[a]$ from somewhere in $U$.
Such a deletion would give us a tuple $U'$ with a syllable decomposition $T'$ differing only from $T$ in some deletions of elements of $[a]$.
The corresponding matrix $\nu(T')$ would also essentially already be in $G_{\Q}$--minimal form (up to a permutation), and therefore $U$ and $U'$ cannot be in the same orbit under $\whset{a}$ by Proposition~\ref{pr:normalformexists}.

To see the second part of the statement, we suppose that $\alpha$ is a generalized Whitehead automorphism sending $U$ to a tuple of the same length.
Of course, it is possible that $\alpha$ is a permutation automorphism.
We suppose that it is not, so that $\alpha$ is in $\whset{a}$ for some $a$.
By the form of $\nu(T)$, we know that the only way $\alpha$ can send $\nu(T)$ to itself is to add a linear combination of the bottom $k$ rows of $\nu(T)$ to some rows in the top, where that linear combination has a trivial value.
However, the only such linear combination that is trivial is the one corresponding to conjugation by some element: we add $0$ times each $r_b$ with $b$ commuting with $a$, $+1$ times each $r_c$ and $-1$ times each $l_c$ with $c$ not commuting with $a$ and with $a$ dominating $c$, and $+1$ times each $r_Y$ row.
This can be seen from the enumeration of columns above.
\end{proof}

\begin{proof}[Proof of Proposition~\ref{pr:homotopeinZ}]
Recall that $Z$ is our alleged presentation complex for $\Aut(A_\Gamma)_W$ and we have an edge-loop $p$ based at $W$ that we wish to homotope so that its edge labels are only permutations and inner automorphisms.

We define a non-locally finite graph $\widehat Z$ that is related to $Z$.
The vertices of $\widehat Z$ are the same as the vertices of $Z$.
The edges are as follows: if $W'$ is a vertex in $\widehat Z$ and $\alpha$ is a generalized Whitehead automorphism in $P$ or in $\whset{a}$ for some $a$ with $\abs{\alpha\cdot W'}=\abs{W'}$, then there is an edge labeled by $\alpha$ from $W'$ to $\alpha\cdot W'$.
Of course $Z^1$ is a finite subgraph of $\widehat Z$.

Now we define a map $\phi$ from edge paths in $\widehat Z$ to edge paths in $Z^1$.
For edges in $Z^1\subset\widehat Z$, we define $\phi$ to be the identity.
Suppose we have an edge in $\widehat Z$ starting at a vertex $W'$ that is not in $Z^1$.
Then this edge is labeled by an element $\alpha$ in $\whset{a}$ for some $a\in X$.
Let $S$ denote $(X\setminus[a])^{\pm1}\setminus \supp(\alpha)$.
Then $\alpha\in\whsetsr{a}{S}$.
By the construction of $Z$, there is an edge starting at $W'$ in $Z^1$ labeled by some $\gamma\in\whsetsr{a}{S}$.
Since $\alpha\gamma^{-1}\in(\whsetsr{a}{S})_{W'}$, there is an edge path $w$ in the loops at $W'$ representing $\alpha\gamma^{-1}$ (this step involves a choice, but we make these choices once and for all and forget about them).
Then $\phi$ sends the edge in $\widehat Z$ starting at $W'$ labeled with $\alpha$ to the concatenation of the path $w$ with the edge labeled by $\gamma$.
We extend $\phi$ to all edge paths in $Z^1$ by concatenation.

For any edge path in $\widehat Z$, the composition of edge labels gives us the same automorphism before and after applying $\phi$.
Further, the initial and terminal vertices of an edge path are the same before and after applying $\phi$.

Our path $p$ corresponds to a factorization
$1=\alpha_k\dotsm\alpha_1$
such that the intermediate images $\alpha_i\dotsm\alpha_1\cdot W$ have the same length as $W$ and there is an edge labeled by $\alpha_{i+1}$ from $\alpha_i\dotsm\alpha_1\cdot W$ to $\alpha_{i+1}\dotsm\alpha_1\cdot W$ for each $i$.
By Lemma~\ref{le:mastertuple}, there is a tuple $U$ of cyclic words, minimal length in its orbit, with the following property:
any generalized Whitehead automorphism sending one minimal length image of $U$ to another must be an inner automorphism or a permutation automorphism.
We consider the intermediate images $\alpha_i\dotsm\alpha_1\cdot U$ of $U$ and let $m$ denote the maximum length of any of these.
Let $V$ be the concatenation of $U$ with $m$ copies of $W$.
Our plan to prove the proposition is to peak-reduce the factorization $\alpha_k\dotsm\alpha_1$ with respect to $V$.
Peak reduction proceeds by finding peak-lowering substitutions to apply to peaks (subwords of length two) in the factorization.
Since we have $m$ copies of $W$ in $V$, a peak-lowering substitution can never produce a new factorization that has an intermediate image of $W$ longer than the original $W$---any increase in length in an image of $W$ would have to be countered by a decrease in length in an image of $U$ that is greater than the length of the longest intermediate image of $U$.
Then if we peak-reduce $\alpha_k\dotsm\alpha_1$ with respect to $V$, at each step the factorization remains the composition of labels on an edge path in $Z$.
So to prove the proposition, it is enough to explain how the peak-lowering substitutions from the proof of Theorem~\ref{th:fullfeaturedpeakreduction} correspond to homotopies of the path $p$.

By inspecting Sections~\ref{ss:algbasic} and~\ref{ss:alggen} above, we can see that the peak-lowering substitutions we use are from the following list:
\begin{itemize}
\item Factorization by classic Whitehead automorphisms: we take a long-range Whitehead automorphism from $\Omega$ and replace it with a product of classic long-range Whitehead automorphisms that follow an edge path in $Z$.
\item Applying classic peak reduction to a peak between two classic long-range Whitehead automorphisms.
\item Conjugating an automorphism across an inner automorphism as in Lemma~\ref{le:lowerconj}.
\item Applying a relation of length three between Whitehead automorphisms with the same multiplier, as in Lemma~\ref{le:samemultset}. 
\item Conjugating a permutation automorphism across another automorphism, as in Lemma~\ref{le:permlower}.
\item Applying a Steinberg relation from Lemma~\ref{le:steinbergrel}.
\end{itemize}

We proceed to show how applying any of the above rules to $\alpha_k\dotsm\alpha_1$ corresponds to homotoping the path given by $\phi(\alpha_k\dotsm\alpha_1)$ across some of the $2$--cells in $Z$.
Applying a substitution to $\alpha_k\dotsm\alpha_1$ corresponds to homotoping $\phi(\alpha_k\dotsm\alpha_1)$ across $2$--cells in $Z$.
In each of the following items, we indicate which cells we need.
\begin{itemize}
\item First we suppose we replace a long-range automorphism by a peak-reduced product of classic Whitehead automorphisms.
We suppose we have a vertex $W'$ in $\widehat Z$ and a long-range generalized Whitehead automorphism $\gamma$ labeling an edge in $\widehat Z$ starting at $W'$.
We  know $\phi$ sends $\gamma$ to a product $\gamma_2\gamma_1$ labeling edges starting at $W'$ in $Z$, where $\gamma_i$ has the same multiplier and support as $\gamma$, for $i=1,2$, and where $\gamma_1$ fixes $W'$  and  $\gamma_2\cdot W'=\gamma\cdot W'$.
Both $\gamma_1$ and $\gamma_2$ are long-range automorphisms, so we can drag the path they follow across $2$--cells in $Z$ of type (C2) in order to replace $\gamma_2\gamma_1$ with a product $\delta_l\dotsm\delta_1$ of long-range classic Whitehead automorphisms tracing out an edge path in $Z$.
We substitute this factorization back into $\alpha_k\dotsm\alpha_1$, replacing $\gamma$, and homotope the image under $\phi$ across cells of type (C2).
\item Now we suppose we must perform classic peak reduction.
Suppose $\beta\gamma$ is a subword of length two of $\alpha_k\dotsm\alpha_1$ and $\beta$ and $\gamma$ are classic long-range Whitehead automorphisms that we must peak-reduce with respect to $V'$, an intermediate image of $V$.
Since $V'$ is an intermediate image of $V$, it is the concatenation of $U'$ with several copies of $W'$, where $U'$ is an intermediate image of $U$ and $W'$ is an intermediate image of $W$.
Since $\beta$ and $\gamma$ are classic, $\phi$ sends the edges with these labels in $\widehat Z$ to edges with the same labels in $Z$, and $W'$ is the vertex of $Z$ that these edge are based at.
Lowering this peak with respect to $V'$ will produce a factorization that is peak-lowering with respect to $U'$, but which never increases the lengths of the intermediate images of $W$.
Therefore lowering this peak with respect to $V'$ will give us a factorization $\beta\gamma=\delta_l\dotsm\delta_1$ that follows a path in $Z$.
By Lemma~5.1 of Day~\cite{Day1}, the substitution lowering this peak is an application of a relation between classic Whitehead automorphisms.
Since these relations are all witnessed by $2$--cells of type (C3), we can replace $\beta\gamma$ by $\delta_l\dotsm\delta_1$ in $\alpha_k\dotsm\alpha_1$ and homotope $\phi(\alpha_k\dotsm\alpha_1)$ across cells of type (C3).
\item Next we suppose that we have a vertex $W'$ in $\widehat Z$ and a generalized Whitehead automorphism $\beta\in\whset{b}$ for some $b\in X$, labeling an edge starting at $W'$ and an inner classic Whitehead automorphism $\gamma$, and the next step in the peak reduction of $\alpha_k\dotsm\alpha_1$ requires us to conjugate $\beta$ across $\gamma$.
We know that $\beta$ is $\beta_2\beta_1$, where $\beta_1\in(\whset{b})_{W'}$ and $\beta_2\in\whset{b}$ labels an edge starting at $W'$ in $Z^1$.
We conjugate $\beta_1$ and $\beta_2$ across $\gamma$, which labels a loop in $Z^1$.
This amounts to a homotopy of $\phi(\alpha_k\dotsm\alpha_1)$ across $2$--cells of type (C4).
\item
Now we suppose  $\alpha_k\dotsm\alpha_1$ has a subword $\gamma\beta$ with $\beta,\gamma\in\whset{b}$ for some $b\in X$, and we need to replace $\gamma\beta$ by its product $\delta=\gamma\beta\in\whset{b}$ in the factorization as the next step in the peak reduction.
Here $\beta$ is an edge label on an edge originating at a vertex $W'$ in $\widehat Z$.
Then $\phi(\beta)$ is $\beta_2\beta_1$, $\phi(\gamma)$ is $\gamma_2\gamma_1$, and $\phi(\delta)$ is $\delta_2\delta_1$, where all these new automorphisms are in $\whset{b}$, $\beta_1$, $\gamma_1$ and $\delta_1$ stabilize the vertices they originate at, and $\beta_2$, $\gamma_2$ and $\delta_2$ label edges in $Z^1$.
So $\phi(\alpha_k\dotsm\alpha_1)$ will differ before and after the replacement as follows: before we have $\gamma_2\gamma_1\beta_2\beta_1$ and after we have $\delta_2\delta_1$.
Pulling $\gamma_1\beta_2$ across cells of type (C5) will replace it with $\beta_2\gamma_1$, where the second $\gamma_2$ is a product of edge labels at $W'$ instead of $\beta_2\cdot W'$.
Then we can homotope $\gamma_2\beta_2$ to $\delta_3\delta_1$ using a (C5) cell, where $\delta_3$ is some product of edge labels at $W'$.
Then we have homotoped $\gamma_2\gamma_1\beta_2\beta_1$ to a product of $\delta_2$ with a product of edge labels at $W'$.
At this point we only need to use some cells of type (C1) to finish this case.
\item
For this item we suppose $\alpha_k\dotsm\alpha_1$ has a subword $\gamma\beta$ with $\gamma\in\whset{a}$ for some $a\in X$ and $\beta$ a permutation automorphism.
Suppose we need to replace this $\gamma\beta$ with $\beta\delta$, where $\delta=\beta^{-1}\gamma\beta\in\whset{\beta^{-1}(a)}$.
Suppose $\gamma$ in this sequence is the label on an edge originating at a vertex $W'$.
We know $\phi(\gamma)$ is some $\gamma_2\gamma_1$  where $\gamma_1\in(\whset{a})_{W'}$ and $\gamma_2$ is the label on an edge in $Z$ originating at $W'$,
and $\phi(\delta)$ is some $\delta_2\delta_1$ with $\delta_1\in(\whset{\beta^{-1}(a)})_{\beta^{-1}\gamma\cdot W'}$ and $\delta_2$ is the label on an edge in $Z$ originating at $\beta^{-1}\gamma\cdot W'$.
We know $\gamma_1$ is represented by some edge path in the loops at $W'$.
For each edge in the sequence, we can conjugate $\beta$ across the given edge label, thereby pushing our edge path across a $2$--cell of type (C6).
Then we conjugate $\beta$ across $\gamma_2$, getting $\beta\delta_2$ together with some product of loops at $\beta^{-1}\gamma\cdot W'$.
We use disks of type (C1) to rewrite the remaining sequence as $\beta\delta_2\delta_1$.
This homotopes the image of $\alpha_k\dotsm\alpha_1$ under $\phi$ from before the replacement to the image after the replacement.
\item
Now we suppose we have $\beta\gamma$ in $\alpha_k\dotsm\alpha_1$ with $\beta\in\whset{b}$ and $\gamma\in\whset{a}$ satisfying the hypotheses of Lemma~\ref{le:steinbergrel}.
Suppose $\beta$ is an edge label originating from the vertex $W'$ and to peak-reduce $\alpha_k\dotsm\alpha_1$ we need to conjugate $\gamma$ across $\beta$.
Let $\delta=\gamma^{-1}\beta\gamma$, which is in $\whset{b}$ by Lemma~\ref{le:steinbergrel}.
Then $\phi$ sends the $\gamma$ originating at $\gamma^{-1}\cdot W'$ to some $\gamma_2\gamma_1$ and sends the $\gamma$ originating at $\delta\gamma^{-1}\cdot W'$ to some $\gamma_4\gamma_3$, where these automorphisms are in $\whset{a}$, we know $\gamma_2$ and $\gamma_4$ label edges leaving $\gamma^{-1}\cdot W'$ and $\delta\gamma^{-1}\cdot W'$ in $Z$, and $\gamma_1$ and $\gamma_3$ fix $\gamma^{-1}\cdot W'$ and $\delta\gamma^{-1}\cdot W'$ respectively.
Let  $S=(X\setminus[a])^{\pm}\setminus\supp(\gamma)$.
Then $\gamma\in\whsetsr{a}{S}$, and by the construction of $\phi$, we know $\gamma_i$ is as well, for $i=1,2,3,4$.
Similarly, $\phi$ sends $\beta$ to $\beta_2\beta_1$ where $\beta_2$ labels an edge of $Z$ starting at $W'$, $\beta_1$ fixes $W'$, $\beta_i\in\whset{b}$ and $\supp(\beta_i)\subset\supp(\beta)$ for $i=1,2$, and $\phi$ sends $\delta$ to $\delta_1\delta_2$ where $\delta_2$ labels an edge of $Z$ starting at $\gamma^{-1}\cdot W'$, $\delta_1$ fixes $\gamma^{-1}\cdot W'$, $\delta_i\in\whset{b}$ and $\supp(\delta_i)\subset\supp(\delta)$ for $i=1,2$.
These conditions mean that the edges labeled with $\beta_2$ and $\gamma_2^{-1}$ leaving $W'$ in $Z$ satisfy the hypotheses for Lemma~\ref{le:steinbergrel}, and therefore there is a $2$--cell of type (C7) for this relation.
We can move $\beta_1$ out of the way using cells of type (C5).
Then we slide the edge path across our cell of type (C7), and replace
$\gamma_2\beta_2$ with $\gamma_2w_2w_1\delta_2$, where $w_2$ is an edge loop label sequence representing an element of $(\whset{a})_{\delta\gamma^{-1}\cdot W'}$ and $w_1$ is an edge loop label sequence representing an element of $(\whset{b})_{\delta\gamma^{-1}\cdot W'}$.
Again we use $2$--cells of type (C5) to move $w_1$ to precede $\delta_2$.
We note that the edge labels appearing in $w_1$ and in $\gamma_1$ satisfy the hypotheses of Lemma~\ref{le:steinbergrel}, and so by a sequence of slides across $2$--cells of type (C7) we can move one past the other.
Similarly, the labels in $\gamma_1$ together with $\delta_2$ satisfy the hypothesis of Lemma~\ref{le:steinbergrel}, and we can slide $\gamma_1$ across $\delta_2$.
Then $\delta_1$ and $w_1$ both represent the same element of $(\whset{b})_{\delta\gamma^{-1}\cdot W'}$ and we can use $2$--cells of type (C1) homotope one to the other.
Similarly we homotope $w_2$ to $\gamma_3\gamma_1$.
So we have homotoped $\beta_2\beta_1\gamma_2\gamma_1$ to $\gamma_4\gamma_3\delta_2\delta_1$ relative to endpoints, so we have homotoped the image of $\alpha_k\dotsm\alpha_1$ under $\phi$ before this move to the image after this move.
\end{itemize}

So we peak-reduce $\alpha_k\dotsm\alpha_1$ with respect to $V$ and homotope $\phi(\alpha_k\dotsm\alpha_1)$ at each step, so our edge path describes an edge loop in $\widehat Z$ that sends each intermediate image of $U$ to one of the same length as $U$.
Then each $\alpha_i$ is an inner automorphism or a permutation automorphism by Lemma~\ref{le:mastertuple}.
\end{proof}

\subsection*{Acknowledgements}
This research was supported by a grant from the National Science Foundation, award number DMS-1206981.
I would like to thank the anonymous referee for carefully reading and commenting on this paper.
I am grateful to all the people with whom I had conversations about this research, including Montserrat Casals-Ruiz, Matt Clay, Michael Day, Michael Geline, Mark Johnson, Yo'av Rieck and Ilya Kazachkov.
My mother Rose Day has made important contributions to all my work, indirectly, in no small amount.
\bibliography{fullpeakred}
\bibliographystyle{plain}
\end{document}